\documentclass[11pt,reqno,oneside]{amsart}
\usepackage{amsfonts,amssymb,amsmath,epic,eepic,float,fullpage}
\usepackage[normalem]{ulem}
\usepackage{longtable}
\usepackage[usenames,dvipsnames]{color}
\usepackage{siunitx}
\usepackage{empheq}
\usepackage[mathscr]{euscript}
\usepackage[T1]{fontenc}
\usepackage{mathtools}
\usepackage{comment}
\usepackage{bm}
\usepackage{caption}
\usepackage{hyperref}
\usepackage{enumerate,enumitem}
\usepackage{tikz}
\usetikzlibrary{graphs,patterns,decorations.markings,arrows,matrix}
\usetikzlibrary{calc,decorations.pathmorphing,decorations.pathreplacing,shapes}
\allowdisplaybreaks

\newtheorem{thm}{Theorem}[section]
\newtheorem*{thm*}{Theorem}

\newtheorem{prop}[thm]{Proposition}

\newtheorem{ques}[thm]{Question}
\newtheorem{cor}[thm]{Corollary}

\theoremstyle{definition}
\newtheorem{defn}[thm]{Definition}

\theoremstyle{remark}
\newtheorem{rmk}[thm]{Remark}

\theoremstyle{remark}
\newtheorem{ex}[thm]{Example}
\long\def\junk#1{}

\def\leq{\leqslant}
\def\geq{\geqslant}

\def\l{\lambda}
\def\m{\mu}
\def\n{\nu}
\def\GG{\mathfrak{G}}
\def\JJ{\mathfrak{J}}
\def\T{\mathbb{T}}
\def\TT{\mathbb{T}^{*}}
\def\t{\mathrm{T}}
\def\tt{\mathrm{T}^{*}}
\def\buv{(\mathbf{u},\mathbf{v})}
\def\bvu{(\mathbf{v},\mathbf{u})}

\colorlet{lgray}{white!85!black}
\colorlet{lblue}{white!85!blue}
\colorlet{lblack}{white!85!black}
\colorlet{lred}{white!85!red}
\colorlet{lgreen}{white!80!green}
\colorlet{dgreen}{black!30!green}
\definecolor{green}{rgb}{0.1,0.8,0.1}
\definecolor{yellow}{rgb}{1.0,0.85,0.25}
\usepackage{ytableau}

\newcommand{\bra}[1]{\left\langle #1\right|}
\newcommand{\ket}[1]{\left|#1\right\rangle}

\renewcommand{\tikz}[2]{
\begin{tikzpicture}[scale=#1,baseline=(current bounding box.center),>=stealth]
#2
\end{tikzpicture}}

\def\AA{\bm{A}}
\def\BB{\bm{B}}
\def\CC{\bm{C}}
\def\DD{\bm{D}}
\def\l{{\sf L}}
\def\m{{\sf M}}
\def\n{{\sf N}}

\renewcommand\le{\leq}

\newcommand{\ba}{\mathbf{A}}
\newcommand{\bb}{\mathbf{B}}
\newcommand{\bc}{\mathbf{C}}
\newcommand{\bd}{\mathbf{D}}
\newcommand{\ff}{\mathbb{F}}

\renewcommand\ss{\scriptstyle}

\numberwithin{equation}{section}

\title{Inhomogeneous $q$-Whittaker Polynomials I: Duality and Expansions}
\author{Ajeeth Gunna \and Michael Wheeler \and Paul Zinn-Justin}

\address{Ajeeth Gunna, Michael Wheeler, and Paul Zinn-Justin, School of Mathematics and Statistics, University of Melbourne, Australia}
\email{ajeeth.gunna@unimelb.edu.au}
\email{wheelerm@unimelb.edu.au}
\email{pzinn@unimelb.edu.au}

\begin{document}
\begin{abstract}
We introduce a new family of symmetric polynomials $\mathfrak{G}^{\buv}_{\lambda}$ arising from exactly solvable lattice models associated with the quantised loop algebra $\mathcal{U}_{q}(\mathfrak{sl}_{2}[z^\pm])$. The polynomials  $\mathfrak{G}^{\buv}_{\lambda}$ unify $q$-Whittaker polynomials, inhomogeneous $q$-Whittaker polynomials, Grothendieck polynomials and their duals. Using Yang--Baxter equation, we derive Cauchy identities and combinatorial formulas for the transition coefficients.
\end{abstract}

\maketitle
\setcounter{tocdepth}{1}
\makeatletter
\def\l@subsection{\@tocline{2}{0pt}{2.5pc}{5pc}{}}
\makeatother
\tableofcontents
\section{Introduction}
\subsection{Symmetric polynomials}
Symmetric polynomials (that is, polynomials which are invariant under permutation of variables) are ubiquitous in mathematical physics and related areas. The prototypical example is that of \textbf{Schur polynomials}, a basis of symmetric polynomials denoted below $s_\lambda$ where the index $\lambda$ is a partition. They appear (i) in physics as the wave functions of one-dimensional quantum free fermions; (ii) in probability theory under the name of Schur process, the quintessential determinantal process; (iii) in representation theory as characters of the general linear group and; (iv) in geometry as (representatives of) Schubert classes in the Grassmannian. There are various interrelations between these interpretations (in particular, (i) and (iv) are related to \emph{classical}\/ integrable systems such as the Kadomtsev–Petviashvili hierarchy).
One should also point out that quantum free fermionic models are a special case of \textbf{quantum integrable systems}, which is the point of view that underlies the approach of the present paper; this will be described in more detail below, see \S\ref{ssec:lattice}.

There is a large combinatorial literature devoted to
Schur polynomials, including various expressions (e.g., tableau formulae) and identities which they satisfy (e.g., Cauchy identity, Littlewood identities). One often views them as so-called \emph{symmetric functions} (loosely, symmetric polynomials in an indeterminate number of variables), and in this introduction we shall identify Schur polynomials with the corresponding Schur functions, and similarly for other families.
Define then the Hall inner product to be the
scalar product on the ring of symmetric functions $\Lambda$ under which the Schur functions are orthonormal:
\[
\langle s_{\lambda}, s_{\mu} \rangle = \delta_{\lambda, \mu},
\]
and the (ring) involution $\omega$, which acts
on power-sums as $\omega(p_k)=(-1)^{k-1}p_k$, and satisfies
\[
\omega(s_{\lambda}) = s_{\lambda'}.
\]
where $\lambda'$ is the conjugate partition of $\lambda$.

The present work is devoted to the study of certain families of symmetric rational functions. One can view them as broad generalisations of Schur functions, which encompass many known ones; we describe some of the latter in the next few sections.
These generalisations retain some, but usually not all, of the interpretations above; in particular, they underlie a variety of physical or probabilistic models.


\subsection{Grothendieck polynomials}
Grothendieck polynomials $(G_{\lambda})$ were introduced by Lascoux and Schützenberger~\cite{Gcom:LS1983} 
as polynomial representatives of $K$-classes of Schubert varieties. \textbf{Symmetric Grothendieck polynomials} were studied combinatorially by Buch~\cite{Gcom:B2002}. 
They are inhomogeneous and their lowest degree component coincides with the corresponding Schur polynomial. 
Lam and Pylyavskyy~\cite{Gcom:LP2007} defined three additional families: 
dual Grothendieck polynomials $(g_{\lambda})$, weak Grothendieck polynomials $(J_{\lambda})$, and weak dual Grothendieck polynomials $(j_{\lambda})$. 
Together these four families satisfy:
\begin{align*}
\langle G_{\lambda}, g_{\mu} \rangle &= \delta_{\lambda, \mu}, & \qquad
\langle J_{\lambda}, j_{\mu} \rangle &= \delta_{\lambda, \mu}, \\[0.5em]
\omega(G_{\lambda}) &= J_{\lambda}, & \qquad
\omega(g_{\lambda}) &= j_{\lambda}.
\end{align*}

Since their introduction the family Grothendieck found applications to totally asymmetric simple exclusion process (TASEP)~\cite{gcom:Y2021RP}, the last-passage percolation (LPP) corner growth model ~\cite{gcom:Y2020per}, to enumeration of plane partitions~\cite{gcom:Y2021PP}, and in the context of \emph{Brill--Noether varieties}~\cite{Gcom:2021CP} among others. Generalizations of these polynomials have also been extensively studied, often referred to as ``refined'' versions~\cite{gcom:GGL2015Refined,gcom:HJKS2024Refined,Glat:MScr2025}.

\subsection{$q$-Whittaker polynomials and Hall--Littlewood polynomials}
Two more generalisations arise as follows.
\textbf{Macdonald polynomials} $P_{\lambda}(x;q,t)$ form a two-parameter family of symmetric functions which generalise Schur polynomials (at $q=t$). 
They may be characterised as the joint eigenfunctions of the Macdonald difference operators; see~\cite[Ch.~VI]{Macdonald}.
They are orthogonal under some $q,t$-deformation of the Hall inner product;
more precisely,
the dual Macdonald polynomial $Q_{\lambda}$
defined by
$\langle P_\lambda,Q_\mu \rangle_{q,t} = \delta_{\lambda,\mu}$
is related to $P_\lambda$ by
\[
Q_{\lambda}(x;q,t) = b_{\lambda}(q,t)\, P_{\lambda}(x;q,t),
\]
where $b_{\lambda}(q,t)$ is a scalar depending only on the partition $\lambda$ and the parameters $q,t$,
whose explicit form will not be required in what follows. We have emphasised in the formula above the dependence on both the alphabet $x=(x_1,\ldots,x_n)$ and the two parameters $q,t$.

\medskip
The \textbf{Hall--Littlewood polynomials} $P_\lambda(x;t)$ are the $q=0$ specialisation of Macdonald polynomials. 
They interpolate between monomial symmetric polynomials $(t=1)$ and Schur polynomials $(t=0)$. 
They were introduced by Hall in the study of Hall algebras;
they also arise in a variety of geometric contexts (cohomology of Springer fibers, Borel--Weil construction on cotangent bundles of flag varieties, etc) 
and the representation theory of finite general linear groups.

\medskip
The \textbf{$q$-Whittaker polynomials} $W_\lambda(x;q)$ are then defined as the $t=0$ specialisation of Macdonald polynomials. They also appear as polynomial eigenfunctions of the $q$-deformed Toda chain, and as characters of certain subspaces of integrable representations of affine algebras.

A key structural feature of the Macdonald theory is the $(q,t)$-involution on symmetric functions $\omega_{q,t}$, which deforms the involution $\omega$ and is defined by
$
\omega_{q,t}(p_{k}) \,=\, (-1)^{k-1}\frac{1-q^{k}}{1-t^{k}}\,p_{k}
$.
Then the involution acts on Macdonald polynomials in the following way:
\[
\omega_{q,t}\!\left(P_{\lambda}(x;q,t)\right)
= Q_{\lambda'}(x;t,q),
\]


\medskip
Specializing this identity to the case $t=0$, we obtain
\[
\omega_{q}\!\left(W_{\lambda}(x;q)\right)
=b_{\lambda'}(q) P_{\lambda'}(x;q).
\]
where $\omega_q=\omega_{q,0}$, $b_{\lambda'}(q)=b_{\lambda'}(q,0)$
and $P_{\lambda'}(x;q)$ is the Hall--Littlewood polynomial (with $t$ replaced by $q$).

\subsection{Lattice models}\label{ssec:lattice}
As alluded above, our approach to the study of these
families of symmetric polynomials involves quantum integrability.
It has been known for decades that various expressions in quantum integrable systems are symmetric polynomials or rational functions, such as certain Bethe wave functions, partition functions or correlation functions. However, 
the systematic study of symmetric polynomials using a particular incarnation of integrability that is most suitable for this purpose, namely two-dimensional exactly solvable lattice models, took off more recently (as first advocated in \cite{PZJ-hdr}; see also~\cite{slat:BBFSchurYBE}.
One expresses the desired polynomials as partition functions of such lattice models, and then, by leveraging the underlying \emph{Yang--Baxter equation}, one derives various identities for them.

Lattice models for the family of Grothendieck polynomials~\cite{Glat:GZJ2023,Glat:MScr2025,Glat:MS2013,Glat:WZJ2019LR}, as well as for
Hall--Littlewood~\cite{Garbali2017NewGeneralisation,Hlat:K2013Cylindric,Hlat:T2006,Hlat:WZJ2016} and $q$-Whittaker polynomials, have been constructed. We observe that all such lattice models are obtained by using the vertex weights associated with the most general solution of the \textbf{Yang--Baxter equation} arising from the quantised loop algebra $\mathcal{U}_{q}(\mathfrak{sl}_{2}[z^\pm])$. This provides a strong guiding principle for further generalisations, and indeed,
the new families which we introduce in this paper are of the same form, which is how we discovered them.
In this way, lattice models provide a unifying framework for families of symmetric polynomials that arise in diverse contexts. 

\subsection{Main objectives}
In~\cite{spin-BK2024}, a family of polynomials 
\(\mathbb{F}^{(\mathbf{a},\mathbf{b})}_{\lambda}(x_{1},\dots,x_{n})\) 
was introduced, depending on two families of parameters  
\(\mathbf{a}=(a_{1},\dots,a_{n})\) and \(\mathbf{b}=(b_{1},\dots,b_{n})\);  
these polynomials are referred to as the \textbf{inhomogeneous spin $q$-Whittaker polynomials}. 

\medskip
In this work, we consider only the \(\mathbf{a} = (0,\dots,0)\) specialization of $\mathbb{F}_{\lambda}$, and refer to them as \textbf{inhomogeneous $q$-Whittaker polynomials}. These polynomials are symmetric and inhomogeneous. Their lowest degree coincides with $q$-Whittaker polynomial. It was observed in~\cite{spinWhitt-Muc} that, at $q=0$, the inhomogeneous $q$-Whittaker polynomials $\mathbb{F}_{\lambda}$ reduce to Grothendieck polynomials.

\medskip
In this paper, our main objective is to answer the following question:
\begin{ques}
Are there symmetric polynomials $\mathbb{G}_{\lambda}$ that are a generalisation of $q$-Whittaker polynomials and are dual to $\mathbb{F}_{\lambda}$ under the $q$-Hall inner product?    
\end{ques}

We answer this question by introducing the set of polynomials with two sets of parameters $\mathfrak{G}^{\buv}_{\lambda}$, which contain $q$-Whittaker polynomials, inhomogeneous $q$-Whittaker polynomials, Grothendieck polynomials, and their duals. In particular,
at $\mathbf{v}=\mathbf{0}$, they reduce to $\mathbb{F}_{\lambda}$. We define \emph{dual Inhomogeneous $q$-Whittaker polynomials}$(\mathbb{G}_{\lambda})$ as $\mathbf{u}=\mathbf{0}$ specialization of $\mathfrak{G}^{\buv}_{\lambda}$. The duality of $\mathbb{F}_{\lambda}$ and $\mathbb{G}_{\lambda}$ then follows from the following Cauchy identity for $\mathfrak{G}^{\buv}_{\lambda}$:
\begin{thm*}[Theorem~\ref{thm:spinL_Cauchy_qwhittaker}]\label{thm:cauchy_intro}
Fix two positive numbers $n$ and $m$, and let $\lambda$ and $\mu$ be two partitions. Then the family of $q$-Whittaker polynomials satisfy the following summation identity (assuming that all the variables are in the unit disc):
\begin{multline}
\,\sum_{\kappa}\,\dfrac{c_{\kappa'}}{c_{\lambda'}}\,\GG^{(\mathbf{v},\mathbf{u})}_{\kappa/\lambda}(x_{1},\dots,x_{n})\, \GG^{\buv}_{\kappa/\mu}(y_{1},\dots,y_{m})\,   \\
\,=\,\prod^{n}_{j=1}\prod^{m}_{i=1}\,\dfrac{1}{(x_{i}y_{j};q)_{\infty}}\,\sum_{\kappa}\,\dfrac{c_{\mu'}}{c_{\kappa'}}\,\GG^{(\mathbf{v},\mathbf{u})}_{\mu/\kappa}(x_{1},\dots,x_{n})\,\GG^{\buv}_{\lambda/\kappa}(y_{1},\dots,y_{m}),
\end{multline}
with the left hand sum taken over all partitions that contain both $\mu$ and $\lambda$, and on the right hand side sum is taken over all partition that are contained in both $\lambda$ and $\mu$.
\end{thm*}
We also study a family of symmetric rational functions $\mathfrak{J}^{\buv}_{\lambda/\mu}$, related to $\mathfrak{G}^{\buv}_{\lambda/\mu}$ through $\omega_q$, which simultaneously generalise Hall–Littlewood polynomials, weak Grothendieck polynomials, and their duals. There are similar Cauchy identities involving $\mathfrak{J}^{\buv}_{\lambda/\mu}$ (Theorems~\ref{thm:cauchy_spin1} and \ref{thm:cauchy_dual}).
We note that the rational functions $\mathfrak{J}^{\buv}_{\lambda/\mu}$ have previously appeared in~\cite{Garbali2017NewGeneralisation}, although the corresponding Cauchy identities were not recorded there. Various Cauchy identities have been established for generalisations of Hall--Littlewood and $q$-Whittaker polynomials~\cite{spin-Bor2017,spin-BK2024,spin-BW2021,spin-chen2021,Garbali2017NewGeneralisation,spin-K2024,spin-MP}. 
Our Cauchy identities involving $\mathfrak{G}^{\buv}_{\lambda/\mu}$ are unrelated to these earlier results, except in the specialization where they reduce to the standard Cauchy identity for $q$-Whittaker polynomials.

The polynomials introduced in this paper, $\GG_{\lambda/\mu}^{\buv}$, and those in~\cite{spin-BK2024}, $\ff_{\lambda/\mu}^{(\mathbf{a},\mathbf{b})}$, are different generalisations of $q$-Whittaker polynomials. A quick way to see this is to compare the Boltzmann weights used in their definitions: the weights for $\ff_{\lambda/\mu}^{(\mathbf{a},\mathbf{b})}$ are completely factorised, whereas the weights~\eqref{weights:spinL_uv} used for $\GG_{\lambda/\mu}^{\buv}$ are not. Moreover, the two families satisfy different Cauchy identities; compare Theorem~5.8 of~\cite{spin-K2024} with Theorem~\ref{thm:spinL_Cauchy_qwhittaker}.
Nevertheless, they both reduce to the ordinary $q$-Whittaker polynomials as $\ff^{(\mathbf{0},\mathbf{0})}_{\lambda}$ and $\GG^{(\mathbf{0},\mathbf{0})}_{\lambda}$, and they both also specialise to the inhomogeneous $q$-Whittaker polynomials as $\ff^{(\mathbf{0},\mathbf{b})}_{\lambda}$ and $\GG^{(\mathbf{b},\mathbf{0})}_{\lambda}$.

\medskip
Another goal of the paper is to answer:
\begin{ques}\label{ques}
Find a combinatorial formula for the coefficients that appear in the following expansions:
\begin{align}
    W_{\lambda}(x_{1},\dots,x_{n})\,&=\,\sum_{\mu}\,a_{\lambda,\mu}(q)\,\mathbb{F}_{\mu}(x_{1},\dots,x_{n})\\
    \mathbb{F}_{\lambda}(x_{1},\dots,x_{n})\,&=\,\sum_{\mu}\,b_{\lambda,\mu}(q)\,{W}_{\mu}(x_{1},\dots,x_{n})\\
    \mathbb{G}_{\lambda}(x_{1},\dots,x_{n})\,&=\,\sum_{\mu}\,c_{\lambda,\mu}(q)\,{W}_{\mu}(x_{1},\dots,x_{n})
\end{align}
\end{ques}

We answer these problems by obtaining a combinatorial formula for a general expansion in Theorem~\ref{thm:general_expansion}:
\begin{equation}
\GG^{\buv}_{\lambda}(x_{1},\dots,x_{n})\,=\,\sum_{\mu}\,d_{\lambda,\mu}(q)\,\GG^{(\mathbf{y},\mathbf{0})}_{\mu}(x_{1},\dots,x_{n}),
\end{equation}  

Then, by specializing, we obtain explicit combinatorial formulas for the coefficients in Question~\ref{ques}. In particular, we find that $a_{\lambda,\mu}(q) \in \mathbb{N}[q]$, which implies that the $q$-Whittaker polynomials ($W_{\lambda}$) expand positively in terms of the inhomogeneous $q$-Whittaker polynomials ($\mathbb{F}_{\lambda}$). Furthermore, we obtain that $c_{\lambda,\mu}(q) \in \mathbb{N}[q]$, showing that the functions dual inhomogeneous $q$-Whittaker polynomials ($\mathbb{G}_{\lambda}$) are Schur-positive.

\subsection{Inversion relations and inverse expansions}
One essential feature of exactly solvable models, as first used to great effect by Baxter himself, is that global properties may be deduced from relations among local objects.  
The most standard example is the proof of the commutation of transfer matrices
by repeated application of the Yang--Baxter equation, which is a local relation satisfied by the building blocks ($L$-matrices) of the transfer matrix.
In the context of symmetric polynomials, this commutation relation implies invariance of our polynomials under switching two neighbouring variables, which in turn implies the full symmetry in the variables.
Symmetry -- being invariant under permutations of the variables -- is a global property, yet its proof rests entirely on a local relation.  

We observe a similar phenomenon regarding transition coefficients.
In the final section, we obtain a combinatorial formula for the coefficients appearing in the expansion
\[
W_{\lambda}(x_{1},\dots,x_{n})
   = \sum_{\mu} a_{\lambda,\mu}(q)\,
      \mathbb{F}_{\mu}(x_{1},\dots,x_{n}),
\]
where each coefficient $a_{\lambda,\mu}(q)$ can be expressed as the partition function of a lattice model with vertex weights
\[
\tikz{0.5}{
\draw[lgray,line width=4pt,->] (-1,0) -- (1,0);
\draw[lgray,line width=4pt,->] (0,-1) -- (0,1);
\node[left] at (-1,0) {\tiny $b$};
\node[right] at (1,0) {\tiny $d$};
\node[below] at (0,-1) {\tiny $a$};
\node[above] at (0,1) {\tiny $c$};
}
\;=\;
\mathbf{1}_{a+b=c+d}\,\mathbf{1}_{\,b\le c}\,
\binom{a}{\,c-b\,}_{q}.
\]

Similarly, we obtain an analogous combinatorial formula for the coefficients in the inverse expansion
\[
\mathbb{F}_{\lambda}(x_{1},\dots,x_{n})
   = \sum_{\mu} b_{\lambda,\mu}(q)\,
     W_{\mu}(x_{1},\dots,x_{n}),
\]
where
\[
\tikz{0.5}{
\draw[lgray,line width=4pt,->] (-1,0) -- (1,0);
\draw[lgray,line width=4pt,->] (0,-1) -- (0,1);
\node[left] at (-1,0) {\tiny $b$};
\node[right] at (1,0) {\tiny $d$};
\node[below] at (0,-1) {\tiny $a$};
\node[above] at (0,1) {\tiny $c$};
}
\;=\;
\mathbf{1}_{a+b=c+d}\,\mathbf{1}_{\,b\le c}\,
(-1)^{\,c-b}\,q^{\binom{c-b}{2}}\,
\binom{a}{\,c-b\,}_{q}.
\]

We observe that the weights defining $a_{\lambda,\mu}(q)$ and $b_{\lambda,\mu}(q)$ satisfy an inversion relation.  
In other words, once a formula for the coefficients $a_{\lambda,\mu}(q)$ in the expansion 
\(
W_{\lambda}=\sum_{\mu} a_{\lambda,\mu}(q)\,\ff_{\mu}
\)
is known as a partition function, one only needs to invert these local weights to obtain a combinatorial formula for the inverse expansion
\(
\ff_{\lambda}=\sum_{\mu} b_{\lambda,\mu}(q)\,W_{\mu}.
\)

\medskip
In our setting, however, we did not need to carry out this inversion procedure, since we obtained a more general expansion that simultaneously contains both of the desired formulas as special cases.

\subsection{Layout of the paper}
The paper is organised as follows. In Section~\ref{sec:YBE}, we present the most general solution of the Yang--Baxter equation arising from the quantised loop algebra $\mathcal{U}_{q}(\mathfrak{sl}_{2}[z^\pm])$.

In Section~\ref{sec:HL-family}, we define lattice models with a fundamental horizontal line and use them to introduce the functions
 $\mathfrak{J}^{\buv}_{\lambda}$. In Section~\ref{subsec:HL-degenerations}, we derive a branching formula for $\mathfrak{J}^{\buv}_{\lambda}$ and explain how these polynomials specialise to Hall--Littlewood polynomials, weak Grothendieck polynomials and their duals.

In Section~\ref{sec:q-whitt-family}, we define the polynomials $\mathfrak{G}^{\buv}_{\lambda}$ using lattice models with a fused horizontal line. We then prove a Cauchy identity for $\mathfrak{G}^{\buv}_{\lambda}$, as well as a dual Cauchy identity relating $\mathfrak{G}^{\buv}_{\lambda}$ and $\mathfrak{J}^{\buv}_{\lambda}$. In Section~\ref{subsec:qWhit-degenerations}, we discuss various degenerations of $\mathfrak{G}^{\buv}_{\lambda}$ and their connections to Grothendieck polynomials and their duals. In the final Section~\ref{sec:expansions}, we derive a combinatorial formula for the expansion of $\mathfrak{G}^{\buv}_{\lambda}$ in terms of the functions $\mathbb{F}_{\lambda}$, and in Section~\ref{subsec:Expansions_Degenerations} we study several degenerations of this expansion.

\subsection{Acknowledgements}
A.G. is grateful to Alexandr Garbali for generously sharing his code and for discussions, and Damir Yeliussizov for stimulating conversations. The authors also thank Matteo Mucciconi for valuable comments on an earlier draft.

\section{The Yang--Baxter equation}\label{sec:YBE}
\subsection{Lattice model}
In this section, we establish the notation and conventions used throughout the paper for constructing our lattice models. We realise the polynomials \( \GG^{\buv}_{\lambda} \) and \( \JJ^{\bvu}_{\lambda} \) as partition functions of lattice models. These models are built from vertices of the following type:
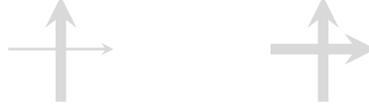
\[
\tikz{0.7}{
\draw[lgray,line width=1pt,->] (-1,0) -- (1,0);
\draw[lgray,line width=4pt,->] (0,-1) -- (0,1);
}
\hspace{2cm}
\tikz{0.7}{
\draw[lgray,line width=4pt,->] (-1,0) -- (1,0);
\draw[lgray,line width=4pt,->] (0,-1) -- (0,1);
}
\]

Each edge of a vertex is labelled by a non-negative integer. Thick lines indicate that the edge may be labelled by any non-negative integer, whereas thin lines indicate edges restricted to the labels \(0\) or \(1\). In the first vertex, the horizontal edges are restricted to the values \(0\) or \(1\), while the vertical edges may take any non-negative integer. In the second vertex, all four edges are allowed to take any arbitrary non-negative integers.

At the heart of the lattice model lies the principle of \emph{conservation}, which asserts that the sum of the labels on the bottom and left edges of a vertex equals the sum of the labels on the top and right edges. This condition admits a natural combinatorial interpretation: the label on each edge represents the number of particles it carries, with particles entering a vertex from the bottom and left, and exiting from the top and right. In this interpretation, conservation ensures that every particle that enters a vertex must also exit.

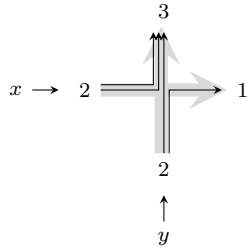
\[
\tikz{0.7}{
\draw[lgray,line width=5pt,->] (-1.2,0) -- (1.2,0);
\draw[lgray,line width=5pt,->] (-0.05,-1.2) -- (-0.05,1.2);
 \node at (0,-1.5) {$\ss 2$};
 \node at (-1.5,0) {$\ss 2$};
 \node at (0,1.5) {$\ss 3$};
 \node at (1.5,0) {$\ss 1$};
 \draw[->](0,-1.2)--(0,1.1);
 \draw[->](0.1,-1.2)--(0.1,0)--(1.1,0);
 \draw[->](-1.2,0)--(-0.1,0)--(-0.1,1.1);
 \draw[->](-1.2,0.1)--(-0.2,0.1)--(-0.2,1.1);
 \draw[->](0,-2.5) node[below] {$\ss y$}--(0,-2);
 \draw[->](-2.5,0) node[left] {$\ss x$}--(-2,0);
}
\]

For each vertex, we attach a weight that depends on two variables, called the spectral parameters \( x \) and \( y \). The variable \( x \) can be thought of as the rapidity of particles entering the vertex from the left, and \( y \) as the rapidity of those entering from the bottom.

We now construct a lattice by arranging these vertices in a rectangular array situated at the corner of the plane. The horizontal lines in the entire lattice are required to be consistently either all thick or all thin. In this setup, each row and column of the lattice is assigned a fixed spectral parameter.

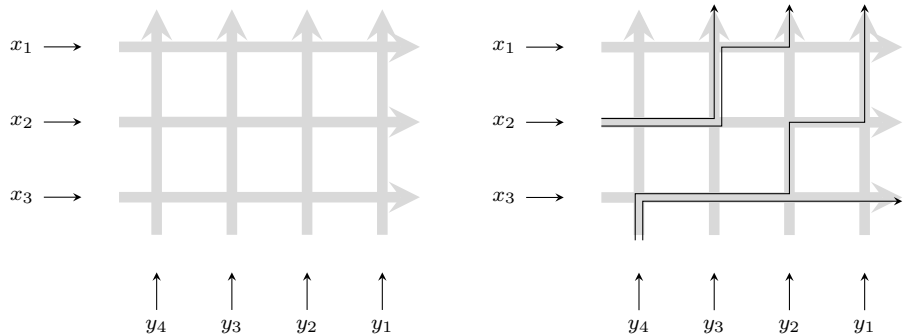
\[
\tikz{1}{
\foreach \x in {1,2,3}{
\draw[lblack,line width= 4 pt,->] (0,\x-0.5)--(4,\x-0.5);
};
\foreach \x in {1,2,3,4}{
\draw[lblack,line width= 4 pt,->] (\x-0.5,0)--(\x-0.5,3);
};
\foreach \x in {1,2,3,4}{
 \draw[->](4.5-\x,-1) node[below] {$\ss y_{\x}$}--(4.5-\x,-0.5);
};
\foreach \x in {1,2,3}{
 \draw[->](-1,3.5-\x) node[left] {$\ss x_{\x}$}--(-0.5,3.5-\x);
};
}
\qquad
\tikz{1}{
\foreach \x in {1,2,3}{
\draw[lblack,line width= 4 pt,->] (0,\x-0.5)--(4,\x-0.5);
};
\foreach \x in {1,2,3,4}{
\draw[lblack,line width= 4 pt,->] (\x-0.5,0)--(\x-0.5,3);
};
\foreach \x in {1,2,3,4}{
 \draw[->](4.5-\x,-1) node[below] {$\ss y_{\x}$}--(4.5-\x,-0.5);
};
\foreach \x in {1,2,3}{
 \draw[->](-1,3.5-\x) node[left] {$\ss x_{\x}$}--(-0.5,3.5-\x);
};
\draw[->](0.45,-0.07)--(0.45,0.55)--(2.5,0.55)--(2.5,1.5)--(3.5,1.5)--(3.5,3.07);
\draw[->](0.55,-0.07)--(0.55,0.45)--(4,0.45);
\draw[->](0,1.55)--(1.5,1.55)--(1.5,3.07);
\draw[->](0,1.45)--(1.6,1.45)--(1.6,2.5)--(2.5,2.5)--(2.5,3.07);
}
\]

We refer to a labelling of all the edges in the lattice as a \emph{configuration}. Since the conservation law holds at every vertex, the total number of particles is conserved across the entire lattice. To each configuration, we assign a weight given by the product of the weights of all vertices in the configuration.

Given a fixed boundary, the sum of the weights of all possible configurations is called the \emph{partition function}. Our objective is to obtain our desired symmetric polynomials as partition functions of lattice models. 

\subsection{Yang--Baxter equation}
All the proofs of the theorems in this paper are straightforward applications of the celebrated Yang--Baxter equation (YBE). We now recall the Bošnjak--Mangazeev solution to the YBE in the case of the quantised loop algebra \( \mathcal{U}_q(\mathfrak{sl}_2[z^\pm]) \). Our main reference for this section is \cite[Appendix~C]{BorodinW-spectral2022}, from which we adopt our notation.

We consider the vertices of the following type:

\begin{equation}\label{eq:generic_Lvertex}
\tikz{0.7}{
\draw[lgray,line width=4pt,->] (-1,0) -- (1,0);
\draw[lgray,line width=4pt,->] (0,-1) -- (0,1);
\node[left] at (-1,0) {\tiny $b$};\node[right] at (1,0) {\tiny $d$};
\node[below] at (0,-1) {\tiny $a$};\node[above] at (0,1) {\tiny $c$};
\draw[->](-2.5,0) node[left] {$ (x,\l)$}--(-2,0);
\draw[->](0,-2.5) node[below] {$ (y,\m)$}--(0,-2);
}\,
=
\,
W_{\l,\m}\left(\dfrac{x}{y};q;\tikz{0.5}{
\draw[lgray,line width=1pt,->] (-1,0) -- (1,0);
\draw[lgray,line width=1pt,->] (0,-1) -- (0,1);
\node[left] at (-1,0) {\tiny $b$};\node[right] at (1,0) {\tiny $d$};
\node[below] at (0,-1) {\tiny $a$};\node[above] at (0,1) {\tiny $c$};
}\right)\,\equiv\,
W_{\l,\m}\left(\dfrac{x}{y};q;a,b,c,d\right) 
\end{equation}

Here, \( a, b, c, d \) are non-negative integers, and \( \l, \m \) are positive integers denoting the \emph{spin} of the corresponding lines; that is, the maximum number of particles a line can carry. The weights are constrained by the conditions \( a, c \leq \m \) and \( b, d \leq \l \); if these conditions are not satisfied, the corresponding weight is defined to be identically zero. Then the Yang--Baxter equation takes the form:

\begin{multline}\label{eq:YBE}
\sum_{c_{1},c_{2},c_{3}}
\,W_{\l,\m}\left(\dfrac{x}{y};q;a_{1},a_{2},c_{1},c_{2}\right)
\,
W_{\l,\n}\left(\dfrac{x}{z};q;a_{3},c_{2},c_{3},b_{2}\right)
\,
W_{\m,\n}\left(\dfrac{y}{z};q;c_{3},c_{1},b_{3},b_{1}\right)\,=\\
\,
\sum_{c_{1},c_{2},c_{3}}\, W_{\m,\n}\left(\dfrac{y}{z};q;a_{3},a_{1},c_{3},c_{1}\right)\,W_{\l,\n}\left(\dfrac{x}{z};q;c_{3},a_{2},b_{3},c_{2}\right)W_{\l,\m}\left(\dfrac{x}{y};q;c_{1},c_{2},b_{1},b_{2}\right),
\end{multline}

where $a_{1},a_{2},a_{3},b_{1},b_{2},b_{3}$ are fixed non-negative integers, and the summation on both sides of the equation is over triples of non-negative integers $c_{1},c_{2},c_{3}$. Equivalently, we can represent the equation~\ref{eq:YBE} graphically:

\begin{align}\label{eq:YBE_pictorial}
\sum_{c_1,c_2,c_{3}}
\tikz{0.9}{
\draw[lgray,line width=4pt,->]
(-2,1) node[above,scale=0.6] {\color{black} $a_2$} -- (-1,0) node[below,scale=0.6] {\color{black} $c_2$} -- (1,0) node[right,scale=0.6] {\color{black} $b_2$};
\draw[lgray,line width=4pt,->] 
(-2,0) node[below,scale=0.6] {\color{black} $a_1$} -- (-1,1) node[above,scale=0.6] {\color{black} $c_1$} -- (1,1) node[right,scale=0.6] {\color{black} $b_1$};
\draw[lgray,line width=4pt,->] 
(0,-1) node[below,scale=0.6] {\color{black} $a_{3}$} -- (0,0.5) node[scale=0.6] {\color{black} $c_{3}$} -- (0,2) node[above,scale=0.6] {\color{black} $b_{3}$};
\node[left] at (-2.2,1) {$(x,\l) \rightarrow$};
\node[left] at (-2.2,0) {$(y,\m) \rightarrow$};
\draw[->](0,-2.5) node[below] {$(z,\n)$}--(0,-2);
}
\quad
=
\quad
\sum_{c_1,c_2,c_{3}}
\tikz{0.9}{
\draw[lgray,line width=4pt,->] 
(-1,1) node[left,scale=0.6] {\color{black} $a_2$} -- (1,1) node[above,scale=0.6] {\color{black} $c_2$} -- (2,0) node[below,scale=0.6] {\color{black} $b_2$};
\draw[lgray,line width=4pt,->] 
(-1,0) node[left,scale=0.6] {\color{black} $a_1$} -- (1,0) node[below,scale=0.6] {\color{black} $c_1$} -- (2,1) node[above,scale=0.6] {\color{black} $b_1$};
\draw[lgray,line width=4pt,->] 
(0,-1) node[below,scale=0.6] {\color{black} $a_{3}$} -- (0,0.5) node[scale=0.6] {\color{black} $c_{3}$} -- (0,2) node[above,scale=0.6] {\color{black} $b_{3}$};
\node[left] at (-1.5,1) {$(x,\l) \rightarrow$};
\node[left] at (-1.5,0) {$(y ,\m)\rightarrow$};
\draw[->](0,-2.5) node[below] {$(z,\n)$}--(0,-2);
}
\end{align}

We recall the standard definitions of the \emph{$q$-Pochhammer symbol} and the \emph{$q$-binomial coefficient}:  

The $q$-Pochhammer symbol is defined as  
\[
(a;q)_{k} := (1 - a)(1 - aq)\cdots(1 - aq^{k-1}), \quad \text{with} \quad (a;q)_{0} = 1. 
\]  

The $q$-binomial coefficient is given by  
\[
\binom{a}{b}_{q} = \dfrac{(q;q)_{a}}{(q;q)_{a-b}(q;q)_{b}}.  
\]

\

\begin{thm}[\cite{BorodinW-spectral2022},~\cite{BosnjakM}]
\label{theorem:weightsoftheYBE}
For any two non-negative numbers $\lambda \leq \mu$, we define the function:
\begin{equation*}
\Phi(\lambda,\mu;x,y):=\dfrac{(x;q)_{\lambda} (y/x ; q)_{\mu-\lambda}}{(y;q)_{\mu}}(y/x)^{\lambda}\binom{\mu}{\lambda}_{q}
\end{equation*}
Then the weights defined as 

\begin{multline}
\label{generalweights}
W_{\l,\m}(z;q;a,b,c,d)=\mathbf{1}_{a+b=c+d} 
\hspace{5mm}z^{d-b} q^{a \l-d \m}\\[0.5 em]
\times \sum_{p}\Phi(c-p,c+d-p;q^{\l-\m}z,q^{-\m}z)\Phi(p,b;q^{-\l}/z,q^{-\l})
\end{multline} where the summation is over non-negative numbers such that $0\leq p \leq \min(b,c)$,
satisfy the Yang--Baxter equation ~\eqref{eq:YBE}.
\end{thm}

\begin{ex} Below, we list the vertex weights in the case where \( \l = 1 \) and \( \m \) is an arbitrary positive integer. In this setting, the horizontal edges can take labels only from \( \{0, 1\} \). Then there are four types of vertices which we tabulate below with the corresponding weights written beneath each vertex.

\begin{equation*}
\tikz{0.9}{
\draw[lgray,line width=1pt,->] (-1,0) -- (1,0);
\draw[lgray,line width=4pt,->] (0,-1) -- (0,1);
\node[left] at (-1,0) {\tiny $\BB$};\node[right] at (1,0) {\tiny $\DD$};
\node[below] at (0,-1) {\tiny $\AA$};\node[above] at (0,1) {\tiny $\CC$};
\node[left] at (-1.5,0) {$x \rightarrow$};
\node[below] at (0,-1.4) {$\uparrow$};
\node[below] at (0,-1.9) {$y$};
}=W_{1,\m}\left(\dfrac{x}{y};q;\ba,\bb,\bc,\bd\right)
\end{equation*}

\begin{align}
\label{weights:spin1_row_weights}
\begin{tabular}{|c|c|}
\hline
\quad
\tikz{0.7}{
\draw[lgray,line width=1pt,->] (-1,0) -- (1,0);
\draw[lgray,line width=4pt,->] (0,-1) -- (0,1);
\node[left] at (-1,0) {\tiny $0$};\node[right] at (1,0) {\tiny $0$};
\node[below] at (0,-1) {\tiny $m$};\node[above] at (0,1) {\tiny $m$};
}
\quad
&
\quad
\tikz{0.7}{
\draw[lgray,line width=1pt,->] (-1,0) -- (1,0);
\draw[lgray,line width=4pt,->] (0,-1) -- (0,1);
\node[left] at (-1,0) {\tiny $0$};\node[right] at (1,0) {\tiny $1$};
\node[below] at (0,-1) {\tiny $m$};\node[above] at (0,1) {\tiny $m-1$};
}
\quad
\\[1.3cm]
\quad
$\dfrac{1-x y^{-1} q^{m-\m}}{1-x y^{-1} q^{-\m}}$
\quad
& 
\qquad
$ \dfrac{(q^{m}-1)x y^{-1}q^{-\m}}{1-x y^{-1}q^{-\m}}$
\qquad
\\[0.7cm]
\hline
\quad
\tikz{0.7}{
\draw[lgray,line width=1pt,->] (-1,0) -- (1,0);
\draw[lgray,line width=4pt,->] (0,-1) -- (0,1);
\node[left] at (-1,0) {\tiny $1$};\node[right] at (1,0) {\tiny $0$};
\node[below] at (0,-1) {\tiny $m$};\node[above] at (0,1) {\tiny $m+1$};
}
\quad
&
\quad
\tikz{0.7}{
\draw[lgray,line width=1pt,->] (-1,0) -- (1,0);
\draw[lgray,line width=4pt,->] (0,-1) -- (0,1);
\node[left] at (-1,0) {\tiny $1$};\node[right] at (1,0) {\tiny $1$};
\node[below] at (0,-1) {\tiny $m$};\node[above] at (0,1) {\tiny $m$};
}
\quad
\\[1.3cm] 
\quad
$\dfrac{(1-q^{m-\m})}{1-x y^{-1}q^{-\m}}$
\quad
&
\quad
$\dfrac{(q^{m-\m}-x y^{-1} q^{-\m})}{1-x y^{-1} q^{-\m}}$
\quad
\\[0.7cm]
\hline
\end{tabular} 
\end{align}    
\end{ex}

\begin{prop}(\textbf{Stochasticity})\label{prop:sum_to_unity} For any two fixed non-negative integers, the following equation holds:
\begin{equation}\label{eq:sum_to_unity}
\sum_{c,d}W_{\l,\m}(x;q;a,b,c,d)\,=\,1,    
\end{equation}    
where the sum is over all non-negative numbers $c$ and $d$.
\end{prop}
Graphically, this identity can be represented as:
\[\sum_{c,d}
\tikz{0.7}{
\draw[lgray,line width=4pt,->] (-1,0) node[left,black]{$\ss b$} -- (1,0) node[right,black]{$\ss d$};
\draw[lgray,line width=4pt,->] (0,-1) node[below,black]{$\ss a$} -- (0,1) node[above,black]{$\ss c$};
\node[left] at (-1.5,0) {$\ss (x,\l) \rightarrow$};
\node[below] at (0,-1.6) {$\uparrow$};
\node[below] at (0,-2.3) {$\ss (y,\m)$};}
\quad = \quad
\tikz{0.7}{
\draw[lgray,line width=4pt,->] (-1,0) node[left,black]{$\ss b$}-- (1,0) node[right,black]{$\ss b$};
\draw[lgray,line width=4pt,->] (-1,1)node[left,black]{$\ss a$} -- (1,1) node[right,black]{$\ss a$};
\node[left] at (-1.5,0) {$\ss (y,\m)\, \rightarrow$};
\node[left] at (-1.5,1) {$\ss (x,\l)\, \rightarrow$};}
\]

\section{Family of Hall--Littlewood polynomials}\label{sec:HL-family}

In this section, we develop a lattice model formulation for the Hall--Littlewood polynomials \( \mathfrak{J}^{\buv}_{\lambda} \). We start from the generic weights given in~\ref{weights:spin1_uv_generic}, we make certain substitutions to obtain the relevant vertex weights. These weights are then used to construct transfer matrices, establish their commutation relations, and derive the associated \emph{Cauchy identity}. Finally, we discuss how these functions reduce to Hall--Littlewood polynomials $(Q_{\lambda})$, \emph{weak Grothendieck functions} $(J_{\lambda})$, and \emph{weak dual Grothendieck polynomials} $(j_{\lambda})$.

\subsection{Lattice model}
We list the Boltzmann weights used to define the functions \( \mathfrak{J}^{\buv}_{\lambda} \). These weights are obtained as a specialization of the weights~\ref{weights:spin1_row_weights} by substituting \( x \mapsto -x \), \( q^{-\m} = u v \), \( y = v \) into~\ref{weights:spin1_row_weights} and then divide the weights by $u$ whenever the right edge is occupied (i.e., labelled by $1$). The weights resulting from these substitutions are presented below.

\begin{equation}\label{weights:spin1_uv_generic}
\mathrm{W}_{x;(u,v)}(a,b,c,d) \equiv \mathrm{W}_{x;(u,v)}\left(\tikz{0.9}{
\draw[lgray,line width=1pt,->] (-1,0) -- (1,0);
\draw[lgray,line width=4pt,->] (0,-1) -- (0,1);
\node[left] at (-1,0) {\tiny $b$};\node[right] at (1,0) {\tiny $d$};
\node[below] at (0,-1) {\tiny $a$};\node[above] at (0,1) {\tiny $c$};
\node[left] at (-1.5,0) {$x \rightarrow$};
\node[below] at (0,-1.4) {$\uparrow$};
\node[below] at (0,-1.9) {$(u,v)$};
}\right)
\end{equation}

\begin{align}
\label{weights:spin1_uv_weights}
\begin{tabular}{|c|c|}
\hline
\quad
\tikz{0.7}{
\draw[lgray,line width=1pt,->] (-1,0) -- (1,0);
\draw[lgray,line width=4pt,->] (0,-1) -- (0,1);
\node[left] at (-1,0) {\tiny $0$};\node[right] at (1,0) {\tiny $0$};
\node[below] at (0,-1) {\tiny $m$};\node[above] at (0,1) {\tiny $m$};
}
\quad
&
\quad
\tikz{0.7}{
\draw[lgray,line width=1pt,->] (-1,0) -- (1,0);
\draw[lgray,line width=4pt,->] (0,-1) -- (0,1);
\node[left] at (-1,0) {\tiny $0$};\node[right] at (1,0) {\tiny $1$};
\node[below] at (0,-1) {\tiny $m$};\node[above] at (0,1) {\tiny $m-1$};
}
\quad
\\[1.3cm]
\quad
$\dfrac{1+ux  q^{m}}{1+ux}$
\quad
& 
\qquad
$ \dfrac{(1-q^{m})x}{1+ux}$
\qquad
\\[0.7cm]
\hline
\quad
\tikz{0.7}{
\draw[lgray,line width=1pt,->] (-1,0) -- (1,0);
\draw[lgray,line width=4pt,->] (0,-1) -- (0,1);
\node[left] at (-1,0) {\tiny $1$};\node[right] at (1,0) {\tiny $0$};
\node[below] at (0,-1) {\tiny $m$};\node[above] at (0,1) {\tiny $m+1$};
}
\quad
&
\quad
\tikz{0.7}{
\draw[lgray,line width=1pt,->] (-1,0) -- (1,0);
\draw[lgray,line width=4pt,->] (0,-1) -- (0,1);
\node[left] at (-1,0) {\tiny $1$};\node[right] at (1,0) {\tiny $1$};
\node[below] at (0,-1) {\tiny $m$};\node[above] at (0,1) {\tiny $m$};
}
\quad
\\[1.3cm] 
\quad
$\dfrac{(1-uvq^{m})}{1+ux}$
\quad
&
\quad
$\dfrac{(x+vq^{m})}{1+ux}$
\quad
\\[0.7cm]
\hline
\end{tabular} 
\end{align}   
\begin{rmk}
For $u = v = -s$, these weights reduce to those used in the definition of 
the \emph{spin Hall--Littlewood functions} given in~\cite[Definition~4.1]{spin-Bor2017}.
\end{rmk}

It is convenient to list the weights that satisfy the YBE together with the  weights~\ref{weights:spin1_uv_weights}.
\begin{equation}\label{weights:Rmatrix_spin1_generic}
\mathrm{R}_{x/y}(a,b,c,d) \equiv \mathrm{R}_{x/y}\left(
{\tikz{0.7}{
\draw[lblack,line width=1pt,->] (-1,0) -- (1,0);
\draw[lblack,line width=1pt,->] (0,-1) -- (0,1);
\node[left] at (-1,0) {\tiny $b$};\node[right] at (1,0) {\tiny $d$};
\node[below] at (0,-1) {\tiny $a$};\node[above] at (0,1) {\tiny $c$};
\node[left] at (-1.5,0) {$x \rightarrow$};
\node[below] at (0,-1.6) {$\uparrow$};
\node[below] at (0,-2.3) {$y$};
}}
\right)
\end{equation}

\begin{align}\label{weights:Rmatrix_spin1}
\begin{tabular}{|c|c|c|}
\hline
\quad
\tikz{0.6}{
	\draw[lgray,line width=1.5pt,->] (-1,0) -- (1,0);
	\draw[lgray,line width=1.5pt,->] (0,-1) -- (0,1);
	\node[left] at (-1,0) {\tiny $0$};\node[right] at (1,0) {\tiny $0$};
	\node[below] at (0,-1) {\tiny $0$};\node[above] at (0,1) {\tiny $0$};
}
\quad
&
\quad
\tikz{0.6}{
	\draw[lgray,line width=1.5pt,->] (-1,0) -- (1,0);
	\draw[lgray,line width=1.5pt,->] (0,-1) -- (0,1);
	\node[left] at (-1,0) {\tiny $0$};\node[right] at (1,0) {\tiny $0$};
	\node[below] at (0,-1) {\tiny $1$};\node[above] at (0,1) {\tiny $1$};
}
\quad
&
\quad
\tikz{0.6}{
	\draw[lgray,line width=1.5pt,->] (-1,0) -- (1,0);
	\draw[lgray,line width=1.5pt,->] (0,-1) -- (0,1);
	\node[left] at (-1,0) {\tiny $0$};\node[right] at (1,0) {\tiny $1$};
	\node[below] at (0,-1) {\tiny $1$};\node[above] at (0,1) {\tiny $0$};
}
\quad
\\[1.3cm]
\quad
$1$
\qquad
& 
\quad
$\dfrac{q(1-y/x)}{1-qy/x}$
\quad
& 
\quad
$\dfrac{1-q}{1-qy/x}$
\quad
\\[0.7cm]
\hline
\quad
\tikz{0.6}{
	\draw[lgray,line width=1.5pt,->] (-1,0) -- (1,0);
	\draw[lgray,line width=1.5pt,->] (0,-1) -- (0,1);
	\node[left] at (-1,0) {\tiny $1$};\node[right] at (1,0) {\tiny $1$};
	\node[below] at (0,-1) {\tiny $1$};\node[above] at (0,1) {\tiny $1$};
}
\quad
&
\quad
\tikz{0.6}{
	\draw[lgray,line width=1.5pt,->] (-1,0) -- (1,0);
	\draw[lgray,line width=1.5pt,->] (0,-1) -- (0,1);
	\node[left] at (-1,0) {\tiny $1$};\node[right] at (1,0) {\tiny $1$};
	\node[below] at (0,-1) {\tiny $0$};\node[above] at (0,1) {\tiny $0$};
}
\quad
&
\quad
\tikz{0.6}{
	\draw[lgray,line width=1.5pt,->] (-1,0) -- (1,0);
	\draw[lgray,line width=1.5pt,->] (0,-1) -- (0,1);
	\node[left] at (-1,0) {\tiny $1$};\node[right] at (1,0) {\tiny $0$};
	\node[below] at (0,-1) {\tiny $0$};\node[above] at (0,1) {\tiny $1$};
}
\quad
\\[1.3cm]
\quad
$1$
\quad
& 
\quad
$\dfrac{1-y/x}{1-qy/x}$
\quad
&
\quad
$\dfrac{(1-q)y/x}{1-qy/x}$
\quad 
\\[0.7cm]
\hline
\end{tabular}
\end{align}

\begin{prop}\label{prop:RLL_relation_spin1_uv}
The weights given in~\eqref{weights:spin1_uv_weights} together with~\eqref{weights:Rmatrix_spin1} satisfy the \emph{YBE} relation. That is, for any fixed boundary, we have the following equality of partition functions:
\begin{align}\label{eq:YBE_pictorial_spin1}
\sum_{c_1,c_2,c_{3}}
\tikz{0.9}{
\draw[lgray,line width=1pt,->]
(-2,1) node[above,scale=0.6] {\color{black} $a_2$} -- (-1,0) node[below,scale=0.6] {\color{black} $c_2$} -- (1,0) node[right,scale=0.6] {\color{black} $b_2$};
\draw[lgray,line width=1pt,->] 
(-2,0) node[below,scale=0.6] {\color{black} $a_1$} -- (-1,1) node[above,scale=0.6] {\color{black} $c_1$} -- (1,1) node[right,scale=0.6] {\color{black} $b_1$};
\draw[lgray,line width=4pt,->] 
(0,-1) node[below,scale=0.6] {\color{black} $a_{3}$} -- (0,0.5) node[scale=0.6] {\color{black} $c_{3}$} -- (0,2) node[above,scale=0.6] {\color{black} $b_{3}$};
\node[left] at (-2.2,1) {$\ss (x,\l) \rightarrow $};
\node[left] at (-2.2,0) {$\ss (y,\m) \rightarrow$};
\draw[->](0,-2) node[below] {$\ss (z,\n)$}--(0,-1.7);
}
\quad
=
\quad
\sum_{c_1,c_2,c_{3}}
\tikz{0.9}{
\draw[lgray,line width=1pt,->] 
(-1,1) node[left,scale=0.6] {\color{black} $a_2$} -- (1,1) node[above,scale=0.6] {\color{black} $c_2$} -- (2,0) node[below,scale=0.6] {\color{black} $b_2$};
\draw[lgray,line width=1pt,->] 
(-1,0) node[left,scale=0.6] {\color{black} $a_1$} -- (1,0) node[below,scale=0.6] {\color{black} $c_1$} -- (2,1) node[above,scale=0.6] {\color{black} $b_1$};
\draw[lgray,line width=4pt,->] 
(0,-1) node[below,scale=0.6] {\color{black} $a_{3}$} -- (0,0.5) node[scale=0.6] {\color{black} $c_{3}$} -- (0,2) node[above,scale=0.6] {\color{black} $b_{3}$};
\node[left] at (-1.5,1) {$\ss (x,\l) \rightarrow$};
\node[left] at (-1.5,0) {$\ss  (y ,\m)\rightarrow$};
\draw[->](0,-2) node[below] {$\ss (z,\n)$}--(0,-1.7);
}
\end{align}
\end{prop}
\begin{proof}
We substitute \( x \mapsto -x \), \( y \mapsto -y \), \( z = v \), and \( q^{-\n} = u v \), along with \( \l = 1 \) and \( \m = 1 \), into Equation~\ref{eq:YBE}. After these substitutions, we multiply both sides of Equation~\ref{eq:YBE} by \( u^{-b_{1}-b_{2}} \). 
\end{proof}

We now define \emph{transfer matrices} as single row lattices. Let $V$\,=\ $\text{Span}_{\mathbb{C}} \{\ket{i}\}_{i\in\mathbb{Z}_{\geq0}}$ be an infinite dimensional vector space. Let $\mathbb{V}(N)=V_{1}\otimes V_{2} \otimes\cdots\otimes V_{N}$ where each of $V_{i}$ is a copy of $V$. We define the \emph{row to row transfer matrices} as element in $\text{End}(\mathbb{V}(N))$.

\begin{defn}
We introduce a $\mathrm{T}(x;N)\in \text{End}(\mathbb{V}(N))$ as:
\begin{multline}\label{def:transfermatrix_spin1}
\mathrm{T}(x;N):\, \ket{i_1}\otimes \ket{i_2}\otimes {\cdots}\ket{i_N}
\mapsto\\
\sum_{k_1,k_2,\dots,k_{N} \geq 0} 
\left(
\tikz{1}{\draw[lgray,line width=1.5pt,->] 
(-2,1) node[black,left] {$\ss 0$}-- (4,1) node[black,right] {$*$};
\draw[lgray,line width=4pt,->] 
(3,0) node[below,scale=0.6] {\color{black} $i_{1}$} -- (3,2) node[above,scale=0.6] {\color{black} $k_{1}$};
\draw[lgray,line width=4pt,->] 
(2,0) node[below,scale=0.6] {\color{black} $i_{2}$} -- (2,2) node[above,scale=0.6] {\color{black} $k_{2}$};
\draw[lgray,line width=4pt,->] 
(1,0) -- (1,2);
\draw[lgray,line width=4pt,->] 
(0,0)  -- (0,2);
\draw[lgray,line width=4pt,->] 
(-1,0) node[below,scale=0.6] {\color{black} $i_{N}$} -- (-1,2) node[above,scale=0.6] {\color{black} $k_{N}$};
\node[left] at (-3,1) {$x \rightarrow$};
\node at (0.5,-0.2) {$\cdots$};
\node at (0.5,2.2) {$\cdots$};
\draw[->](3,-1.3) node[below] {$\ss (u_{1},v_{1})$}--(3,-1);
\draw[->](2,-1.3) node[below] {$\ss (u_{2},v_{2})$}--(2,-1);
\draw (0.5,-1.3) node[below] {$\dots $};
}\right)\,\ket{k_1}\otimes \ket{k_2}\otimes {\cdots}\otimes \ket{k}_{N}.
\end{multline}
\end{defn}

Let us briefly describe the boundary conditions of the single-row lattice above. The top, bottom, and left boundaries are fixed. The right boundary is labelled by \( * \), indicating that it is \emph{free}, a particle may exit or not. Given the boundary conditions and the conservation condition at each vertex, we can conclude that there is a unique configuration. To see this, note that at the first vertex from the left the top, bottom and left edges of this vertex are fixed; by conservation, label on the right edge is forced. We can then apply the same reasoning for every subsequent vertex. Then the weight of this single-row lattice is simply the product of the weights of the vertices.

When proving \emph{Cauchy identities} it is necessary to drop the dependence on $N$. We define the lift of $\mathbb{V}(N)$ to an infinite product:
\[
\mathbb{V}(\infty)=\text{Span}_{\mathbb{C}}\left\{ \bigotimes^{\infty}_{k=1}\ket{i_{k}}\right\}
\]
where $i_{k}\in\mathbb{Z}_{\geq 0}$ for all $k\geq 0$ have the stability property: there exists an positive integer $M$ such that $i_{k}=0$ for all $k\geq M$. We write $\mathbb{V}^{*}(\infty)$ to denote dual of $\mathbb{V}(\infty)$. Then the corresponding lift of the operator~\ref{def:transfermatrix_spin1} is denoted by:

\[
\mathrm{T}(x)\,=\,\mathrm{T}(x;\infty)
\]

\begin{defn}\label{def:spin1_uv_HL}
Let \( \lambda = (\lambda_{1} \geq \lambda_{2} \geq \cdots) \) and \( \mu = (\mu_{1} \geq \mu_{2} \geq \cdots) \) be two partitions to which we associate the vectors:
\[
\bra{\lambda} := \bigotimes_{i=1}^{\infty} \bra{m_{i}(\lambda)} \in \mathbb{V}^{*}(\infty),
\qquad\text{ and }\qquad
\ket{\mu} := \bigotimes_{i=1}^{\infty} \ket{m_{i}(\mu)} \in \mathbb{V}(\infty),
\]
where \( m_{i}(\lambda) \) denotes the number of rows of size \( i \) in the Young diagram of \( \lambda \). For partitions \( \mu \subseteq \lambda \), we define the functions \( \mathfrak{J}^{\buv}_{\lambda/\mu} \) as:
\begin{equation}
  \JJ^{\buv}_{\lambda/\mu}(x_{1}, \dots, x_{n}) := \bra{\lambda}\, \mathrm{T}(x_{n}) \cdots \mathrm{T}(x_{1})\, \ket{\mu}.
\end{equation}

In other words, $\JJ^{\buv}_{\lambda/\mu}$ is the partition function of the following lattice:

\[
\tikz{0.8}{\foreach\y in {-1,0.5,2,3.5,5}{
\draw[lgray,line width=1pt,->] (0,\y) node[left,black]{$\ss 0$}-- (8,\y) node[right,black]{$\ss *$};
}
\foreach\x in {1,2.5,4,5.5,7}{
\draw[lgray,line width=4pt,->] (\x,-2) -- (\x,6);
\node at (2.5,-2.5) {$ \dots$};
\node at (3.5,-2.5) {$ \dots$};
\node at (5.5,-2.5) {$\ss m_{2}(\lambda)$};
\node at (7,-2.5) {$\ss m_{1}(\lambda)$};
\node[black] at (2.5,6.5) {$ \dots$};
\node at (3.5,6.5) {$ \dots$};
\node[black] at (5.5,6.5) {$\ss m_{2}(\mu)$};
\node[black] at (7,6.5) {$\ss m_{1}(\mu)$};
\draw[->](7,-3.5) node[below] {$\ss (u_{1},v_{1})$}--(7,-3.2);
\draw[->](5.5,-3.5) node[below] {$\ss (u_{2},v_{2})$}--(5.5,-3.2);
\draw[->](-1.3,5) node[left,black] {$\ss x_{1}$}--(-1,5);
\draw(-1.5,2) node[left,black] {$\vdots$};
\draw[->](-1.3,-1) node[left,black] {$\ss x_{n}$}--(-1,-1);
}
}
\]
using the weights~\ref{weights:spin1_uv_weights}.
\end{defn}

\

\begin{thm}\label{thm:JJ_symmetric}
The polynomials $\JJ^{\buv}_{\lambda}(x_{1},\dots,x_{n})$ are invariant under permutation of the variables $x$.
\end{thm}
\begin{proof}
Observe that to prove the symmetry of \( \mathfrak{J}^{\buv}_{\lambda}(x_1, \dots, x_n) \), it suffices to show that the transfer matrices \( \mathrm{T}(x) \) and \( \mathrm{T}(y) \) commute. The product \( \mathrm{T}(x)\mathrm{T}(y) \) corresponds to stacking two transfer matrices vertically, one on top of the other; that is, the product is equivalent to considering a two-row lattice configuration. We can then freely introduce ``cross'' at the left end of the lattice as shown below:
\begin{multline}
\tikz{0.7}{
\draw[lgray,line width=1.5pt,->] 
(-2,1) node[black,left] {$\ss 0$}-- (4,1) node[black,right] {$\ss *$};
\draw[lgray,line width=1.5pt,->] 
(-2,2) node[black,left] {$\ss 0$}-- (4,2) node[black,right] {$\ss *$};
\draw[lgray,line width=4pt,->] 
(3,0) -- (3,3) ;
\draw[lgray,line width=4pt,->] 
(1,0) -- (1,3);
\draw[lgray,line width=4pt,->] 
(0,0) -- (0,3);
\draw[lgray,line width=4pt,->] 
(-1,0) -- (-1,3);
\node[left] at (-4,1) {$y \rightarrow$};
\node[left] at (-4,2) {$x \rightarrow$};
\node at (2,0.5) {$\cdots$};
\node at (2,1.5) {$\cdots$};
\node at (2,2.5) {$\cdots$};
}   
\,=\,\tikz{0.7}{
\draw[lgray,line width=1.5pt,->] 
(-3,2) node[black,left] {$\ss 0$}-- (-2,1)-- (4,1) node[black,right] {$\ss *$};
\draw[lgray,line width=1.5pt,->] 
(-3,1) node[black,left] {$\ss 0$}-- (-2,2)-- (4,2) node[black,right] {$\ss *$};
\draw[lgray,line width=4pt,->] 
(3,0) -- (3,3) ;
\draw[lgray,line width=4pt,->] 
(1,0) -- (1,3);
\draw[lgray,line width=4pt,->] 
(0,0) -- (0,3);
\draw[lgray,line width=4pt,->] 
(-1,0) -- (-1,3);
\node[left] at (-4,1) {$y \rightarrow$};
\node[left] at (-4,2) {$x \rightarrow$};
\node at (2,0.5) {$\cdots$};
\node at (2,1.5) {$\cdots$};
\node at (2,2.5) {$\cdots$};
}   
\end{multline}

We then repeatedly apply Proposition~\ref{prop:RLL_relation_spin1_uv}, 
and subsequently use the sum-to-unity property Proposition~\ref{prop:sum_to_unity} 
to obtain the desired result, as illustrated below.

\begin{multline}
\tikz{0.7}{
\draw[lgray,line width=1.5pt,->] 
(-2,1) node[black,left] {$\ss 0$}-- (4,1)--(5,2) node[black,right] {$\ss *$};
\draw[lgray,line width=1.5pt,->] 
(-2,2) node[black,left] {$\ss 0$}-- (4,2)  --(5,1)node[black,right] {$\ss *$};
\draw[lgray,line width=4pt,->] 
(3,0) -- (3,3) ;
\draw[lgray,line width=4pt,->] 
(1,0) -- (1,3);
\draw[lgray,line width=4pt,->] 
(0,0) -- (0,3);
\draw[lgray,line width=4pt,->] 
(-1,0) -- (-1,3);
\node[left] at (-3,1) {$y \rightarrow$};
\node[left] at (-3,2) {$x \rightarrow$};
\node at (2,0.5) {$\cdots$};
\node at (2,1.5) {$\cdots$};
\node at (2,2.5) {$\cdots$};
} \,=\,\tikz{0.7}{
\draw[lgray,line width=1.5pt,->] 
(-2,1) node[black,left] {$\ss 0$}-- (4,1)--(5,1) node[black,right] {$\ss *$};
\draw[lgray,line width=1.5pt,->] 
(-2,2) node[black,left] {$\ss 0$}-- (4,2)  --(5,2)node[black,right] {$\ss *$};
\draw[lgray,line width=4pt,->] 
(3,0) -- (3,3) ;
\draw[lgray,line width=4pt,->] 
(1,0) -- (1,3);
\draw[lgray,line width=4pt,->] 
(0,0) -- (0,3);
\draw[lgray,line width=4pt,->] 
(-1,0) -- (-1,3);
\node[left] at (-3,1) {$y \rightarrow$};
\node[left] at (-3,2) {$x \rightarrow$};
\node at (2,0.5) {$\cdots$};
\node at (2,1.5) {$\cdots$};
\node at (2,2.5) {$\cdots$};
} 
\end{multline} 
\end{proof}

\
\begin{prop}
The polynomials $\mathfrak{J}_{\lambda/\mu}^{\buv}$ satisfy the following stability property:
\begin{equation}
 \mathfrak{J}_{\lambda/\mu}^{\buv}(x_{1},\dots,x_{n},x_{n+1})|_{x_{n+1}=0}\,=\,\mathfrak{J}_{\lambda/\mu}^{\buv}(x_{1},\dots,x_{n}).
\end{equation}
\end{prop}
\begin{proof}
From Definition~\ref{def:spin1_uv_HL} of 
\(
\mathfrak{J}^{\buv}_{\lambda/\mu}(x_{1},\dots,x_{n},x_{n+1}),
\)
we obtain
\begin{equation}
\JJ^{\buv}_{\lambda/\mu}(x_{1},\dots,x_{n},x_{n+1})
   \,=\, \bra{\lambda}\, \mathrm{T}(x_{n+1}) \cdots \mathrm{T}(x_{1})\, \ket{\mu} \,
   =\, \sum_{\nu \prec \lambda} 
      \bra{\lambda}\, \mathrm{T}(x_{n+1}) \ket{\nu}\,
      \bra{\nu}\, \mathrm{T}(x_{n}) \cdots \mathrm{T}(x_{1})\, \ket{\mu}.
\end{equation}
The proposition follows once we show that
\[
\bra{\lambda}\mathrm{T}(0)\ket{\nu} \;=\; \delta_{\lambda,\nu}.
\]

This identity is immediate from the graphical interpretation of the transfer matrix.  
At \(x=0\), the weight of the vertex
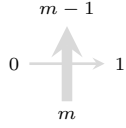
\[
\tikz{0.5}{
\draw[lgray,line width=1pt,->] (-1,0) -- (1,0);
\draw[lgray,line width=4pt,->] (0,-1) -- (0,1);
\node[left] at (-1,0) {\tiny $0$};
\node[right] at (1,0) {\tiny $1$};
\node[below] at (0,-1) {\tiny $m$};
\node[above] at (0,1) {\tiny $m-1$};
}
\]
vanishes. Given that particles only enter from the bottom, the only admissible configuration is the one in which all particles entering from the bottom travel straight upward without turning to the right. This completes the proof.
\end{proof}

\

\subsection{Dual weights}

In order to prove Cauchy identities, we need to consider an alternative version of the vertex weights~\ref{weights:spin1_uv_weights}. This convention is obtained by dividing all the vertices by $(x+v)/(1+ux)$ and then inverting the spectral parameter $x$ in the weights~\ref{weights:spin1_uv_weights}. The Yang--Baxter equation~\ref{eq:YBE_pictorial_spin1} is preserved under this transformation (up to inversion of $x\mapsto -x^{-1}$ and $y\mapsto -y^{-1}$). The alternative weights are tabulated below:

\begin{align}
\label{dualweights:spin1_uv_weights_nonreversed}
\begin{tabular}{|c|c|}
\hline
\quad
\tikz{0.7}{
\draw[lgray,line width=1pt,->] (-1,0) -- (1,0);
\draw[lgray,line width=4pt,->] (0,-1) -- (0,1);
\node[left] at (-1,0) {\tiny $0$};\node[right] at (1,0) {\tiny $0$};
\node[below] at (0,-1) {\tiny $m$};\node[above] at (0,1) {\tiny $m$};
}
\quad
&
\quad
\tikz{0.7}{
\draw[lgray,line width=1pt,->] (-1,0) -- (1,0);
\draw[lgray,line width=4pt,->] (0,-1) -- (0,1);
\node[left] at (-1,0) {\tiny $0$};\node[right] at (1,0) {\tiny $1$};
\node[below] at (0,-1) {\tiny $m$};\node[above] at (0,1) {\tiny $m-1$};
}
\quad
\\[1.3cm]
\quad
$\dfrac{x+uq^{m}}{1+vx}$
\quad
& 
\qquad
$ \dfrac{(1-q^{m})}{1+vx}$
\qquad
\\[0.7cm]
\hline
\quad
\tikz{0.7}{
\draw[lgray,line width=1pt,->] (-1,0) -- (1,0);
\draw[lgray,line width=4pt,->] (0,-1) -- (0,1);
\node[left] at (-1,0) {\tiny $1$};\node[right] at (1,0) {\tiny $0$};
\node[below] at (0,-1) {\tiny $m$};\node[above] at (0,1) {\tiny $m+1$};
}
\quad
&
\quad
\tikz{0.7}{
\draw[lgray,line width=1pt,->] (-1,0) -- (1,0);
\draw[lgray,line width=4pt,->] (0,-1) -- (0,1);
\node[left] at (-1,0) {\tiny $1$};\node[right] at (1,0) {\tiny $1$};
\node[below] at (0,-1) {\tiny $m$};\node[above] at (0,1) {\tiny $m$};
}
\quad
\\[1.3cm] 
\quad
$\dfrac{(1-uvq^{m})x}{1+vx}$
\quad
&
\quad
$\dfrac{(1+v xq^{m})}{1+vx}$
\quad
\\[0.7cm]
\hline
\end{tabular} 
\end{align}   
Graphically, we distinguish the vertices in~\ref{dualweights:spin1_uv_weights_nonreversed} 
from their counterparts in~\ref{weights:spin1_uv_weights} by reversing the orientation 
of the horizontal line and interchanging $0$ and $1$ on these lines. 
We then find that these vertices are converted to:

\begin{align}
\label{dualweights:spin1_uv_weights_reversed}
\begin{tabular}{|c|c|}
\hline
\quad
\tikz{0.7}{
\draw[lred,line width=1pt,<-] (-1,0) -- (1,0);
\draw[lgray,line width=4pt,->] (0,-1) -- (0,1);
\node[left] at (-1,0) {\tiny $1$};\node[right] at (1,0) {\tiny $1$};
\node[below] at (0,-1) {\tiny $m$};\node[above] at (0,1) {\tiny $m$};
}
\quad
&
\quad
\tikz{0.7}{
\draw[lred,line width=1pt,<-] (-1,0) -- (1,0);
\draw[lgray,line width=4pt,->] (0,-1) -- (0,1);
\node[left] at (-1,0) {\tiny $1$};\node[right] at (1,0) {\tiny $0$};
\node[below] at (0,-1) {\tiny $m$};\node[above] at (0,1) {\tiny $m-1$};
}
\quad
\\[1.3cm]
\quad
$\dfrac{x+uq^{m}}{1+vx}$
\quad
& 
\qquad
$ \dfrac{(1-q^{m})}{1+vx}$
\qquad
\\[0.7cm]
\hline
\quad
\tikz{0.7}{
\draw[lred,line width=1pt,<-] (-1,0) -- (1,0);
\draw[lgray,line width=4pt,->] (0,-1) -- (0,1);
\node[left] at (-1,0) {\tiny $0$};\node[right] at (1,0) {\tiny $1$};
\node[below] at (0,-1) {\tiny $m$};\node[above] at (0,1) {\tiny $m+1$};
}
\quad
&
\quad
\tikz{0.7}{
\draw[lred,line width=1pt,<-] (-1,0) -- (1,0);
\draw[lgray,line width=4pt,->] (0,-1) -- (0,1);
\node[left] at (-1,0) {\tiny $0$};\node[right] at (1,0) {\tiny $0$};
\node[below] at (0,-1) {\tiny $m$};\node[above] at (0,1) {\tiny $m$};
}
\quad
\\[1.3cm] 
\quad
$\dfrac{(1-uvq^{m})x}{1+vx}$
\quad
&
\quad
$\dfrac{(1+v xq^{m})}{1+vx}$
\quad
\\[0.7cm]
\hline
\end{tabular} 
\end{align}   

In this way we define our \emph{dual vertex weights} $\mathrm{W}^{*}_{x;(u,v)}(a,b,c,d)$, whose non zero values are indicated above~\ref{dualweights:spin1_uv_weights_reversed}.

\begin{equation}\label{weights:spin1_uv_generic_reversed}
\mathrm{W}^{*}_{x;(u,v)}(a,b,c,d) \equiv \mathrm{W}^{*}_{x;(u,v)}\left(\tikz{0.9}{
\draw[lred,line width=1pt,<-] (-1,0) -- (1,0);
\draw[lgray,line width=4pt,->] (0,-1) -- (0,1);
\node[left] at (-1,0) {\tiny $b$};\node[right] at (1,0) {\tiny $d$};
\node[below] at (0,-1) {\tiny $a$};\node[above] at (0,1) {\tiny $c$};
\node[right] at (1.7,0) {$\leftarrow x $};
\node[below] at (0,-1.4) {$\uparrow$};
\node[below] at (0,-1.9) {$(u,v)$};
}\right)
\end{equation}

There is a notable resemblance between the weights in~\ref{weights:spin1_uv_generic_reversed} and the original ones in~\ref{weights:spin1_uv_weights}. By reflecting the vertices in~\ref{dualweights:spin1_uv_weights_reversed} across the horizontal axis and interchanging \( u \leftrightarrow v \), we nearly recover the original weights~\ref{weights:spin1_uv_weights}, with only a minor discrepancy in the second and third vertex weights. This mismatch can be resolved through a simple gauge transformation. More precisely, we observe that:

\begin{equation}\label{eq:weights_and_dualweights}
\mathrm{W}_{x,(u,v)}(a,b,c,d)\,=\,\dfrac{(uv;q)_{c}}{(uv;q)_{a}} \,\dfrac{(q;q)_{a}}{(q;q)_{c}}\,  \mathrm{W}^{*}_{x,(v,u)}(c,b,a,d). 
\end{equation}

\begin{defn}
We introduce a $\mathrm{T}^{*}(x;N)\in \text{End}(\mathbb{V}(N))$ as:
\begin{multline}\label{def:transfermatrix_spin1_dual}
\mathrm{T}^{*}(x;N):\, \ket{i_1}\otimes \ket{i_2}\otimes {\cdots}\ket{i_N}
\mapsto\\
\sum_{k_1,k_2,\dots,k_{N} \geq 0} 
\left(
\tikz{1}{\draw[lred,line width=1.5pt,<-] 
(-2,1) node[black,left] {$\ss 0$}-- (4,1) node[black,right] {$*$};
\draw[lgray,line width=4pt,->] 
(3,0) node[below,scale=0.6] {\color{black} $i_{1}$} -- (3,2) node[above,scale=0.6] {\color{black} $k_{1}$};
\draw[lgray,line width=4pt,->] 
(2,0) node[below,scale=0.6] {\color{black} $i_{2}$} -- (2,2) node[above,scale=0.6] {\color{black} $k_{2}$};
\draw[lgray,line width=4pt,->] 
(1,0) -- (1,2);
\draw[lgray,line width=4pt,->] 
(0,0)  -- (0,2);
\draw[lgray,line width=4pt,->] 
(-1,0) node[below,scale=0.6] {\color{black} $i_{N}$} -- (-1,2) node[above,scale=0.6] {\color{black} $k_{N}$};
\node[left] at (-3,1) {$x \rightarrow$};
\node at (0.5,-0.2) {$\cdots$};
\node at (0.5,2.2) {$\cdots$};
\draw[->](3,-1.3) node[below] {$\ss (u_{1},v_{1})$}--(3,-1);
\draw[->](2,-1.3) node[below] {$\ss (u_{2},v_{2})$}--(2,-1);
\draw (0.5,-1.3) node[below] {$\dots $};
}\right)\,\ket{k_1}\otimes \ket{k_2}\otimes {\cdots}\otimes \ket{k}_{N}.
\end{multline}
\end{defn}

Let us briefly describe the boundary conditions of the single-row lattice above. The top, bottom, and left boundaries are fixed. The right boundary is again labelled by \( * \), indicating that it is \emph{free}, a particle may enter or not. Given the boundary conditions and the conservation condition at each vertex, we can conclude that there is a unique configuration. Then the weight of this single-row lattice is simply the product of the weights of the vertices.

Similar to $\mathrm{T}(x)$, we write $\mathrm{T}^{*}(x)\,=\,\mathrm{T}^{*}(x;\infty)$ to denote the corresponding lift of the operator~\ref{def:transfermatrix_spin1_dual} to $\mathbb{V}(\infty)$.

\begin{prop}\label{prop:dua_RLL_relation_spin1_uv}
The weights given in~\eqref{weights:spin1_uv_weights} together with~\eqref{weights:spin1_uv_generic_reversed} satisfy the \emph{YBE} equation with the following weights:

\begin{align}\label{weights:Rmatrix_spin1_dual}
\begin{tabular}{|c|c|c|}
\hline
\quad
\tikz{0.6}{
	\draw[lred,line width=1.5pt,<-] (-1,0) -- (1,0);
	\draw[lgray,line width=1.5pt,->] (0,-1) -- (0,1);
	\node[left] at (-1,0) {\tiny $1$};\node[right] at (1,0) {\tiny $1$};
	\node[below] at (0,-1) {\tiny $0$};\node[above] at (0,1) {\tiny $0$};
}
\quad
&
\quad
\tikz{0.6}{
	\draw[lred,line width=1.5pt,<-] (-1,0) -- (1,0);
	\draw[lgray,line width=1.5pt,->] (0,-1) -- (0,1);
	\node[left] at (-1,0) {\tiny $1$};\node[right] at (1,0) {\tiny $1$};
	\node[below] at (0,-1) {\tiny $1$};\node[above] at (0,1) {\tiny $1$};
}
\quad
&
\quad
\tikz{0.6}{
	\draw[lred,line width=1.5pt,<-] (-1,0) -- (1,0);
	\draw[lgray,line width=1.5pt,->] (0,-1) -- (0,1);
	\node[left] at (-1,0) {\tiny $1$};\node[right] at (1,0) {\tiny $0$};
	\node[below] at (0,-1) {\tiny $1$};\node[above] at (0,1) {\tiny $0$};
}
\quad
\\[1.3cm]
\quad
$1$
\qquad
& 
\quad
$\dfrac{q(1-xy)}{1-qxy}$
\quad
& 
\quad
$\dfrac{1-q}{1-qxy}$
\quad
\\[0.7cm]
\hline
\quad
\tikz{0.6}{
	\draw[lred,line width=1.5pt,<-] (-1,0) -- (1,0);
	\draw[lgray,line width=1.5pt,->] (0,-1) -- (0,1);
	\node[left] at (-1,0) {\tiny $0$};\node[right] at (1,0) {\tiny $0$};
	\node[below] at (0,-1) {\tiny $1$};\node[above] at (0,1) {\tiny $1$};
}
\quad
&
\quad
\tikz{0.6}{
	\draw[lred,line width=1.5pt,<-] (-1,0) -- (1,0);
	\draw[lgray,line width=1.5pt,->] (0,-1) -- (0,1);
	\node[left] at (-1,0) {\tiny $0$};\node[right] at (1,0) {\tiny $0$};
	\node[below] at (0,-1) {\tiny $0$};\node[above] at (0,1) {\tiny $0$};
}
\quad
&
\quad
\tikz{0.6}{
	\draw[lred,line width=1.5pt,<-] (-1,0) -- (1,0);
	\draw[lgray,line width=1.5pt,->] (0,-1) -- (0,1);
	\node[left] at (-1,0) {\tiny $0$};\node[right] at (1,0) {\tiny $1$};
	\node[below] at (0,-1) {\tiny $0$};\node[above] at (0,1) {\tiny $1$};
}
\quad
\\[1.3cm]
\quad
$1$
\quad
& 
\quad
$\dfrac{1-xy}{1-qxy}$
\quad
&
\quad
$\dfrac{(1-q)xy}{1-qxy}$
\quad 
\\[0.7cm]
\hline
\end{tabular}
\end{align}

That is, for any fixed boundary, we have the following equality of partition functions:
\begin{align}\label{eq:YBE_pictorial_spin1_dual}
\sum_{c_1,c_2,c_{3}}
\tikz{0.9}{
\draw[lred,line width=1pt,<-]
(-2,1) node[above,scale=0.6] {\color{black} $a_2$} -- (-1,0) node[below,scale=0.6] {\color{black} $c_2$} -- (1,0) node[right,scale=0.6] {\color{black} $b_2$};
\draw[lgray,line width=1pt,->] 
(-2,0) node[below,scale=0.6] {\color{black} $a_1$} -- (-1,1) node[above,scale=0.6] {\color{black} $c_1$} -- (1,1) node[right,scale=0.6] {\color{black} $b_1$};
\draw[lgray,line width=4pt,->] 
(0,-1) node[below,scale=0.6] {\color{black} $a_{3}$} -- (0,0.5) node[scale=0.6] {\color{black} $c_{3}$} -- (0,2) node[above,scale=0.6] {\color{black} $b_{3}$};
\node[left] at (-2.2,1) {$\ss x \rightarrow $};
\node[left] at (-2.2,0) {$\ss y \rightarrow$};
\draw[->](0,-2) node[below] {$\ss (u,v)$}--(0,-1.7);
}
\quad
=
\quad
\sum_{c_1,c_2,c_{3}}
\tikz{0.9}{
\draw[lred,line width=1pt,<-] 
(-1,1) node[left,scale=0.6] {\color{black} $a_2$} -- (1,1) node[above,scale=0.6] {\color{black} $c_2$} -- (2,0) node[below,scale=0.6] {\color{black} $b_2$};
\draw[lgray,line width=1pt,->] 
(-1,0) node[left,scale=0.6] {\color{black} $a_1$} -- (1,0) node[below,scale=0.6] {\color{black} $c_1$} -- (2,1) node[above,scale=0.6] {\color{black} $b_1$};
\draw[lgray,line width=4pt,->] 
(0,-1) node[below,scale=0.6] {\color{black} $a_{3}$} -- (0,0.5) node[scale=0.6] {\color{black} $c_{3}$} -- (0,2) node[above,scale=0.6] {\color{black} $b_{3}$};
\node[left] at (-1.5,1) {$\ss x \rightarrow$};
\node[left] at (-1.5,0) {$\ss  y \rightarrow$};
\draw[->](0,-2) node[below] {$\ss (u,v)$}--(0,-1.7);
}
\end{align}
\end{prop}
\begin{proof}

We begin by dividing both sides of the Equation~\ref{eq:YBE_pictorial_spin1} by \( \dfrac{x+v}{1+ux} \) and then make the substitutions \( x \mapsto x^{-1} \). Lastly, we complement the labels on the horizontal edges along the row associated with the spectral parameter \( x \).
\end{proof}

\begin{prop}\label{prop:commutationrelation_spin1_Cauchy}
 The following exchange relation holds:
 \begin{equation}\label{eq:communtation_t_and_tt}
     \,\mathrm{T}^{*}(x)\,\mathrm{T}(y)\,=\,\dfrac{1-qxy}{1-xy}\,\mathrm{T}(y)\,\mathrm{T}^{*}(x)
 \end{equation}
provided that
\[
\left| \dfrac{x + u}{1 + v x} \cdot \dfrac{y + v}{1 + u y} \right| < 1.
\]
\end{prop}

\begin{proof}
We begin considering the following two row lattice with a ``cross" attached at the left boundary:
\begin{equation}
\tikz{0.7}{
\draw[lred,line width=1.5pt,<-] 
(-3,2) node[black,left] {$\ss 0$}-- (-2,1)-- (4,1) node[black,right] {$\ss *$};
\draw[lgray,line width=1.5pt,->] 
(-3,1) node[black,left] {$\ss 0$}-- (-2,2)-- (4,2) node[black,right] {$\ss *$};
\draw[lgray,line width=4pt,->] 
(3,0) -- (3,3) ;
\draw[lgray,line width=4pt,->] 
(2,0) -- (2,3);
\draw[lgray,line width=4pt,->] 
(1,0) -- (1,3);
\draw[lgray,line width=4pt,->] 
(0,0) -- (0,3);
\node[left] at (-4,1) {$y \rightarrow$};
\node[left] at (-4,2) {$x \rightarrow$};
\node at (-1,0.5) {$\cdots$};
\node at (-1,1.5) {$\cdots$};
\node at (-1,2.5) {$\cdots$};
} 
\,
\end{equation} 
Then there are two possibilities at the left boundary of the lattice, as shown below:
\begin{multline}
\,
\tikz{0.7}{
\draw[lred,line width=1.5pt,<-] 
(-3,2) node[black,left] {$\ss 0$}-- (-2,1) node[black] {$\ss 0$}-- (4,1) node[black,right] {$\ss *$};
\draw[lgray,line width=1.5pt,->] 
(-3,1) node[black,left] {$\ss 0$}-- (-2,2) node[black] {$\ss 0$}-- (4,2) node[black,right] {$\ss *$};
\draw[lgray,line width=4pt,->] 
(3,0) -- (3,3) ;
\draw[lgray,line width=4pt,->] 
(2,0) -- (2,3);
\draw[lgray,line width=4pt,->] 
(1,0) -- (1,3);
\draw[lgray,line width=4pt,->] 
(0,0) -- (0,3);
\node[left] at (-4,1) {$y \rightarrow$};
\node[left] at (-4,2) {$x \rightarrow$};
\node at (-1,0.5) {$\cdots$};
\node at (-1,1.5) {$\cdots$};
\node at (-1,2.5) {$\cdots$};
}
\,
+
\,
\tikz{0.7}{
\draw[lred,line width=1.5pt,<-] 
(-3,2) node[black,left] {$\ss 0$}-- (-2,1) node[black] {$\ss 1$}-- (4,1) node[black,right] {$\ss *$};
\draw[lgray,line width=1.5pt,->] 
(-3,1) node[black,left] {$\ss 0$}-- (-2,2) node[black] {$\ss 1$}-- (4,2) node[black,right] {$\ss *$};
\draw[lgray,line width=4pt,->] 
(3,0) -- (3,3) ;
\draw[lgray,line width=4pt,->] 
(2,0) -- (2,3);
\draw[lgray,line width=4pt,->] 
(1,0) -- (1,3);
\draw[lgray,line width=4pt,->] 
(0,0) -- (0,3);
\node[left] at (-4,1) {$y \rightarrow$};
\node[left] at (-4,2) {$x \rightarrow$};
\node at (-1,0.5) {$\cdots$};
\node at (-1,1.5) {$\cdots$};
\node at (-1,2.5) {$\cdots$};
} 
\end{multline}
Then the convergence constraint implies that the weight of all the configuration with left boundary occupied with particles vanishes. On the other hand,  we apply the YBE relation to obtain the following equivalent expression, shown below, which concludes the proof.
\begin{multline}
\tikz{0.7}{
\draw[lgray,line width=1.5pt,->] 
(-2,1) node[black,left] {$\ss 0$}-- (4,1)--(5,2) node[black,right] {$\ss *$};
\draw[lred,line width=1.5pt,<-] 
(-2,2) node[black,left] {$\ss 0$}-- (4,2)  --(5,1)node[black,right] {$\ss *$};
\draw[lgray,line width=4pt,->] 
(3,0) -- (3,3) ;
\draw[lgray,line width=4pt,->] 
(2,0) -- (2,3);
\draw[lgray,line width=4pt,->] 
(1,0) -- (1,3);
\draw[lgray,line width=4pt,->] 
(0,0) -- (0,3);
\node[left] at (-3,1) {$y \rightarrow$};
\node[left] at (-3,2) {$x \rightarrow$};
\node at (-1,0.5) {$\cdots$};
\node at (-1,1.5) {$\cdots$};
\node at (-1,2.5) {$\cdots$};
} 
\,
=
\,
\tikz{0.7}{
\draw[lgray,line width=1.5pt,->] 
(-2,1) node[black,left] {$\ss 0$}-- (4,1) node[black,right] {$\ss *$};
\draw[lred,line width=1.5pt,<-] 
(-2,2) node[black,left] {$\ss 0$}-- (4,2) node[black,right] {$\ss *$};
\draw[lgray,line width=4pt,->] 
(3,0) -- (3,3) ;
\draw[lgray,line width=4pt,->] 
(2,0) -- (2,3);
\draw[lgray,line width=4pt,->] 
(1,0) -- (1,3);
\draw[lgray,line width=4pt,->] 
(0,0) -- (0,3);
\node[left] at (-3,1) {$y \rightarrow$};
\node[left] at (-3,2) {$x \rightarrow$};
\node at (-1,0.5) {$\cdots$};
\node at (-1,1.5) {$\cdots$};
\node at (-1,2.5) {$\cdots$};
} 
\end{multline}    
\end{proof}

\

\begin{prop}\label{prop:dualweights_HL}
Let $\lambda$ be a partition, and define 
\[
c_{\lambda} = \prod_{i=1}^{\ell(\lambda)} \dfrac{(u_{i}v_{i};q)_{m_{i}(\lambda)}}{(q;q)_{m_{i}(\lambda)}}
\]
where $\ell(\lambda)$ denotes the length of the partition $\lambda$, and $m_i(\lambda)$ is number of rows of size $i$ in the Young diagram of $\lambda$. Then for any two partitions $\mu\subseteq\lambda$, we have
\begin{equation}
\bra{\mu}\,\tt(x_1) \cdots \tt(x_n) \ket{\lambda} 
= \frac{c_{\lambda}}{c_{\mu}} \, \mathfrak{J}_{\lambda/\mu}^{(\mathbf{v}, \mathbf{u})}(x_1, \dots, x_n).
\end{equation}
\end{prop}

\begin{proof}
The proof follows immediately from the Definition~\ref{def:spin1_uv_HL} and the relation~\ref{eq:weights_and_dualweights}. 
\end{proof}

\

\begin{thm}\label{thm:cauchy_spin1}
Fix two positive numbers $n$ and $m$, and let $\lambda$ and $\mu$ be two partitions. For any complex numbers $x_{1},\dots,x_{n}$ and $y_{1},\dots,y_{m}$ such that for $i,j,k,\ell$:
\[
\left| \dfrac{x_{i} + u_{k}}{1 + v_{k} x_{i}} \cdot \dfrac{y_{j} + v_{\ell}}{1 + u_{\ell} y_{j}} \right| < 1,
\]
then the family of Hall--Littlewood polynomials satisfy the following summation identity:
\begin{multline}\label{eq:spin1_cauchy}
\,\sum_{\kappa}\,\dfrac{c_{\kappa}}{c_{\lambda}}\,\JJ^{(\mathbf{v},\mathbf{u})}_{\kappa/\lambda}(x_{1},\dots,x_{n})\,\, \JJ^{\buv}_{\kappa/\mu}(y_{1},\dots,y_{m})\,   \,\\
=\,\prod^{m}_{j=1}\prod^{n}_{i=1}\,\dfrac{1-q x_{i}y_{j}}{1-x_{i}y_{j}}\,\sum_{\kappa}\,\dfrac{c_{\mu}}{c_{\kappa}}\,\JJ^{(\mathbf{v},\mathbf{u})}_{\mu/\kappa}(x_{1},\dots,x_{n})\,\JJ^{\buv}_{\lambda/\kappa}(y_{1},\dots,y_{m}),
\end{multline}
where the summation on the left hand side is over all partitions that contain both $\lambda$ and $\mu$, and on the right hand side the summation is over all partitions that are contained in both $\lambda$ and $\mu$.
\end{thm}

\begin{proof}
Consider the following:
\begin{align}
\mathcal{Z}_{\lambda,\mu}(x_{1},\dots,x_{n},y_{1},\dots,y_{m})\,&:=\,\bra{\lambda}\tt(x_{1})\cdots \tt(x_{n})\t(y_{1})\cdots\t(y_{m})\ket{\mu}.
\end{align}
Then using~\ref{def:spin1_uv_HL} and ~\ref{prop:dualweights_HL}, we get the following summation:
\

\begin{align}
\mathcal{Z}_{\lambda,\mu}(x_{1},\dots,x_{n},y_{1},\dots,y_{m})\,&=\,\sum_{\kappa}\,\bra{\lambda}\tt(x_{1})\cdots \tt(x_{n})\ket{\kappa}\bra{\kappa}\t(y_{1})\cdots\t(y_{m})\ket{\mu}\\ 
&=\,\sum_{\kappa}\,\dfrac{c_{\kappa}}{c_{\lambda}}\,\JJ^{(\mathbf{v},\mathbf{u})}_{\kappa/\lambda}(x_{1},\dots,x_{n})\,\, \JJ^{\buv}_{\kappa/\mu}(y_{1},\dots,x_{m}).
\end{align}

\

On the other hand, by repeated use of the exchange relation~\ref{eq:communtation_t_and_tt}, we also get the following expression for $\mathcal{Z}_{\lambda,\mu}(x_{1},\dots,x_{m},y_{1},\dots,y_{n})$:

\begin{align}
&=\,\prod^{n}_{j=1}\prod^{m}_{i=1}\,\dfrac{1-q x_{i}y_{j}}{1-x_{i}y_{j}}\,\,\bra{\lambda}\t(y_{1})\cdots \t(y_{n})\tt(x_{1})\cdots\tt(x_{n})\ket{\mu}\\
&=\,\prod^{n}_{j=1}\prod^{m}_{i=1}\,\dfrac{1-q x_{i}y_{j}}{1-x_{i}y_{j}}\,\,\bra{\lambda}\t(y_{1})\cdots \t(y_{n})\ket{\kappa}\bra{\kappa}\tt(x_{1})\cdots\tt(x_{n})\ket{\mu}\\
&=\prod^{n}_{j=1}\prod^{m}_{i=1}\,\dfrac{1-q x_{i}y_{j}}{1-x_{i}y_{j}}\,\sum_{\kappa}\,\dfrac{c_{\mu}}{c_{\kappa}}\,\JJ^{(\mathbf{v},\mathbf{u})}_{\mu/\kappa}(x_{1},\dots,x_{n})\,\JJ^{\buv}_{\lambda/\kappa}(y_{1},\dots,y_{m}).
\end{align}
Then by comparing the two equivalent expressions for $\mathcal{Z}$, we recover the desired Cauchy identity.
\end{proof}

\subsection{Degenerations}\label{subsec:HL-degenerations}
The symmetric functions $\JJ^{\buv}_{\lambda/\mu}$ reduces to many known families of symmetric functions/polynomials. The reduction to Schur polynomials and Hall--Littlewood polynomials are well known. However, we shall include them to give a comprehensive list. We believe that the best way to prove these reductions is by comparing the branching formulae. We begin by writing a branching formula, which is a closed expression for the skew functions in one variable, $\JJ^{\buv}_{\lambda/\mu}(x)$.

\subsubsection{Transfer matrix and partitions}
We now describe how the data of $\lambda$, $\mu$, and the skew shape $\lambda/\mu$ are encoded in a single-row transfer matrix. If a box appears in the $i$-th column (from the left) of the skew diagram, then the right edge of the $i$-th vertex (from the right) in the lattice row is labelled $1$. Below, we illustrate this bijection. To aid understanding, we colour the boxes in the skew diagram to match the corresponding labels in the configuration of the single-row lattice. We emphasise that this colouring has no significance beyond this example.

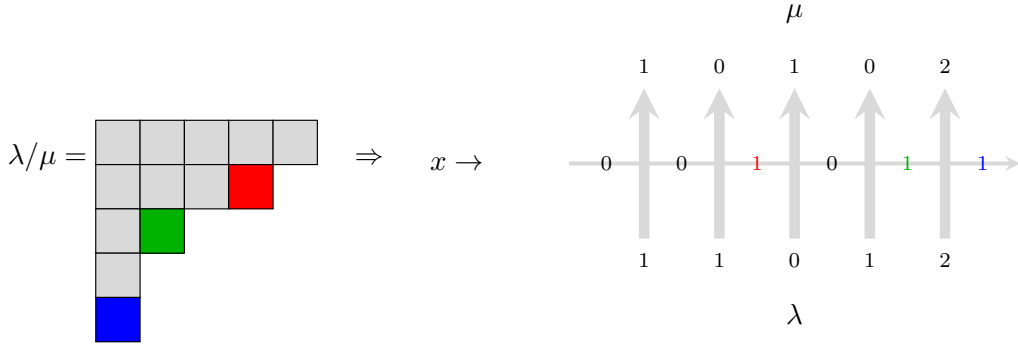
\[
\lambda/\mu =
\begin{ytableau}
*(lgray) & *(lgray) & *(lgray)& *(lgray)& *(lgray)\\
*(lgray) & *(lgray) & *(lgray) & *(red) \\
*(lgray) & *(dgreen) \\
*(lgray)\\
*(blue)
\end{ytableau}
\quad
\Rightarrow
\quad
\tikz{1}{\draw[lgray,line width=1.5pt,->] 
(-2,1)-- (4,1);
\draw[lgray,line width=4pt,->] 
(3,0) node[below] {\color{black} $\ss 2$} -- (3,2) node[above] {\color{black} $\ss 2$};
\draw[lgray,line width=4pt,->] 
(2,0) node[below] {\color{black} $\ss 1$} -- (2,2) node[above] {\color{black} $\ss 0$};
\draw[lgray,line width=4pt,->] 
(1,0) node[below] {\color{black} $\ss 0$} -- (1,2) node[above] {\color{black} $\ss 1$};
\draw[lgray,line width=4pt,->] 
(0,0) node[below] {\color{black} $\ss 1$}  -- (0,2) node[above] {\color{black} $\ss 0$};
\draw[lgray,line width=4pt,->] 
(-1,0) node[below] {\color{black} $\ss 1$} -- (-1,2) node[above] {\color{black} $\ss 1$};
\node[left] at (-3,1) {$x \rightarrow$};
\node at (1,-1) {$\lambda$};
\node at (1,3) {$\mu$};
\node at (-1.5,1) {$\ss 0$};
\node at (-0.5,1) {$\ss 0$};
\node[red] at (0.5,1) {$\ss 1$};
\node at (1.5,1) {$\ss 0$};
\node[dgreen] at (2.5,1) {$\ss 1$};
\node[blue] at (3.5,1) {$\ss 1$};
}
\]

To present the branching formula in a concise form, we introduce the following notation: Whenever $\mu \prec \lambda$, we write
\begin{align*}
c^{++}_{\lambda/\mu} &:= \{i \in \{1,\dots,\lambda_1\} \mid \text{$i$-th and $(i+1)$-th columns are non-empty} \}, \\
c^{+-}_{\lambda/\mu} &:= \{i \in \{1,\dots,\lambda_1\} \mid \text{$i$-th column is non-empty, $(i+1)$-th is empty} \}, \\
c^{-+}_{\lambda/\mu} &:= \{i \in \{1,\dots,\lambda_1\} \mid \text{$i$-th column is empty, $(i+1)$-th is non-empty} \}, \\
c^{--}_{\lambda/\mu} &:= \{i \in \{1,\dots,\lambda_1\} \mid \text{$i$-th and $(i+1)$-th columns are empty and $m_{i}(\lambda)\neq 0$} \}.
\end{align*}
\begin{rmk}
In this notation, we assume that the $(\lambda_1 + 1)$-st column is always empty.
\end{rmk}

\

\begin{prop}\label{prop:HL_family_branching_formula}
For two partitions $\mu \prec \lambda$, the single variable $\JJ^{\buv}_{\lambda/\mu}$ function has the following expression:
\begin{multline}
\JJ^{\buv}_{\lambda/\mu}(x)=\, \prod_{i\,\in\, c^{--}_{\lambda/\mu}}\,\dfrac{1+u_{i}x q^{m_{i}(\lambda)}}{1+u_{i}x} \hspace{0.2cm} \prod_{i\,\in\, c^{+-}_{\lambda/\mu}} \dfrac{ (1-q^{m_{i}(\lambda)})x}{1+u_{i}x}\\
\prod_{i\,\in\, c^{-+}_{\lambda/\mu}}\dfrac{(1-u_{i}v_{i}q^{m_{i}(\lambda)})}{1+u_{i}x}\hspace{0.2cm} \prod_{i\,\in\, c^{++}_{\lambda/\mu}}\dfrac{(x+v_{i}q^{m_{i}(\lambda)})}{1+u_{i}x}
\end{multline}
and is equal to $0$ whenever $\mu\not\prec\lambda$.
\end{prop}
\begin{proof}
 Follows immediately from the analysis above.
\end{proof}

\

We now show how the branching formula reduces to the various known families of symmetric polynomials:

\subsubsection{Schur polynomials} At $\mathbf{u}=\mathbf{v}=\mathbf{0}$ and $q=0$, the branching formula of $\JJ^{\buv}_{\lambda/\mu}$ reduces as follows:
    \begin{align}
        \JJ^{(\mathbf{0},\mathbf{0})}_{\lambda/\mu}(x)|_{q=0}\,&=\,  \prod_{i\,\in\, c^{+-}_{\lambda/\mu}}x \hspace{0.2cm} \prod_{i\,\in\, c^{++}_{\lambda/\mu}}x\\
        &=\,x^{| c^{+-}_{\lambda/\mu}|+| c^{++}_{\lambda/\mu}|}\\
        &=\,x^{|\lambda/\mu|}
    \end{align}
    From this we can conclude that at $\mathbf{u}=\mathbf{v}=\mathbf{0}$ and $q=0$, $\JJ^{\buv}_{\lambda/\mu}$ reduces to the well known \emph{Schur polynomials}.

\subsubsection{Weak Grothendieck polynomials}  At $\mathbf{v}=\mathbf{0}$, $u_{i}=\alpha$ for all $i\in \mathbb{Z}_{>0}$, and $q=0$, the branching formula of $\JJ^{\buv}_{\lambda/\mu}$ reduces as follows:
      \begin{align}
        \JJ^{((\alpha,\dots,\alpha),\mathbf{0})}_{\lambda/\mu}(x)|_{q=0}\,&=\, \prod_{i\,\in\, c^{--}_{\lambda/\mu}}\,\dfrac{1}{1+\alpha x} \hspace{0.2cm} \prod_{i\,\in\, c^{+-}_{\lambda/\mu}} \dfrac{x}{1+\alpha x}
\prod_{i\,\in\, c^{-+}_{\lambda/\mu}}\dfrac{1}{1+\alpha x}\hspace{0.2cm} \prod_{i\,\in\, c^{++}_{\lambda/\mu}}\dfrac{x}{1+\alpha x}
\end{align}
We can rewrite this expression more compactly as:

\begin{align}
        \JJ^{((\alpha,\dots,\alpha),\mathbf{0})}_{\lambda/\mu}(x)|_{q=0}\,=\, \prod_{i\,\in\, c^{--}_{\lambda/\mu}\,\, \cup\,\, c^{-+}_{\lambda/\mu}}\,\dfrac{1}{1+\alpha x} \hspace{0.2cm} \prod_{i\,\in\, c^{+-}_{\lambda/\mu}\,\,\cup\,\, c^{++}_{\lambda/\mu}} \dfrac{x}{1+\alpha x}
\end{align}
Observe that
\[
|c^{--}_{\lambda/\mu}| + |c^{-+}_{\lambda/\mu}| = r(\mu / \widetilde{\lambda}),
\]
where $\widetilde{\lambda} := (\lambda_2, \lambda_3, \dots)$ is the partition obtained by removing the first part of $\lambda$, and $r(\kappa / \nu)$ denotes the number of non zero rows in the skew diagram $\kappa / \nu$.

Thus, the expression further simplifies to:

\begin{align}
        \JJ^{((\alpha,\dots,\alpha),\mathbf{0})}_{\lambda/\mu}(x)|_{q=0}\,=\,\left(\dfrac{1}{1+\alpha x}\right)^{r(\mu/\widetilde{\lambda})}\,\,\left( \dfrac{x}{1+\alpha x}\right)^{|\lambda/\mu|}
\end{align}
This matches with the branching formula for the \emph{weak Grothendieck polynomials} $J_{\lambda}$ at $\beta=0$ from ~\cite[Proposition 8.8]{Gcom:Y2017Duality}.

\subsubsection{Weak dual Grothendieck polynomials}  At $\mathbf{u}=\mathbf{0}$, $v_{i}=\alpha$ for all $i\in \mathbb{Z}_{>0}$, and $q=0$, the branching formula of $\JJ^{\buv}_{\lambda/\mu}$ reduces as follows:
      \begin{align}
        \JJ^{(\mathbf{0},(\alpha,\dots,\alpha))}_{\lambda/\mu}(x)|_{q=0}\,&=\,\prod_{i\,\in\, c^{+-}_{\lambda/\mu}} x
\hspace{0.2cm} \prod_{i\,\in\, c^{++}_{\lambda/\mu}}(x\,+\,\alpha\,\mathbf{1}_{m_{i}(\lambda)=0})
\end{align}
We can rewrite this expression more compactly as:
  \begin{align}
        \JJ^{(\mathbf{0},(\alpha,\dots,\alpha))}_{\lambda/\mu}(x)|_{q=0}\,&=\,x^{r(\lambda/\mu)}\,(x+\alpha)^{|\lambda/\mu|-r(\lambda/\mu)}
\end{align}
This matches with the branching formula for the \emph{weak dual Grothendieck polynomials} $j_{\lambda}$ at $\beta=0$ from ~\cite[Theorem 8.6]{Gcom:Y2017Duality}. 
\subsubsection{Hall--Littlewood polynomials}  At $\mathbf{u}=\mathbf{0}$, $\mathbf{v}=\mathbf{0}$, the branching formula of $\JJ^{\buv}_{\lambda/\mu}$ reduces as follows:
\begin{equation}
\JJ^{(\mathbf{0},\mathbf{0})}_{\lambda/\mu}(x)=\,  \prod_{i\,\in\, c^{+-}_{\lambda/\mu}} (1-q^{m_{i}(\lambda)})x \hspace{0.2cm} \prod_{i\,\in\, c^{++}_{\lambda/\mu}}x
\end{equation}
We can rewrite this expression more compactly as:
  \begin{align}
        \JJ^{(\mathbf{0},\mathbf{0})}_{\lambda/\mu}(x)\,&=\,\varphi_{\lambda/\mu}(q)\,x^{|\lambda/\mu|}\,
\end{align}
where
\[
\varphi_{\lambda/\mu}(q)\,=\,\prod_{i:m_{i}(\lambda)=m_{i}(\mu)+1}(1-q^{m_{i}(\lambda)})
\]
This matches with the branching formula for the \emph{Hall--Littlewood polynomials} $Q_{\lambda/\mu}$ from~\cite[Equation 33]{Hlat:WZJ2016}.

\section{Family of \texorpdfstring{$q$}{q}-Whittaker polynomials}\label{sec:q-whitt-family}
In this section, we obtain the a family of polynomials \( \GG^{(u,v)}_{\lambda} \) as partition functions. These polynomials contain $q$-Whittaker polynomials, Inhomogeneous $q$-Whittaker polynomials, Grothendieck polynomials and dual Grothendieck polynomials. We shall begin by introducing necessary weights and build transfer matrices which we use to define our polynomials. We then list a series of identities and properties about these polynomials all of which are obtained by repeated application of the \emph{Yang--Baxter equation}. 

In order to define the polynomials $\GG^{(u,v)}_{\lambda}$ as partition functions, we need the following \emph{Boltzmann weights}:

\begin{multline}\label{weights:spinL_uv}
\mathbb{W}_{x;(u,v)}(a,b,c,d) \equiv \mathbb{W}_{x;(u,v)}
\left(\tikz{0.9}{
\draw[lgray,line width=4pt,->] (-1,0) -- (1,0);
\draw[lgray,line width=4pt,->] (0,-1) -- (0,1);
\node[left] at (-1,0) {\tiny $b$};\node[right] at (1,0) {\tiny $d$};
\node[below] at (0,-1) {\tiny $a$};\node[above] at (0,1) {\tiny $c$};
\node[left] at (-1.5,0) {$x \rightarrow$};
\node[below] at (0,-1.4) {$\uparrow$};
\node[below] at (0,-1.9) {$(u,v)$};
}\right)\,\\
=\, \delta_{a + b = c + d} \, x^{d}\,
\sum_{p}\,{(ux;q)_{c-p}\,}\,{(v/x;q)_{p}\,}(x/v )^{p-b}\,\binom{c+d-p}{c-p}_{q}\,\binom{b}{p}_{q}
\end{multline}

These weights are obtained by take $q^{-\m}=uv$, $y=v$, $x\mapsto x q^{-\l}$,  divide the weights by $u^{d}$, and then take $q^{-\l}\mapsto 0$ in the general weights~\ref{generalweights}.

We now define \emph{transfer matrices} as single row lattices. Let $V$\,=\ $\text{Span}_{\mathbb{C}} \{\ket{i}\}_{i\in\mathbb{Z}_{\geq 0}}$ be an infinite dimensional vector space. Let $\mathbb{V}(N)=V_{1}\otimes V_{2} \otimes\cdots\otimes V_{N}$ where each of $V_{i}$ is a copy of $V$. We define the \emph{row to row transfer matrices} as element in $\text{End}(\mathbb{V}(N))$.

\begin{defn}
We introduce a $\mathbb{T}(x;N)\in \text{End}(\mathbb{V}(N))$ as:
\begin{multline}\label{def:spinL_transfermatrix}
\mathbb{T}(x;N):\, \ket{i_1}\otimes \ket{i_2}\otimes {\cdots}\ket{i_N}
\mapsto\\
\sum_{k_1,k_2,\dots,k_{N} \geq 0} 
\left(
\tikz{1}{\draw[lgray,line width=4pt,->] 
(-2,1) node[black,left] {$\ss 0$}-- (4,1) node[black,right] {$*$};
\draw[lgray,line width=4pt,->] 
(3,0) node[below,scale=0.6] {\color{black} $i_{1}$} -- (3,2) node[above,scale=0.6] {\color{black} $k_{1}$};
\draw[lgray,line width=4pt,->] 
(2,0) node[below,scale=0.6] {\color{black} $i_{2}$} -- (2,2) node[above,scale=0.6] {\color{black} $k_{2}$};
\draw[lgray,line width=4pt,->] 
(1,0) -- (1,2);
\draw[lgray,line width=4pt,->] 
(0,0)  -- (0,2);
\draw[lgray,line width=4pt,->] 
(-1,0) node[below,scale=0.6] {\color{black} $i_{N}$} -- (-1,2) node[above,scale=0.6] {\color{black} $k_{N}$};
\node[left] at (-3,1) {$x \rightarrow$};
\node at (0.5,-0.2) {$\cdots$};
\node at (0.5,2.2) {$\cdots$};
\draw[->](3,-1.3) node[below] {$\ss (u_{1},v_{1})$}--(3,-1);
\draw[->](2,-1.3) node[below] {$\ss (u_{2},v_{2})$}--(2,-1);
\draw (0.5,-1.3) node[below] {$\dots $};
}\right)\,\ket{k_1}\otimes \ket{k_2}\otimes {\cdots}\otimes \ket{k}_{N}.
\end{multline}
\end{defn}

\

Let us briefly describe the boundary conditions of the single-row lattice above. The top, bottom, and left boundaries are fixed. The right boundary is labelled by \( * \), indicating that it is \emph{free}, meaning any number of particles may exit. Given the boundary conditions and the conservation condition at each vertex, we can conclude that there is a unique configuration. Then the weight of this single-row lattice is simply the product of the weights of the vertices. Similar to the case of the transfer matrices in the previous section, we write
\[
\mathbb{T}(x)\,=\,\mathbb{T}(x;\infty)
\]
to denote the corresponding lift of the operator~\ref{def:spinL_transfermatrix}.

\begin{defn}\label{def:spinL_uv_poly}
Let \( \lambda = (\lambda_{1} \geq \lambda_{2} \geq \cdots) \) be a partition. We associate to it the vector
\[
\bra{\lambda'} := \bigotimes_{i=1}^{\infty} \bra{m_{i}(\lambda')} \in \mathbb{V}^{*}(\infty)
\qquad
\ket{\mu'} := \bigotimes_{i=1}^{\infty} \ket{m_{i}(\mu')} \in \mathbb{V}(\infty)
\]
where \( m_{i}(\lambda') \) denotes the number of columns of size \( i \) in the Young diagram of \( \lambda \). For partitions \( \mu \subseteq \lambda \), we define the polynomial \( \GG^{\buv}_{\lambda/\mu} \) as the partition function of a lattice model:
\begin{equation}\label{eq:spinL_uv_poly}
  \GG^{\buv}_{\lambda/\mu}(x_{1}, \dots, x_{n}) := \bra{\lambda'}\, \mathbb{T}(x_{n}) \cdots \mathbb{T}(x_{1})\, \ket{\mu'}.
\end{equation}

In other words, $\GG^{\buv}_{\lambda/\mu}$ is the partition function of the following lattice:

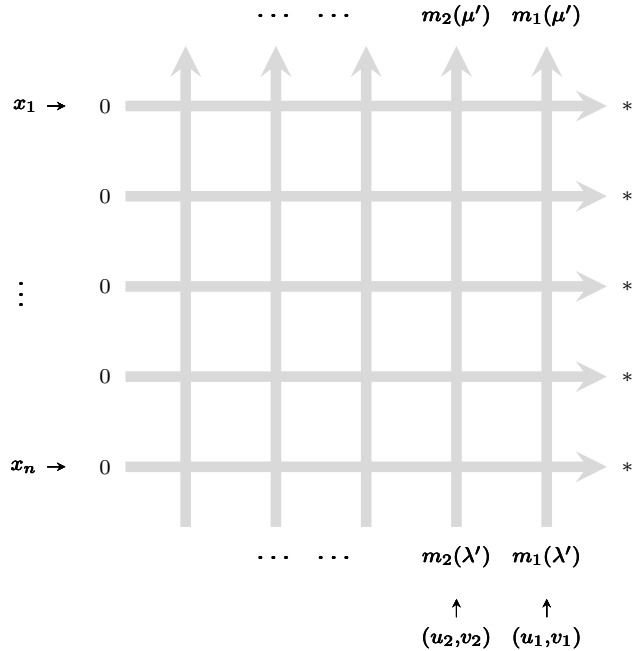
\[
\tikz{0.8}{\foreach\y in {-1,0.5,2,3.5,5}{
\draw[lgray,line width=4pt,->] (0,\y) node[left,black]{$\ss 0$}-- (8,\y) node[right,black]{$\ss *$};
}
\foreach\x in {1,2.5,4,5.5,7}{
\draw[lgray,line width=4pt,->] (\x,-2) -- (\x,6);
\node at (2.5,-2.5) {$ \dots$};
\node at (3.5,-2.5) {$ \dots$};
\node at (5.5,-2.5) {$\ss m_{2}(\lambda')$};
\node at (7,-2.5) {$\ss m_{1}(\lambda')$};
\node[black] at (2.5,6.5) {$ \dots$};
\node at (3.5,6.5) {$ \dots$};
\node[black] at (5.5,6.5) {$\ss m_{2}(\mu')$};
\node[black] at (7,6.5) {$\ss m_{1}(\mu')$};
\draw[->](7,-3.5) node[below] {$\ss (u_{1},v_{1})$}--(7,-3.2);
\draw[->](5.5,-3.5) node[below] {$\ss (u_{2},v_{2})$}--(5.5,-3.2);
\draw[->](-1.3,5) node[left,black] {$\ss x_{1}$}--(-1,5);
\draw(-1.5,2) node[left,black] {$\vdots$};
\draw[->](-1.3,-1) node[left] {$\ss x_{n}$}--(-1,-1);
}
}
\]
using the weights~\ref{weights:spinL_uv}.
\end{defn}

We shall now proceed to prove that the polynomials $\GG^{\buv}_{\lambda/\mu}$ are symmetric. In order to prove this, we need a new set of weights which will satisfy the \emph{Yang--Baxter equation} with the weights~\ref{weights:spinL_uv}.

\begin{equation}\label{weights:spinL_Rmatrix}
\mathbb{R}_{x/y}(a,b,c,d) \equiv \mathbb{R}_{x/y}\left(
{\tikz{0.7}{
\draw[lgray,line width=4pt,->] (-1,0) -- (1,0);
\draw[lgray,line width=4pt,->] (0,-1) -- (0,1);
\node[left] at (-1,0) {\tiny $b$};\node[right] at (1,0) {\tiny $d$};
\node[below] at (0,-1) {\tiny $a$};\node[above] at (0,1) {\tiny $c$};
\node[left] at (-1.5,0) {$x \rightarrow$};
\node[below] at (0,-1.6) {$\uparrow$};
\node[below] at (0,-2.3) {$y$};
}=}
\right)\,=\, \delta_{a + b = c + d} \times 
\left(\dfrac{x}{y}\right)^{d} \left(\dfrac{x}{y};q\right)_{c-b}\, \binom{a}{c-b}_{q}
\end{equation}

Before we present the YBE relations, we pause to present a property of the weights~\eqref{weights:spinL_Rmatrix}. For any fixed pair of non-negative integers $a$ and $b$, the $R$-matrix satisfies the following \emph{sum to unity} condition:
\[
\sum_{c,d} \,\mathbb{R}_{x/y}(a,b,c,d)\,=\, 1.
\]
Graphically, this identity can be represented as:
\[
\tikz{0.7}{
\draw[lgray,line width=4pt,->] (-1,0) -- (1,0);
\draw[lgray,line width=4pt,->] (0,-1) -- (0,1);
\node[right] at (1,0) { $*$};
\node[above] at (0,1) { $*$};
\node[left] at (-1.5,0) {$x \rightarrow$};
\node[below] at (0,-1.6) {$\uparrow$};
\node[below] at (0,-2.3) {$y$};}
\quad = \quad
\tikz{0.7}{
\draw[lgray,line width=4pt,->] (-1,0) -- (1,0);
\draw[lgray,line width=4pt,->] (-1,1) -- (1,1);
\node[left] at (-1.5,0) {$y \rightarrow$};
\node[left] at (-1.5,1) {$x \rightarrow$};}
\]
\begin{rmk}
Observe that at $u=y^{-1}$ and $v=0$, we have the following equality of weights:
\[
\mathbb{R}_{x/y}(a,b,c,d)\,=\,y^{-d}\,\mathbb{W}_{x,(y^{-1},0)}(a,b,c,d)
\]
\end{rmk}

\begin{prop}\label{prop:spinL_RLL_relation}
The weights given in~\eqref{weights:spinL_uv} together with~\eqref{weights:spinL_Rmatrix} satisfy the \emph{YBE} relation. That is, for any fixed boundary, we have the following equality of partition functions:
\begin{align}\label{eq:YBE_pictorial_spinL}
\sum_{c_1,c_2,c_{3}}
\tikz{0.9}{
\draw[lgray,line width=4pt,->]
(-2,1) node[above,scale=0.6] {\color{black} $a_2$} -- (-1,0) node[below,scale=0.6] {\color{black} $c_2$} -- (1,0) node[right,scale=0.6] {\color{black} $b_2$};
\draw[lgray,line width=4pt,->] 
(-2,0) node[below,scale=0.6] {\color{black} $a_1$} -- (-1,1) node[above,scale=0.6] {\color{black} $c_1$} -- (1,1) node[right,scale=0.6] {\color{black} $b_1$};
\draw[lgray,line width=4pt,->] 
(0,-1) node[below,scale=0.6] {\color{black} $a_{3}$} -- (0,0.5) node[scale=0.6] {\color{black} $c_{3}$} -- (0,2) node[above,scale=0.6] {\color{black} $b_{3}$};
\node[left] at (-2.2,1) {$\ss (x,\l) \rightarrow $};
\node[left] at (-2.2,0) {$\ss (y,\m) \rightarrow$};
\draw[->](0,-2) node[below] {$\ss (z,\n)$}--(0,-1.7);
}
\quad
=
\quad
\sum_{c_1,c_2,c_{3}}
\tikz{0.9}{
\draw[lgray,line width=4pt,->] 
(-1,1) node[left,scale=0.6] {\color{black} $a_2$} -- (1,1) node[above,scale=0.6] {\color{black} $c_2$} -- (2,0) node[below,scale=0.6] {\color{black} $b_2$};
\draw[lgray,line width=4pt,->] 
(-1,0) node[left,scale=0.6] {\color{black} $a_1$} -- (1,0) node[below,scale=0.6] {\color{black} $c_1$} -- (2,1) node[above,scale=0.6] {\color{black} $b_1$};
\draw[lgray,line width=4pt,->] 
(0,-1) node[below,scale=0.6] {\color{black} $a_{3}$} -- (0,0.5) node[scale=0.6] {\color{black} $c_{3}$} -- (0,2) node[above,scale=0.6] {\color{black} $b_{3}$};
\node[left] at (-1.5,1) {$\ss (x,\l) \rightarrow$};
\node[left] at (-1.5,0) {$\ss  (y ,\m)\rightarrow$};
\draw[->](0,-2) node[below] {$\ss (z,\n)$}--(0,-1.7);
}
\end{align}
\end{prop}
\begin{proof}
We substitute \( x \mapsto xq^{-\l} \), \( y \mapsto yq^{-\m} \), \( z = v \), and \( q^{-\n} = u v \) into Equation~\ref{eq:YBE}. And then take $q^{-\l}$ and $q^{-\m}$ to zero.
\end{proof}

\begin{thm}\label{thm:GG_symmetric}
The polynomials $\GG^{\buv}_{\lambda/\mu}(x_{1},\dots,x_{n})$ are invariant under permutation of the variables $x$.
\end{thm}
\begin{proof}
same as ~\ref{thm:JJ_symmetric}.    
\end{proof}

\

\begin{prop}
The polynomials $\mathfrak{G}_{\lambda/\mu}^{\buv}$ satisfy the following stability property:
\begin{equation}
 \mathfrak{G}_{\lambda/\mu}^{\buv}(x_{1},\dots,x_{n},x_{n+1})|_{x_{n+1}=0}\,=\,\mathfrak{G}_{\lambda/\mu}^{\buv}(x_{1},\dots,x_{n}).
\end{equation}
\end{prop}
\begin{proof}
From Definition~\ref{def:spinL_uv_poly} of 
\(
\mathfrak{G}^{\buv}_{\lambda/\mu}(x_{1},\dots,x_{n},x_{n+1}),
\)
we obtain
\begin{equation}
\GG^{\buv}_{\lambda/\mu}(x_{1},\dots,x_{n},x_{n+1})
   \,=\, \bra{\lambda}\, \mathrm{T}(x_{n+1}) \cdots \mathrm{T}(x_{1})\, \ket{\mu} \,
   =\, \sum_{\nu \prec \lambda} 
      \bra{\lambda}\, \mathbb{T}(x_{n+1}) \ket{\nu}\,
      \bra{\nu}\, \mathbb{T}(x_{n}) \cdots \mathbb{T}(x_{1})\, \ket{\mu}.
\end{equation}
The proposition follows once we show that
\[
\bra{\lambda}\mathbb{T}(0)\ket{\nu} \;=\; \delta_{\lambda,\nu}.
\]

This identity is immediate from the graphical interpretation of the transfer matrix.  
At \(x=0\), the weight of the vertex
\[
\tikz{0.5}{
\draw[lgray,line width=4pt,->] (-1,0) -- (1,0);
\draw[lgray,line width=4pt,->] (0,-1) -- (0,1);
\node[left] at (-1,0) {\tiny $0$};
\node[right] at (1,0) {\tiny $d$};
\node[below] at (0,-1) {\tiny $a$};
\node[above] at (0,1) {\tiny $c$};
}\,=\, \,\delta_{a = c + d} \,\, x^{d}\,
\,{(ux;q)_{c}\,}\,\binom{c+d}{c}_{q}\,
\]
vanishes whenever $d>0$. Since the leftmost vertex in the graphical representation of $\mathbb{T}$ is of this form, its weight is zero unless \(d=0\).  
Repeating the same argument at each subsequent vertex shows that the only admissible configuration is the one in which all particles entering from the bottom travel straight upward without turning to the right, and the weight of this configuration at \(x=0\) is equal to \(1\).
\end{proof}

\begin{prop}\label{prop:spinL_mixed_RLL_relation}
The weights given in~\eqref{weights:spinL_uv} together with~\eqref{weights:spin1_uv_generic_reversed} satisfy the \emph{YBE} relation. That is, for any fixed boundary, we have the following equality of partition functions:
\begin{align}\label{eq:YBE_pictorial_spinL_mixed}
\sum_{c_1,c_2,c_{3}}
\tikz{0.9}{
\draw[lred,line width=1.5pt,<-]
(-2,1) node[above,scale=0.6] {\color{black} $a_2$} -- (-1,0) node[below,scale=0.6] {\color{black} $c_2$} -- (1,0) node[right,scale=0.6] {\color{black} $b_2$};
\draw[lgray,line width=4pt,->] 
(-2,0) node[below,scale=0.6] {\color{black} $a_1$} -- (-1,1) node[above,scale=0.6] {\color{black} $c_1$} -- (1,1) node[right,scale=0.6] {\color{black} $b_1$};
\draw[lgray,line width=4pt,->] 
(0,-1) node[below,scale=0.6] {\color{black} $a_{3}$} -- (0,0.5) node[scale=0.6] {\color{black} $c_{3}$} -- (0,2) node[above,scale=0.6] {\color{black} $b_{3}$};
\node[left] at (-2.2,1) {$\ss (x,\l) \rightarrow $};
\node[left] at (-2.2,0) {$\ss (y,\m) \rightarrow$};
\draw[->](0,-2) node[below] {$\ss (z,\n)$}--(0,-1.7);
}
\quad
=
\quad
\sum_{c_1,c_2,c_{3}}
\tikz{0.9}{
\draw[lred,line width=1.5pt,<-] 
(-1,1) node[left,scale=0.6] {\color{black} $a_2$} -- (1,1) node[above,scale=0.6] {\color{black} $c_2$} -- (2,0) node[below,scale=0.6] {\color{black} $b_2$};
\draw[lgray,line width=4pt,->] 
(-1,0) node[left,scale=0.6] {\color{black} $a_1$} -- (1,0) node[below,scale=0.6] {\color{black} $c_1$} -- (2,1) node[above,scale=0.6] {\color{black} $b_1$};
\draw[lgray,line width=4pt,->] 
(0,-1) node[below,scale=0.6] {\color{black} $a_{3}$} -- (0,0.5) node[scale=0.6] {\color{black} $c_{3}$} -- (0,2) node[above,scale=0.6] {\color{black} $b_{3}$};
\node[left] at (-1.5,1) {$\ss (x,\l) \rightarrow$};
\node[left] at (-1.5,0) {$\ss  (y ,\m)\rightarrow$};
\draw[->](0,-2) node[below] {$\ss (z,\n)$}--(0,-1.7);
}
\end{align}
\end{prop}

\begin{proof}
We begin by dividing both sides of the Equation~\ref{eq:YBE} by \( \dfrac{x+v}{1+ux} \) and then make the substitutions \( x \mapsto -x^{-1} \), \( y \mapsto y q^{-\m} \), \( z = v \), and \( q^{-\n} = uv \), along with setting \( \l = 1 \). We then complement the labels on the horizontal edges along the line associated with the spectral parameter \( x \). Finally, we take $q^{-\m}$ to zero.
\end{proof}

\begin{prop}\label{prop:commutationrelation_dual_Cauchy}
 The following exchange relation holds:
 \begin{equation}\label{eq:commutation_relation_spin1_dual_Cauchy}
     \dfrac{1}{1+xy}\,\mathrm{T}^{*}(x)\,\mathbb{T}(y)\,=\,\mathbb{T}(y)\,\mathrm{T}^{*}(x).
 \end{equation}
\end{prop}

\begin{proof}
The proof is exactly the same as~\ref{prop:commutationrelation_spin1_Cauchy}.

\begin{multline}
\tikz{0.7}{
\draw[lred,line width=1.5pt,<-] 
(-3,2) node[black,left] {$\ss 0$}-- (-2,1)-- (4,1) node[black,right] {$\ss *$};
\draw[lgray,line width=1.5pt,->] 
(-3,1) node[black,left] {$\ss 0$}-- (-2,2)-- (4,2) node[black,right] {$\ss *$};
\draw[lgray,line width=4pt,->] 
(3,0) -- (3,3) ;
\draw[lgray,line width=4pt,->] 
(2,0) -- (2,3);
\draw[lgray,line width=4pt,->] 
(1,0) -- (1,3);
\draw[lgray,line width=4pt,->] 
(0,0) -- (0,3);
\node[left] at (-4,1) {$y \rightarrow$};
\node[left] at (-4,2) {$x \rightarrow$};
\node at (-1,0.5) {$\cdots$};
\node at (-1,1.5) {$\cdots$};
\node at (-1,2.5) {$\cdots$};
}\,
=
\,
\tikz{0.7}{
\draw[lgray,line width=1.5pt,->] 
(-2,1) node[black,left] {$\ss 0$}-- (4,1)--(5,2) node[black,right] {$\ss *$};
\draw[lred,line width=1.5pt,<-] 
(-2,2) node[black,left] {$\ss 0$}-- (4,2)  --(5,1)node[black,right] {$\ss *$};
\draw[lgray,line width=4pt,->] 
(3,0) -- (3,3) ;
\draw[lgray,line width=4pt,->] 
(2,0) -- (2,3);
\draw[lgray,line width=4pt,->] 
(1,0) -- (1,3);
\draw[lgray,line width=4pt,->] 
(0,0) -- (0,3);
\node[left] at (-3,1) {$y \rightarrow$};
\node[left] at (-3,2) {$x \rightarrow$};
\node at (-1,0.5) {$\cdots$};
\node at (-1,1.5) {$\cdots$};
\node at (-1,2.5) {$\cdots$};
} 
\end{multline}

Then we just need to the consider the entry of the cross on the left hand side. We then note that the red line indicates the swapping $0\leftrightarrow 1$, then from~\ref{weights:spin1_row_weights}, we get that the weight of the cross is equal to the vertex:
\[
\tikz{0.7}{
\draw[lgray,line width=4pt,->] (0,0)node[black,left] {\tiny$ 0$} -- (2,2) node[black,right] {\tiny $ 0$};
\draw[lred,line width=1pt,<-] (0,2) node[black,left] {\tiny $ 0$} -- (2,0) node[black,right] {\tiny $0$};
}\,=\,
\tikz{0.7}{
\draw[lgray,line width=1pt,->] (-1,0) -- (1,0);
\draw[lgray,line width=4pt,->] (0,-1) -- (0,1);
\node[left] at (-1,0) {\tiny $1$};\node[right] at (1,0) {\tiny $1$};
\node[below] at (0,-1) {\tiny $0$};\node[above] at (0,1) {\tiny $0$};
}=\dfrac{(q^{-\m}-x y^{-1} q^{-\m})}{1-x y^{-1} q^{-\m}}
\]

We then set $x\mapsto -x^{-1}$, $y \mapsto yq^{-\m}$, and then take $q^{-\m}\mapsto 0$, to get:

\[
\tikz{0.7}{
\draw[lgray,line width=1pt,->] (-1,0) -- (1,0);
\draw[lgray,line width=4pt,->] (0,-1) -- (0,1);
\node[left] at (-1,0) {\tiny $1$};\node[right] at (1,0) {\tiny $1$};
\node[below] at (0,-1) {\tiny $0$};\node[above] at (0,1) {\tiny $0$};
}=\dfrac{(0+x^{-1} y^{-1})}{1+x^{-1} y^{-1} }=\dfrac{1}{1+xy}
\]
\end{proof}

\begin{thm}\label{thm:cauchy_dual}
For two positive numbers $n$ and $m$, and let $\lambda$ and $\mu$ be two partitions. Then the family of Hall--Littlewood polynomials and the family of $q$-Whittaker polynomials satisfy the following summation identity:
\begin{multline}\label{eq:cauchy_dual}
\,\sum_{\kappa}\,\dfrac{c_{\kappa^{'}}}{c_{\mu^{'}}}\,\GG^{\buv}_{\kappa/\lambda}(x_{1},\dots,x_{n})\, \JJ^{(\mathbf{v},\mathbf{u})}_{\kappa'/\mu'}(y_{1},\dots,y_{m})\,   \,\\
=\,\prod^{m}_{j=1}\prod^{n}_{i=1}\,{(1+x_{i}y_{j})}\,\sum_{\kappa}\,\dfrac{c_{\lambda^{'}}}{c_{\kappa^{'}}}\,\GG^{\buv}_{\mu/\kappa}(x_{1},\dots,x_{n})\,\JJ^{(\mathbf{v},\mathbf{u})}_{\lambda'/\kappa'}(y_{1},\dots,y_{m})
\end{multline}
\end{thm}

\begin{proof}
The proof proceeds analogously to that of Theorem~\ref{thm:cauchy_spin1}, with the commutation relation~\ref{eq:commutation_relation_spin1_dual_Cauchy} replacing the one used there. 
\end{proof}

\subsection{Dual Weights}

We define our \emph{dual vertex weights} $\mathbb{W}^{*}_{x;(u,v)}(a,b,c,d)$ along the equation~\ref{eq:weights_and_dualweights}:

\begin{equation}\label{weights:spinL_dual_weights}
\mathbb{W}^{*}_{x;(u,v)}(a,b,c,d) \equiv \mathbb{W}^{*}_{x;(u,v)}\left(\tikz{0.9}{
\draw[lred,line width=4pt,<-] (-1,0) -- (1,0);
\draw[lgray,line width=4pt,->] (0,-1) -- (0,1);
\node[left] at (-1,0) {\tiny $b$};\node[right] at (1,0) {\tiny $d$};
\node[below] at (0,-1) {\tiny $a$};\node[above] at (0,1) {\tiny $c$};
\node[right] at (1.7,0) {$\leftarrow x $};
\node[below] at (0,-1.4) {$\uparrow$};
\node[below] at (0,-1.9) {$(u,v)$};
}\right)\,:=\,\,\dfrac{(uv;q)_{c}}{(uv;q)_{a}} \,\dfrac{(q;q)_{a}}{(q;q)_{c}}\, \mathbb{W}_{x,(v,u)}(c,b,a,d)\\,
\end{equation}

Observe that the dual vertex weights $\mathbb{W}^{*}_{x;(u,v)}$ are expressed in terms of the original vertex weights $\mathbb{W}_{x,(v,u)}$, with the parameters $u$ and $v$ interchanged.

\begin{defn}
We introduce a $\mathbb{T}^{*}(x;N)\in \text{End}(\mathbb{V}(N))$ as:
\begin{multline}\label{def:transfermatrix_spinL_dual}
\mathbb{T}^{*}(x;N):\, \ket{i_1}\otimes \ket{i_2}\otimes {\cdots}\ket{i_N}
\mapsto\\
\sum_{k_1,k_2,\dots,k_{N} \geq 0} 
\left(
\tikz{1}{\draw[lred,line width=4pt,<-] 
(-2,1) node[black,left] {$\ss 0$}-- (4,1) node[black,right] {$*$};
\draw[lgray,line width=4pt,->] 
(3,0) node[below,scale=0.6] {\color{black} $i_{1}$} -- (3,2) node[above,scale=0.6] {\color{black} $k_{1}$};
\draw[lgray,line width=4pt,->] 
(2,0) node[below,scale=0.6] {\color{black} $i_{2}$} -- (2,2) node[above,scale=0.6] {\color{black} $k_{2}$};
\draw[lgray,line width=4pt,->] 
(1,0) -- (1,2);
\draw[lgray,line width=4pt,->] 
(0,0)  -- (0,2);
\draw[lgray,line width=4pt,->] 
(-1,0) node[below,scale=0.6] {\color{black} $i_{N}$} -- (-1,2) node[above,scale=0.6] {\color{black} $k_{N}$};
\node[left] at (-3,1) {$x \rightarrow$};
\node at (0.5,-0.2) {$\cdots$};
\node at (0.5,2.2) {$\cdots$};
\draw[->](3,-1.3) node[below] {$\ss (u_{1},v_{1})$}--(3,-1);
\draw[->](2,-1.3) node[below] {$\ss (u_{2},v_{2})$}--(2,-1);
\draw (0.5,-1.3) node[below] {$\dots $};
}\right)\,\ket{k_1}\otimes \ket{k_2}\otimes {\cdots}\otimes \ket{k}_{N}.
\end{multline}
\end{defn}

Let us briefly describe the boundary conditions of the single-row lattice above. The top, bottom, and left boundaries are fixed. The right boundary is again labelled by \( * \), indicating that it is \emph{free}, any number of particles may enter.  Given the boundary conditions and the conservation condition at each vertex, we can conclude that there is a unique configuration. Then the weight of this single-row lattice is simply the product of the weights of the vertices. Similar to $\mathbb{T}(x)$, we write

\[
\mathbb{T}^{*}(x)\,=\,\mathbb{T}^{*}(x;\infty)
\]

to denote the corresponding lift of the operator~\ref{def:transfermatrix_spinL_dual}.

\begin{prop}\label{prop:dualweights_qWhittaker}
Let $\lambda$ be a partition, and define 
\[
c_{\lambda} = \prod_{i=1}^{\ell(\lambda)} \dfrac{(uv;q)_{m_{i}(\lambda)}}{(q;q)_{m_{i}(\lambda)}}
\]
where $\ell(\lambda)$ denotes the length of the partition $\lambda$, and $m_i(\lambda)$ is number of rows of size $i$ in the Young diagram of $\lambda$. Then for any two partitions $\mu\subseteq\lambda$, we have
\begin{equation}
\bra{\mu'}\,\TT(x_1) \cdots \TT(x_n) \ket{\lambda'} 
= \dfrac{c_{\lambda^{'}}}{c_{\mu^{'}}} \, \mathfrak{G}_{\lambda/\mu}^{(\mathbf{v}, \mathbf{u})}(x_1, \dots, x_n).
\end{equation}
\end{prop}

\begin{proof}
The proof follows immediately from the Definition~\ref{def:spin1_uv_HL} and the relation~\ref{eq:weights_and_dualweights}. 
\end{proof}

\begin{prop}\label{prop:RLL_spinL_vs_SpinL_}
The transfer matrices $\TT$ and $\T$ satisfy the following commutation relation:
\begin{equation}\label{eq:communtation_T_and_TT}
\TT(x)\,\T(y)\,=\,\dfrac{1}{(x y;q)_{\infty}}\,\T(y)\,\TT(x)  
\end{equation}
whenever
\[
|x|,|y|<\,1.
\]
\end{prop}

\begin{proof}
See Appendix~\ref{sec:appendix-proof}    
\end{proof}
\

\begin{thm}\label{thm:spinL_Cauchy_qwhittaker}
Fix two positive numbers $n$ and $m$, and let $\lambda$ and $\mu$ be two partitions. Then the family of $q$-Whittaker polynomials satisfy the following summation identity (assuming that all the variables are in the unit disc):
\begin{multline}\label{eq:spinL_Cauchy_qwhittaker}
\,\sum_{\kappa}\,\dfrac{c_{\kappa'}}{c_{\lambda'}}\,\GG^{(\mathbf{v},\mathbf{u})}_{\kappa/\lambda}(x_{1},\dots,x_{n})\, \GG^{\buv}_{\kappa/\mu}(y_{1},\dots,y_{m})\,   \\
\,=\,\prod^{n}_{j=1}\prod^{m}_{i=1}\,\dfrac{1}{(x_{i}y_{j};q)_{\infty}}\,\sum_{\kappa}\,\dfrac{c_{\mu'}}{c_{\kappa'}}\,\GG^{(\mathbf{v},\mathbf{u})}_{\mu/\kappa}(x_{1},\dots,x_{n})\,\GG^{\buv}_{\lambda/\kappa}(y_{1},\dots,y_{m}),
\end{multline}
with the left hand sum taken over all partitions that contain both $\mu$ and $\lambda$, and on the right hand side sum is taken over all partition that are contained in both $\lambda$ and $\mu$.
\end{thm}

\begin{proof}
The proof is same as the proof of~\ref{thm:cauchy_spin1} but with the commutation relation~\ref{eq:communtation_T_and_TT}.
\end{proof}

\subsection{Degenerations}\label{subsec:qWhit-degenerations}
As in the case of $\JJ^{\buv}_{\lambda/\mu}$, we now turn to various degenerations of $\GG^{\buv}_{\lambda/\mu}$. The functions $\JJ^{\buv}_{\lambda/\mu}$ specialise to Hall–Littlewood polynomials, weak Grothendieck polynomials and their duals. In contrast, $\GG^{\buv}_{\lambda/\mu}$ includes $q$-Whittaker polynomials, Grothendieck polynomials and their duals. Unlike the situation with $\JJ^{\buv}_{\lambda/\mu}$, the full branching formula for $\GG^{\buv}_{\lambda/\mu}$ is rather cumbersome to state. We therefore restrict ourselves to presenting the branching formulas only in special cases.

\subsubsection{Transfer Matrix and Partitions}
 Let us consider the transfer matrix. The labels on all the edges corresponding to a statistic in skew Young diagrams.  Consider a vertex at site $i$ from the right:
 
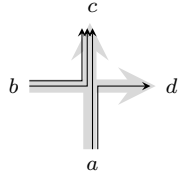
\[
\tikz{0.7}{
\draw[lgray,line width=5pt,->] (-1.2,0) -- (1.2,0);
\draw[lgray,line width=5pt,->] (-0.05,-1.2) -- (-0.05,1.2);
 \node at (0,-1.5) {$\ss a$};
 \node at (-1.5,0) {$\ss b$};
 \node at (0,1.5) {$\ss c$};
 \node at (1.5,0) {$\ss d$};
 \draw[->](0,-1.2)--(0,1.1);
 \draw[->](0.1,-1.2)--(0.1,0)--(1.1,0);
 \draw[->](-1.2,0)--(-0.1,0)--(-0.1,1.1);
 \draw[->](-1.2,0.1)--(-0.2,0.1)--(-0.2,1.1);
}
\]
 The label $a$ on the bottom edge of a vertex corresponds to number of columns of size $i$ in $\lambda$. Similarly, the top label corresponds to the number of boxes in row $i$ in the Young diagram of $\mu$. We can then easily identify each label of a vertex at site $i$ in a single row transfer matrix~\ref{def:spinL_transfermatrix} as follows:
\begin{equation}\label{eq:vertex_interpretation}
\tikz{0.7}{
\draw[lgray,line width=5pt,->] (-1.2,0) -- (1.2,0);
\draw[lgray,line width=5pt,->] (-0.05,-1.2) -- (-0.05,1.2);
 \draw[->](0,-1.1)--(0,1.1);
 \draw[->](0.1,-1.1)--(0.1,0)--(1.1,0);
 \draw[->](-1.1,0)--(-0.1,0)--(-0.1,1.1);
 \draw[->](-1.1,0.1)--(-0.2,0.1)--(-0.2,1.1);
 \node at (0,-1.5) {$\ss \lambda_{i}-\lambda_{i+1}$};
 \node at (-2.7,0) {$\ss \lambda_{i+1}-\mu_{i+1}$};
 \node at (0,1.5) {$\ss \mu_{i}-\mu_{i+1}$};
 \node at (2.2,0) {$\ss \lambda_{i}-\mu_{i}$};
}
\end{equation}
In certain degenerations, we often encounter a constraint on the vertex weights: specifically, the label on the left edge must be weakly smaller than the label on the top edge. In light of this constraint and the above interpretation (~\ref{eq:vertex_interpretation}), the single-variable skew polynomial will vanish unless $\lambda_{i+1} \leq \mu_i$ for all $i$. Observe that this condition is equivalent to requiring the skew diagram $\lambda/\mu$ to be a horizontal strip.

\begin{ex} For $\lambda=(6,5,3,1,1)=(1^{2}\,2^{0}\,3^{1}\,4^{0}\,5^{1}\,6^{1})$ and $\mu=(4,4,1,1)=(1^{2}\,2^{0}\,3^{0}\,4^{2})$, then the conjugate partitions $\lambda'=(1^{1}\,2^{2}\,3^{2}\,4^{0}\,5^{1})$ and $\mu'=(1^{0}\,2^{3}\,3^{0}\,4^{1}\,5^{0})$
\[
\lambda/\mu =
\begin{ytableau}
*(lgray) & *(lgray) & *(lgray)& *(lgray)&*(red)& *(red)\\
*(lgray) & *(lgray) & *(lgray)& *(lgray) & *(dgreen) \\
*(lgray) & *(blue)  & *(blue) \\
*(lgray)\\
*(pink)
\end{ytableau}
\quad
\Rightarrow
\quad
\tikz{1}{\draw[lgray,line width=4pt,->] 
(-2,1)-- (4,1);
\draw[lgray,line width=4pt,->] 
(3,0) node[below] {\color{black} $\ss 1$} -- (3,2) node[above] {\color{black} $\ss 0$};
\draw[lgray,line width=4pt,->] 
(2,0) node[below] {\color{black} $\ss 2$} -- (2,2) node[above] {\color{black} $\ss 3$};
\draw[lgray,line width=4pt,->] 
(1,0) node[below] {\color{black} $\ss 2$} -- (1,2) node[above] {\color{black} $\ss 0$};
\draw[lgray,line width=4pt,->] 
(0,0) node[below] {\color{black} $\ss 0$}  -- (0,2) node[above] {\color{black} $\ss 1$};
\draw[lgray,line width=4pt,->] 
(-1,0) node[below] {\color{black} $\ss 1$} -- (-1,2) node[above] {\color{black} $\ss 0$};
\node at (1,-1) {$\lambda'$};
\node at (1,3) {$\mu'$};
\node at (-1.5,1) {$\ss 0$};
\node[pink] at (-0.5,1) {$\ss 1$};
\node at (0.5,1) {$\ss 0$};
\node[blue] at (1.5,1) {$\ss 2$};
\node[dgreen] at (2.5,1) {$\ss 1$};
\node[red] at (3.5,1) {$\ss 2$};
}
\]    
\end{ex}

For convenience, let us recall the weights~\ref{weights:spinL_uv} that are used to define $\GG^{\buv}_{\lambda/\mu}$:

\begin{multline}
\mathbb{W}_{x;(u,v)}(a,b,c,d) \equiv \mathbb{W}_{x;(u,v)}
\left(\tikz{0.9}{
\draw[lgray,line width=4pt,->] (-1,0) -- (1,0);
\draw[lgray,line width=4pt,->] (0,-1) -- (0,1);
\node[left] at (-1,0) {\tiny $b$};\node[right] at (1,0) {\tiny $d$};
\node[below] at (0,-1) {\tiny $a$};\node[above] at (0,1) {\tiny $c$};
\node[left] at (-1.5,0) {$x \rightarrow$};
\node[below] at (0,-1.4) {$\uparrow$};
\node[below] at (0,-1.9) {$(u,v)$};
}\right)\,\\
=\, \delta_{a + b = c + d} \, x^{d}\,
\sum_{p}\,{(ux;q)_{c-p}\,}\,{(v/x;q)_{p}\,}(x/v )^{p-b}\,\binom{c+d-p}{c-p}_{q}\,\binom{b}{p}_{q}
\end{multline}
where the sum is over non negative numbers such that $0\leq p\leq \min(b,c)$.
\subsubsection{Schur polynomials}
At $q=0$ and $\mathbf{u}=\mathbf{v}=\mathbf{0}$, the weights reduce to the following:
\begin{equation}\label{eq:reduction_spinL_schur}
\mathbb{W}_{x;(0,0)}(a,b,c,d)|_{q=0}\,=\,\mathbf{1}_{b\leq c}\, x^{d}.   
\end{equation}

Then the single variable skew for $\mu \prec \lambda$, we have:
\begin{align}
\GG^{(\mathbf{0},\mathbf{0})}_{\lambda/\mu}|_{q=0}\,&=\,\prod^{\ell(\lambda)}_{i=1}\,x^{\lambda_{i}-\mu_{i}}\,=\,x^{|\lambda/\mu|} 
\end{align}

\subsubsection{Grothendieck polynomials}
At $q=0$ and $u_{i}=\beta$ for all $i\in \mathbb{Z}_{>0}$ and $\mathbf{v}=\mathbf{0}$, the weights reduce to the following:
\begin{equation}\label{eq:reduction_spinL_groth}
\mathbb{W}_{x;(\beta,0)}(a,b,c,d)\big|_{q=0}\,=\,\mathbf{1}_{b\leq c}\, x^{d} (1-\beta x)^{\mathbf{1}_{b<c}}.   
\end{equation}

Then the single variable skew for $\mu \prec \lambda$, we have:
\begin{align}
\GG^{((\beta,\dots,\beta)),\mathbf{0})}_{\lambda/\mu}\big|_{q=0}\,&=\,\prod^{\ell(\lambda)}_{i=1}\,x^{\lambda_{i}-\mu_{i}} (1-\beta x)^{\mathbf{1}_{\lambda_{i+1}<\mu_{i}}} \\
&=\,x^{|\lambda/\mu|}\,(1-\beta x)^{r(\mu/\widetilde{\lambda})}
\end{align}

where $\widetilde{\lambda} := (\lambda_2, \lambda_3, \dots)$ is the partition obtained by removing the first part of $\lambda$, and $r(\kappa / \nu)$ denotes the number of non zero rows in the skew diagram $\kappa / \nu$. This matches with the branching formula for the \emph{ Grothendieck polynomials} $G_{\lambda}$ at $\alpha=0$ from ~\cite[Proposition 8.8]{Gcom:Y2017Duality}.

\subsubsection{dual Grothendieck polynomials}
At $q=0$ and $\mathbf{u}=\mathbf{0}$ and $v_{i}=\beta$ for all $i\in \mathbb{Z}_{>0}$, the weights reduce to the following:
\begin{equation}\label{eq:reduction_spinL_dual_Groth}
\mathbb{W}_{x;(0,\beta)}(a,b,c,d)\big|_{q=0} = 
\begin{cases}
  x^{d} & \text{if } b \leq c, \\
  \beta^{d-a} x^{a} & \text{if } b > c.
\end{cases}
\end{equation}

Then the single variable skew for $\mu \subseteq \lambda$, we have:
\begin{align}
\GG^{(\mathbf{0},(\beta,\dots,\beta))}_{\lambda/\mu}\big|_{q=0}\,
&=\, \beta^{|\lambda/\mu|-c(\lambda/\mu)}x^{c(\lambda/\mu)}
\end{align}
where $c(\kappa / \nu)$ denotes the number of non zero columns in the skew diagram $\kappa / \nu$.  This matches with the branching formula for the \emph{dual Grothendieck polynomials} $g_{\lambda}$ at $\alpha=0$ from ~\cite[Theorem 8.6]{Gcom:Y2017Duality}.

\begin{rmk}
At $u_{i}=v_{i}=1$ for all $i\in\mathbb{Z}_{\geq 0}$ and $q=0$, $\mathfrak{G}^{\buv}_{\lambda}$ have close resemblance to \emph{Hybrid Grothendieck polynomials} introduced in~\cite{Hybrid}.
\end{rmk}

\subsubsection{$q$-Whittaker polynomials}
At $\mathbf{u}=\mathbf{0}$ and $\mathbf{v}=\mathbf{0}$, the weights reduce to the following:
\begin{align}\label{eq:reduction_spinL_Whitt}
\mathbb{W}_{x;(\mathbf{0},\mathbf{0})}\left(\tikz{0.7}{
\draw[lgray,line width=5pt,->] (-1.2,0) -- (1.2,0);
\draw[lgray,line width=5pt,->] (-0.05,-1.2) -- (-0.05,1.2);
 \draw[->](0,-1.1)--(0,1.1);
 \draw[->](0.1,-1.1)--(0.1,0)--(1.1,0);
 \draw[->](-1.1,0)--(-0.1,0)--(-0.1,1.1);
 \draw[->](-1.1,0.1)--(-0.2,0.1)--(-0.2,1.1);
 \node at (0,-1.5) {$\ss \lambda_{i}-\lambda_{i+1}$};
 \node at (-2.7,0) {$\ss \lambda_{i+1}-\mu_{i+1}$};
 \node at (0,1.5) {$\ss \mu_{i}-\mu_{i+1}$};
 \node at (2.2,0) {$\ss \lambda_{i}-\mu_{i}$};
}\right)\, &=\, \mathbf{1}_{\lambda_{i+1}\leq \mu_{i} }\,x^{\lambda_{i}-\mu_{i}}\,\binom{\lambda_{i}-\lambda_{i+1}}{\mu_{i}-\lambda_{i+1}}_{q}\\[0.7em]
&=\, \mathbf{1}_{\lambda_{i+1}\leq \mu_{i} }\,x^{\lambda_{i}-\mu_{i}}\,\dfrac{(q;q)_{\lambda_{i}-\lambda_{i+1}}}{(q;q)_{\lambda_{i}-\mu_{i}}\,(q;q)_{\mu_{i}-\lambda_{i+1}}}
\end{align}

Then the single variable skew for $\mu \prec \lambda$, we have:
\begin{align}
\GG^{(\mathbf{0},\mathbf{0})}_{\lambda/\mu}\,
&=\, x^{|\lambda/\mu|}\,\prod_{i\geq 1}\,\dfrac{(q;q)_{\lambda_{i}-\lambda_{i+1}}}{(q;q)_{\lambda_{i}-\mu_{i}}\,(q;q)_{\mu_{i}-\lambda_{i+1}}}.
\end{align}
This matches precisely with the branching formula for $q$-Whittaker polynomials~\cite[subsection~6.3]{spin-BW2021}.

\subsubsection{Inhomogeneous $q$-Whittaker polynomials}

At $\mathbf{v}=\mathbf{0}$, the weights reduce to the following:
\begin{align}\label{eq:reduction_spinL_dual_inWhit}
\mathbb{W}_{x;(\mathbf{u},\mathbf{0})}\left(\tikz{0.7}{
\draw[lgray,line width=5pt,->] (-1.2,0) -- (1.2,0);
\draw[lgray,line width=5pt,->] (-0.05,-1.2) -- (-0.05,1.2);
 \draw[->](0,-1.1)--(0,1.1);
 \draw[->](0.1,-1.1)--(0.1,0)--(1.1,0);
 \draw[->](-1.1,0)--(-0.1,0)--(-0.1,1.1);
 \draw[->](-1.1,0.1)--(-0.2,0.1)--(-0.2,1.1);
 \node at (0,-1.5) {$\ss \lambda_{i}-\lambda_{i+1}$};
 \node at (-2.7,0) {$\ss \lambda_{i+1}-\mu_{i+1}$};
 \node at (0,1.5) {$\ss \mu_{i}-\mu_{i+1}$};
 \node at (2.2,0) {$\ss \lambda_{i}-\mu_{i}$};
}\right)\, &=\, \mathbf{1}_{\lambda_{i+1}\leq \mu_{i} }\,x^{\lambda_{i}-\mu_{i}}\,(u_{i}x;q)_{\mu_{i}-\lambda_{i+1}}\,\binom{\lambda_{i}-\lambda_{i+1}}{\mu_{i}-\lambda_{i+1}}_{q}\\[0.7em]
&=\, \mathbf{1}_{\lambda_{i+1}\leq \mu_{i} }\,x^{\lambda_{i}-\mu_{i}}\,(u_{i}x;q)_{\mu_{i}-\lambda_{i+1}}\,\dfrac{(q;q)_{\lambda_{i}-\lambda_{i+1}}}{(q;q)_{\lambda_{i}-\mu_{i}}\,(q;q)_{\mu_{i}-\lambda_{i+1}}}
\end{align}

Then the single variable skew for $\mu \prec \lambda$, we have:
\begin{align}
\GG^{(\mathbf{u},\mathbf{0})}_{\lambda/\mu}\,
&=\, x^{|\lambda/\mu|}\,\prod_{i\geq 1} \,(u_{i}x;q)_{\mu_{i}-\lambda_{i+1}}\,\dfrac{(q;q)_{\lambda_{i}-\lambda_{i+1}}}{(q;q)_{\lambda_{i}-\mu_{i}}\,(q;q)_{\mu_{i}-\lambda_{i+1}}}.
\end{align}
This coincides, up to a factor depending on $u_i$, with the branching formula for inhomogeneous $q$-Whittaker polynomials \cite[Proposition~5.1]{spin-K2024} at $a_{i}=0$.

\section{Expansions}\label{sec:expansions}
In this section, we derive a combinatorial formula for the coefficients appearing in the expansion of the polynomials $\mathfrak{G}^{\buv}_{\lambda}$ in terms of inhomogeneous $q$-Whittaker polynomials. By making suitable specialisations, we then obtain explicit formulas for the expansions listed in Question~\ref{ques}.

\begin{thm}\label{thm:general_expansion} Fix a positive integer $n$, and let $\lambda$ be a partition. Then the polynomial $\GG^{\buv}_{\lambda}$ can be expanded interms of Inhomogeneous $q$-Whittaker polynomials as follows:
\begin{equation}
\GG^{\buv}_{\lambda}(x_{1},\dots,x_{n})\,=\,\sum_{\mu}\,d_{\lambda,\mu}(\mathbf{u},\mathbf{v},\mathbf{y})\,\GG^{(\mathbf{y},\mathbf{0})}_{\mu}(x_{1},\dots,x_{n})
\end{equation}  
where
\begin{equation}
\,d_{\lambda,\mu}(\mathbf{u},\mathbf{v},\mathbf{y})\,:=\,
 \tikz{1}{
\foreach \x in {1,2,3}{
\draw[lgray,line width= 4 pt,->] (0,\x-0.5)--(3,\x-0.5);
};
\foreach \x in {1,2,3}{
\draw[lgray,line width= 4 pt,->] (\x-0.5,0)--(\x-0.5,3);
};
\node at (3.7,2.5) {$\ss m_{n}(\mu')$};
\node at (3.5,1.5) {$ \vdots$};
\node at (3.7,0.5) {$\ss m_{1}(\mu')$};
\node at (0.5,3.3) {$\ss 0$};
\node at (1.5,3.3) {$\dots$};
\node at (2.5,3.3) {$\ss 0$};
\node at (0.5,-0.5) {$\ss m_{n}(\lambda')$};
\node at (1.5,-0.5) {$\dots$};
\node at (2.5,-0.5) {$\ss m_{1}(\lambda')$};
\node at (-0.3,0.5) {$\ss 0$};
\node at (-0.3,1.5) {$\ss 0$};
\node at (-0.3,2.5) {$\ss 0$};
\node[left] at (-0.7,2.5) {$\ss y_{n}^{-1} \rightarrow$};
\node[left] at (-1,1.5) {$\vdots$};
\node[left] at (-0.7,0.5) {$\ss y_{1}^{-1} \rightarrow$};
\node[below] at (2.5,-1) {$\uparrow$};
\node[below] at (2.5,-1.5) {$\ss (u_{1},v_{1})$};
\node[below] at (1.5,-1.5) {$\dots$};
\node[below] at (0.5,-1) {$\uparrow$};
\node[below] at (0.5,-1.5) {$\ss (u_{n},v_{n})$};
}
\end{equation}
using the weights:
\begin{equation}
\tikz{0.9}{
\draw[lgray,line width=4pt,->] (-1,0) -- (1,0);
\draw[lgray,line width=4pt,->] (0,-1) -- (0,1);
\node[left] at (-1,0) {\tiny $b$};\node[right] at (1,0) {\tiny $d$};
\node[below] at (0,-1) {\tiny $a$};\node[above] at (0,1) {\tiny $c$};
\node[left] at (-1.5,0) {$y \rightarrow$};
\node[below] at (0,-1.4) {$\uparrow$};
\node[below] at (0,-1.9) {$(u,v)$};
}
=\, \delta_{a + b = c + d} \,\, y^{a-d}\,
\sum_{p}\,{(uy^{-1};q)_{c-p}\,}\,{(vy;q)_{p}\,}(vy)^{b-p}\,\binom{c+d-p}{c-p}_{q}\,\binom{b}{p}_{q}
\end{equation}
\end{thm}
\begin{proof}
Our proof is again an application of the YBE. We begin by considering the right hand side of~\ref{prop:spinL_RLL_relation} but with many crosses. We then choose the boundary conditions in such a way that the partition function of this lattice will be $\GG^{(u,v)}_{\lambda}$. Then by applying the YBE, we obtain an equivalent expression which will gives us the desired expansion. Below, we present the lattice under consideration on the left-hand side and its equivalent expression, obtained via the Yang–Baxter equation (YBE), on the right-hand side.
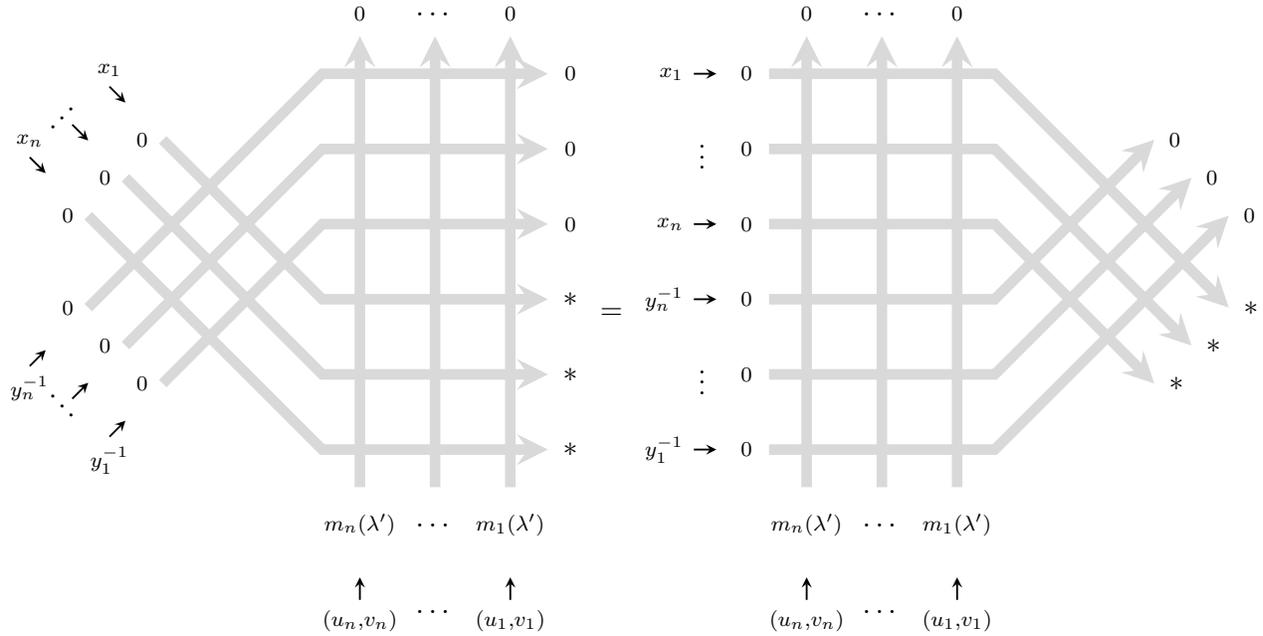
\[
 \tikz{1}{
\foreach \x in {1,2,3}{
\draw[lgray,line width= 4 pt,->] (\x-0.5,0)--(\x-0.5,6);
};
\draw[lgray,line width= 4 pt,->] (0,0.5)++(135:4.4) node[left,black]{$\ss 0$}--++(-45:4.4)--++(0:3);
\draw[lgray,line width= 4 pt,,->] (0,1.5)++(135:3.7) node[left,black]{$\ss 0$}--++(-45:3.7)--++(0:3);
\draw[lgray,line width= 4 pt,->] (0,2.5)++(135:3) node[left,black]{$\ss 0$}--++(-45:3)--++(0:3);
\draw[lgray,line width= 4 pt,->] (0,3.5)++(-135:3) node[left,black]{$\ss 0$}--++(45:3)--++(0:3);
\draw[lgray,line width= 4 pt,->] (0,4.5)++(-135:3.7) node[left,black]{$\ss 0$}--++(45:3.7)--++(0:3);
\draw[lgray,line width= 4 pt,->] (0,5.5)++(-135:4.4)node[left,black]{$\ss 0$}--++(45:4.4)--++(0:3);
\draw[line width= 0.7 pt,->] (0,0.5)++(135:5.5) node[above] {$\ss x_{n}$}--++(-45:0.3);
\draw[line width= 0.7 pt,->] (0,1.5)++(135:4.7) node[above,rotate=45] {$ \dots$}--++(-45:0.3);
\draw[line width= 0.7 pt,->] (0,2.5)++(135:4) node[above] {$\ss x_{1}$}--++(-45:0.3);
\draw[line width= 0.7 pt,->] (0,3.5)++(-135:4) node[below,black]{$\ss y_{1}^{-1}$}--++(45:0.3);
\draw[line width= 0.7 pt,->] (0,4.5)++(-135:4.7) node[below,black,rotate=-45]{$ \dots$}--++(45:0.3);
\draw[line width= 0.7 pt,->] (0,5.5)++(-135:5.5) node[below,black]{$\ss y_{n}^{-1}$}--++(45:0.3);
\node at (3.3,5.5) {$\ss 0$};
\node at (3.3,4.5) {$\ss 0$};
\node at (3.3,3.5) {$\ss 0$};
\node at (3.3,2.5) {$ *$};
\node at (3.3,1.5) {$ *$};
\node at (3.3,0.5) {$ *$};
\node at (0.5,6.3) {$\ss 0$};
\node at (1.5,6.3) {$\dots$};
\node at (2.5,6.3) {$\ss 0$};
\node at (0.5,-0.5) {$\ss m_{n}(\lambda')$};
\node at (1.5,-0.5) {$\dots$};
\node at (2.5,-0.5) {$\ss m_{1}(\lambda')$};
\draw[line width= 0.7 pt,->] (2.5,-1.5) node[below] {$\ss (u_{1},v_{1})$}--(2.5,-1.2);
\draw[line width= 0.7 pt] (1.5,-1.5) node[below] {$\dots$};
\draw[line width= 0.7 pt,->] (0.5,-1.5) node[below] {$\ss (u_{n},v_{n})$}--(0.5,-1.2);
}\,=
\,
\tikz{1}{
\foreach \x in {2,3,4}{
\draw[lgray,line width= 4 pt,->] (\x-0.5,0)--(\x-0.5,6);
};
\draw[lgray,line width= 4 pt,->] (1,0.5)--(4,0.5)--++(45:4.4)
node[right,black] {$\ss 0$};
\draw[lgray,line width= 4 pt,->] (1,1.5)--(4,1.5)--++(45:3.7)
node[right,black] {$\ss 0$};
\draw[lgray,line width= 4 pt,->] (1,2.5)--(4,2.5)--++(45:3) node[right,black] {$\ss 0$};
\draw[lgray,line width= 4 pt,->] (1,3.5)--(4,3.5)--++(-45:3) node[right,black] {$*$};
\draw[lgray,line width= 4 pt,->] (1,4.5)--(4,4.5)--++(-45:3.7) node[right,black] {$*$};
\draw[lgray,line width= 4 pt,->] (1,5.5)--(4,5.5)--++(-45:4.4) node[right,black] {$*$};
\node at (1-0.3,0.5) {$\ss 0$};
\node at (1-0.3,1.5) {$\ss 0$};
\node at (1-0.3,2.5) {$\ss 0$};
\node at (1-0.3,3.5) {$\ss 0$};
\node at (1-0.3,4.5) {$\ss 0$};
\node at (1-0.3,5.5) {$\ss 0$};
\node at (1.5,6.3) {$\ss 0$};
\node at (2.5,6.3) {$\dots$};
\node at (3.5,6.3) {$\ss 0$};
\node at (1.5,-0.5) {$\ss m_{n}(\lambda')$};
\node at (2.5,-0.5) {$\dots$};
\node at (3.5,-0.5) {$\ss m_{1}(\lambda')$};
\draw[line width= 0.7 pt,->]  (0,3.5) node[left] {$\ss x_{n}$}--(0.3,3.5);
\node[left] at (0.3,4.5) {$\vdots$};
\draw[line width= 0.7 pt,->]  (0,5.5) node[left] {$\ss x_{1}$}--(0.3,5.5);
\draw[line width= 0.7 pt,->]  (0,0.5) node[left] {$\ss y^{-1}_{1}$}--(0.3,0.5);
\node[left] at (0.3,1.5) {$\vdots$};
\draw[line width= 0.7 pt,->]  (0,2.5) node[left] {$\ss y^{-1}_{n}$}--(0.3,2.5);
\draw[line width= 0.7 pt,->] (3.5,-1.5) node[below] {$\ss (u_{1},v_{1})$}--(3.5,-1.2);
\draw[line width= 0.7 pt] (2.5,-1.5) node[below] {$\dots$};
\draw[line width= 0.7 pt,->] (1.5,-1.5) node[below] {$\ss (u_{n},v_{n})$}--(1.5,-1.2);
}
\]
On the left hand side, we have particles entering from the bottom. These particles can only take North-East paths. Therefore, then can only exit through the right boundary. We choose boundary of the right boundary such that the particles can only exit from the bottom $n$ rows of the grid. This implies that the entire lattice can be factored as a product of three partition functions as shown below:
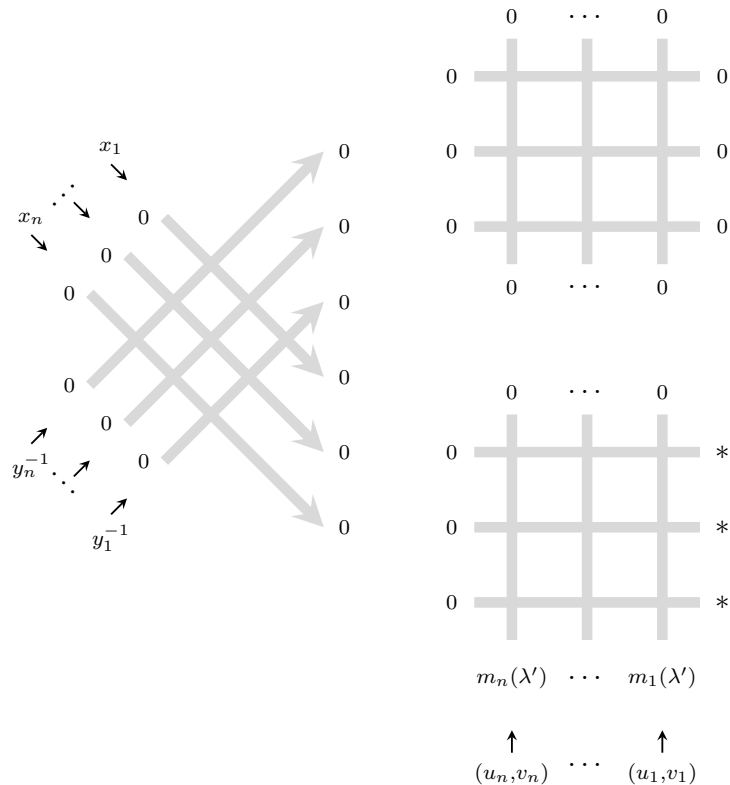
\[
 \tikz{1}{
\foreach \x in {5,6,7}{
\draw[lblack,line width= 4 pt] (2,\x-0.5)--(5,\x-0.5);
};
\foreach \x in {0,1,2}{
\draw[lblack,line width= 4 pt] (2,\x-0.5)--(5,\x-0.5);
};
\foreach \x in {3,4,5}{
\draw[lblack,line width= 4 pt] (\x-0.5,-1)--(\x-0.5,2);
\draw[lblack,line width= 4 pt] (\x-0.5,4)--(\x-0.5,7);
};
\draw[lblack,line width= 4 pt,<-] (0,0.5) node[right,black]{$\ss 0$}--++(135:4.4) node[left,black]{$\ss 0$};
\draw[lblack,line width= 4 pt,<-] (0,1.5) node[right,black]{$\ss 0$}--++(135:3.7) node[left,black]{$\ss 0$};
\draw[lblack,line width= 4 pt,<-] (0,2.5) node[right,black]{$\ss 0$}--++(135:3) node[left,black]{$\ss 0$};
\draw[lblack,line width= 4 pt,<-] (0,3.5) node[right,black]{$\ss 0$}--++(-135:3) node[left,black]{$\ss 0$};
\draw[lblack,line width= 4 pt,<-] (0,4.5) node[right,black]{$\ss 0$}--++(-135:3.7) node[left,black]{$\ss 0$};
\draw[lblack,line width= 4 pt,<-] (0,5.5) node[right,black]{$\ss 0$}--++(-135:4.4)node[left,black]{$\ss 0$};
\draw[line width= 0.7 pt,->] (0,0.5)++(135:5.5) node[above] {$\ss x_{n}$}--++(-45:0.3);
\draw[line width= 0.7 pt,->] (0,1.5)++(135:4.7) node[above,rotate=45] {$ \dots$}--++(-45:0.3);
\draw[line width= 0.7 pt,->] (0,2.5)++(135:4) node[above] {$\ss x_{1}$}--++(-45:0.3);
\draw[line width= 0.7 pt,->] (0,3.5)++(-135:4) node[below,black]{$\ss y_{1}^{-1}$}--++(45:0.3);
\draw[line width= 0.7 pt,->] (0,4.5)++(-135:4.7) node[below,black,rotate=-45]{$ \dots$}--++(45:0.3);
\draw[line width= 0.7 pt,->] (0,5.5)++(-135:5.5) node[below,black]{$\ss y_{n}^{-1}$}--++(45:0.3);
\node at (2.5,7.3) {$\ss 0$};
\node at (3.5,7.3) {$\dots$};
\node at (4.5,7.3) {$\ss 0$};
\node at (5.3,6.5) {$\ss 0$};
\node at (5.3,5.5) {$\ss 0$};
\node at (5.3,4.5) {$\ss 0$};
\node at (1.7,6.5) {$\ss 0$};
\node at (1.7,5.5) {$\ss 0$};
\node at (1.7,4.5) {$\ss 0$};
\node at (2.5,3.7) {$\ss 0$};
\node at (3.5,3.7) {$\dots$};
\node at (4.5,3.7) {$\ss 0$};
\node at (5.3,1.5) {$ *$};
\node at (5.3,0.5) {$ *$};
\node at (5.3,-0.5) {$ *$};
\node at (1.7,1.5) {$\ss 0$};
\node at (1.7,0.5) {$\ss 0$};
\node at (1.7,-0.5) {$\ss 0$};
\node at (2.5,2.3) {$\ss 0$};
\node at (3.5,2.3) {$\dots$};
\node at (4.5,2.3) {$\ss 0$};
\node at (2.5,-1.5) {$\ss m_{n}(\lambda')$};
\node at (3.5,-1.5) {$\dots$};
\node at (4.5,-1.5) {$\ss m_{1}(\lambda')$};
\draw[line width= 0.7 pt,->] (4.5,-2.5) node[below] {$\ss (u_{1},v_{1})$}--(4.5,-2.2);
\draw[line width= 0.7 pt] (3.5,-2.5) node[below] {$\dots$};
\draw[line width= 0.7 pt,->] (2.5,-2.5) node[below] {$\ss (u_{n},v_{n})$}--(2.5,-2.2);
}
\]
We can write the right hand side as the following summation:

\[
\sum_{a_{i},b_{i}\,\in\,\mathbb{Z}_{\geq0}}\,\,
\tikz{1}{
\foreach \x in {1,2,3,4,5,6}{
\draw[lblack,line width= 4 pt] (1,\x-0.5)--(4,\x-0.5);
};
\foreach \x in {2,3,4}{
\draw[lblack,line width= 4 pt] (\x-0.5,0)--(\x-0.5,6);
};
\draw[lblack,line width= 4 pt,->] (6,0.5) node[left,black]{$\ss b_{n}$}--++(45:4.4)
node[right,black] {$\ss 0$};
\draw[lblack,line width= 4 pt,->] (6,1.5) node[left,black]{$\vdots$}--++(45:3.7)
node[right,black] {$\ss 0$};
\draw[lblack,line width= 4 pt,->] (6,2.5) node[left,black]{$\ss b_{1}$}--++(45:3) node[right,black] {$\ss 0$};
\draw[lblack,line width= 4 pt,->] (6,3.5) node[left,black]{$\ss a_{n}$}--++(-45:3) node[right,black] {$*$};
\draw[lblack,line width= 4 pt,->] (6,4.5) node[left,black]{$\vdots$}--++(-45:3.7) node[right,black] {$*$};
\draw[lblack,line width= 4 pt,->] (6,5.5) node[left,black]{$\ss a_{1}$}--++(-45:4.4) node[right,black] {$*$};
\node at (1-0.3,0.5) {$\ss 0$};
\node at (1-0.3,1.5) {$\ss 0$};
\node at (1-0.3,2.5) {$\ss 0$};
\node at (1-0.3,3.5) {$\ss 0$};
\node at (1-0.3,4.5) {$\ss 0$};
\node at (1-0.3,5.5) {$\ss 0$};
\node at (4.3,0.5) {$\ss b_{n}$};
\node at (4.3,1.5) {$\ss \vdots$};
\node at (4.3,2.5) {$\ss b_{1}$};
\node at (4.3,3.5) {$\ss a_{n}$};
\node at (4.3,4.5) {$\ss \vdots$};
\node at (4.3,5.5) {$\ss a_{1}$};
\node at (1.5,6.3) {$\ss 0$};
\node at (2.5,6.3) {$\dots$};
\node at (3.5,6.3) {$\ss 0$};
\node at (1.5,-0.5) {$\ss m_{n}(\lambda')$};
\node at (2.5,-0.5) {$\dots$};
\node at (3.5,-0.5) {$\ss m_{1}(\lambda')$};
\draw[line width= 0.7 pt,->]  (0,3.5) node[left] {$\ss x_{n}$}--(0.3,3.5);
\node[left] at (0.3,4.5) {$\vdots$};
\draw[line width= 0.7 pt,->]  (0,5.5) node[left] {$\ss x_{1}$}--(0.3,5.5);
\draw[line width= 0.7 pt,->]  (0,0.5) node[left] {$\ss y^{-1}_{1}$}--(0.3,0.5);
\node[left] at (0.3,1.5) {$\vdots$};
\draw[line width= 0.7 pt,->]  (0,2.5) node[left] {$\ss y^{-1}_{n}$}--(0.3,2.5);
\draw[line width= 0.7 pt,->] (3.5,-1.5) node[below] {$\ss (u_{1},v_{1})$}--(3.5,-1.2);
\draw[line width= 0.7 pt] (2.5,-1.5) node[below] {$\dots$};
\draw[line width= 0.7 pt,->] (1.5,-1.5) node[below] {$\ss (u_{n},v_{n})$}--(1.5,-1.2);
}
\]

 We claim that all $a_{i}=0$ for all $1\leq i\leq n$ in other words, all the particles that are entering from the bottom must exit through the bottom $n$ rows. To conclude this we observe the possible configurations in the ``crosses'':

\begin{equation}
 \tikz{1}{
\foreach \x in {1,2,3}{
\draw[lgray,line width= 4 pt,->] (0,\x-0.5)--(3,\x-0.5);
};
\foreach \x in {1,2,3}{
\draw[lgray,line width= 4 pt,->] (\x-0.5,0)--(\x-0.5,3);
};
\node at (3.3,2.5) {$\ss *$};
\node at (3.3,1.5) {$ \ss *$};
\node at (3.3,0.5) {$\ss *$};
\node at (0.5,3.3) {$\ss 0$};
\node at (1.5,3.3) {$\ss 0$};
\node at (2.5,3.3) {$\ss 0$};
\node at (0.5,-0.5) {$\ss b_{1}$};
\node at (1.5,-0.5) {$\dots$};
\node at (2.5,-0.5) {$\ss b_{n}$};
\node at (-0.3,0.5) {$\ss a_{n}$};
\node at (-0.3,1.5) {$\ss \vdots$};
\node at (-0.3,2.5) {$\ss a_{1}$};
\node[left] at (-0.7,2.5) {$\ss x_{1} \rightarrow$};
\node[left] at (-1,1.5) {$\vdots$};
\node[left] at (-0.7,0.5) {$\ss x_{n} \rightarrow$};
\node[below] at (2.5,-1) {$\uparrow$};
\node[below] at (2.5,-1.5) {$\ss y_{1}^{-1}$};
\node[below] at (1.5,-1.5) {$\dots$};
\node[below] at (0.5,-1) {$\uparrow$};
\node[below] at (0.5,-1.5) {$\ss y^{-1}_{n}$};
}
\end{equation}
using the weights:
\begin{equation}
\tikz{0.9}{
\draw[lgray,line width=4pt,->] (-1,0) -- (1,0);
\draw[lgray,line width=4pt,->] (0,-1) -- (0,1);
\node[left] at (-1,0) {\tiny $b$};\node[right] at (1,0) {\tiny $d$};
\node[below] at (0,-1) {\tiny $a$};\node[above] at (0,1) {\tiny $c$};
\node[left] at (-1.5,0) {${\ss x} \rightarrow$};
\node[below] at (0,-1.4) {$\uparrow$};
\node[below] at (0,-1.9) {$\ss y^{-1}$};
}
=\, \delta_{a + b = c + d} \,\, x^{d}\,y^{d}\,(xy;q)_{c-b}\,\binom{a}{c-b}_{q}\,
\end{equation}

Observe that the above weights vanish identically whenever \( c < b \). This enforces the condition that if the top edge of a vertex is labelled zero, then its left edge must also be zero. Since all the top boundary labels are zero, this constraint implies that the left edges of all vertices in the first row must also be zero, as illustrated in the figure below (on the left).

\begin{equation} 
 \tikz{1}{
\foreach \x in {1,2,3}{
\draw[lgray,line width= 4 pt,->] (0,\x-0.5)--(3,\x-0.5);
};
\foreach \x in {1,2,3}{
\draw[lgray,line width= 4 pt,->] (\x-0.5,0)--(\x-0.5,3);
};
\node at (3.3,2.5) {$\ss *$};
\node at (3.3,1.5) {$ \ss *$};
\node at (3.3,0.5) {$\ss *$};
\node at (0.5,3.3) {$\ss 0$};
\node at (1.5,3.3) {$\ss 0$};
\node at (2.5,3.3) {$\ss 0$};
\node at (0.5,-0.5) {$\ss b_{1}$};
\node at (1.5,-0.5) {$\dots$};
\node at (2.5,-0.5) {$\ss b_{n}$};
\node at (-0.3,0.5) {$\ss a_{n}$};
\node at (-0.3,1.5) {$\ss \vdots$};
\node at (-0.3,2.5) {$\ss 0$};
\node at (1,2.5) {$\ss 0$};
\node at (2,2.5) {$\ss 0$};
}
\hspace{2cm}
 \tikz{1}{
\foreach \x in {1,2,3}{
\draw[lgray,line width= 4 pt,->] (0,\x-0.5)--(3,\x-0.5);
};
\foreach \x in {1,2,3}{
\draw[lgray,line width= 4 pt,->] (\x-0.5,0)--(\x-0.5,3);
};
\node at (3.3,2.5) {$\ss *$};
\node at (3.3,1.5) {$ \ss *$};
\node at (3.3,0.5) {$\ss *$};
\node at (0.5,3.3) {$\ss 0$};
\node at (1.5,3.3) {$\ss 0$};
\node at (2.5,3.3) {$\ss 0$};
\node at (0.5,-0.5) {$\ss b_{1}$};
\node at (1.5,-0.5) {$\dots$};
\node at (2.5,-0.5) {$\ss b_{n}$};
\node at (-0.3,0.5) {$\ss 0$};
\node at (-0.3,1.5) {$\ss 0$};
\node at (-0.3,2.5) {$\ss 0$};
\node at (1,2.5) {$\ss 0$};
\node at (2,2.5) {$\ss 0$};
\node at (0.5,2) {$\ss 0$};
\node at (1.5,2) {$\ss 0$};
\node at (1,1.5) {$\ss 0$};
\node at (0.5,1) {$\ss 0$};
}
\end{equation}

This reasoning implies that in the second row, the top edges of the vertices in the first \( n-1 \) columns from the left are again labelled zero. Applying the same argument to this row, we deduce that \( a_2 = 0 \). Iterating this process row by row, we conclude that \( a_i = 0 \) for all \( i \in \{1, \dots, n\} \), as illustrated above (on the right).We can then conclude that the right hand side factors as following summation:
 
\[
\sum_{\mu}
\tikz{1}{
\foreach \x in {5,6,7}{
\draw[lblack,line width= 4 pt] (0,\x-0.5) node[left,black]{$\ss 0$}--(3,\x-0.5) node[right,black]{$\ss 0$};
};
\foreach \x in {1,2,3}{
\draw[lblack,line width= 4 pt] (\x-0.5,4) node[below,black]{$\ss 0$}--(\x-0.5,7) node[above,black]{$\ss 0$};
};
\foreach \x in {0,1,2}{
\draw[lblack,line width= 4 pt] (0,\x-0.5) node[left,black]{$\ss 0$} --(3,\x-0.5);
};
\foreach \x in {1,2,3}{
\draw[lblack,line width= 4 pt] (\x-0.5,-1)--(\x-0.5,2) node[above,black]{$\ss 0$};
};
\draw[lblack,line width= 4 pt,->] (6,0.5)node[left,black] {$\ss m_{1}(\mu')$} --++(45:4.4)
node[right,black] {$\ss 0$};
\draw[lblack,line width= 4 pt,->] (6,1.5) node[left,black] {$\ss \vdots $}--++(45:3.7)
node[right,black] {$\ss 0$};
\draw[lblack,line width= 4 pt,->] (6,2.5) node[left,black] {$\ss m_{n}(\mu')$}--++(45:3) node[right,black] {$\ss 0$};
\draw[lblack,line width= 4 pt,->] (6,3.5) node[left,black] {$\ss 0$}--++(-45:3) node[right,black] {$*$};
\draw[lblack,line width= 4 pt,->] (6,4.5)node[left,black] {$\ss 0$}--++(-45:3.7) node[right,black] {$*$};
\draw[lblack,line width= 4 pt,->] (6,5.5) node[left,black] {$\ss 0$}--++(-45:4.4) node[right,black] {$*$};
\node at (3.5,-0.5) {$\ss m_{1}(\mu')$};
\node at (3.5,0.5) {$\vdots$};
\node at (3.5,1.5) {$\ss m_{n}(\mu')$};
\node at (0.5,-1.5) {$\ss m_{n}(\lambda')$};
\node at (1.5,-1.5) {$\dots$};
\node at (2.5,-1.5) {$\ss m_{1}(\lambda')$};
\draw[line width= 0.7 pt,->]  (-1,4.5) node[left] {$\ss x_{n}$}--(-0.7,4.5);
\node[left] at (-0.7,5.5) {$\vdots$};
\draw[line width= 0.7 pt,->]  (-1,6.5) node[left] {$\ss x_{1}$}--(-0.7,6.5);
\draw[line width= 0.7 pt,->]  (-1,-0.5) node[left] {$\ss y^{-1}_{1}$}--(-0.7,-0.5);
\node[left] at (-0.7,0.5) {$\vdots$};
\draw[line width= 0.7 pt,->]  (-1,1.5) node[left] {$\ss y^{-1}_{n}$}--(-0.7,1.5);
\draw[line width= 0.7 pt,->] (2.5,-2.5) node[below] {$\ss (u_{1},v_{1})$}--(2.5,-2.2);
\draw[line width= 0.7 pt] (1.5,-2.5) node[below] {$\dots$};
\draw[line width= 0.7 pt,->] (0.5,-2.5) node[below] {$\ss (u_{n},v_{n})$}--(0.5,-2.2);
}
\]

We have one final point to address. Observe that the weights used in the definition of $\GG^{(\mathbf{y},\mathbf{0})}_{\lambda}$ differ from the weights of the ``crosses'' by a factor of $y^{d}$. This leads to the following expansion:
\[
\GG^{\buv}_{\lambda}(x_{1},\dots,x_{n})\,=\,\sum_{\mu}\,e^{\lambda}_{\mu}(\mathbf{u},\mathbf{v},\mathbf{y})\,\left(\prod^{n}_{i=1}\,y^{\sum^{n}_{j=i}m^{c}_{j}(\mu)}_{i}\right)\,\GG^{(\mathbf{y},\mathbf{0})}_{\mu}(x_{1},\dots,x_{n})
\]

\begin{equation}
e^{\lambda}_{\mu}(\mathbf{u},\mathbf{v},\mathbf{y})\,:=\,
 \tikz{1}{
\foreach \x in {1,2,3}{
\draw[lgray,line width= 4 pt,->] (0,\x-0.5)--(3,\x-0.5);
};
\foreach \x in {1,2,3}{
\draw[lgray,line width= 4 pt,->] (\x-0.5,0)--(\x-0.5,3);
};
\node at (3.7,2.5) {$\ss m_{n}(\mu')$};
\node at (3.5,1.5) {$ \vdots$};
\node at (3.7,0.5) {$\ss m_{1}(\mu')$};
\node at (0.5,3.3) {$\ss 0$};
\node at (1.5,3.3) {$\dots$};
\node at (2.5,3.3) {$\ss 0$};
\node at (0.5,-0.5) {$\ss m_{n}(\lambda')$};
\node at (1.5,-0.5) {$\dots$};
\node at (2.5,-0.5) {$\ss m_{1}(\lambda')$};
\node at (-0.3,0.5) {$\ss 0$};
\node at (-0.3,1.5) {$\ss 0$};
\node at (-0.3,2.5) {$\ss 0$};
\node[left] at (-0.7,2.5) {$\ss y_{n}^{-1} \rightarrow$};
\node[left] at (-1,1.5) {$\vdots$};
\node[left] at (-0.7,0.5) {$\ss y_{1}^{-1} \rightarrow$};
\node[below] at (2.5,-1) {$\uparrow$};
\node[below] at (2.5,-1.5) {$\ss (u_{1},v_{1})$};
\node[below] at (1.5,-1.5) {$\dots$};
\node[below] at (0.5,-1) {$\uparrow$};
\node[below] at (0.5,-1.5) {$\ss (u_{n},v_{n})$};
}
\end{equation}
using the weights:
\begin{equation}
\tikz{0.9}{
\draw[lgray,line width=4pt,->] (-1,0) -- (1,0);
\draw[lgray,line width=4pt,->] (0,-1) -- (0,1);
\node[left] at (-1,0) {\tiny $b$};\node[right] at (1,0) {\tiny $d$};
\node[below] at (0,-1) {\tiny $a$};\node[above] at (0,1) {\tiny $c$};
\node[left] at (-1.5,0) {$y \rightarrow$};
\node[below] at (0,-1.4) {$\uparrow$};
\node[below] at (0,-1.9) {$(u,v)$};
}
=\,  \mathbf{1}_{a + b = c + d}\, \,\, y^{-d}\,
\sum_{p}\,{(uy^{-1};q)_{c-p}\,}\,{(vy;q)_{p}\,}(vy)^{b-p}\,\binom{c+d-p}{c-p}_{q}\,\binom{b}{p}_{q}
\end{equation}
In the final step, we integrate the factor $\left(\prod^{n}_{i=1}\,y^{\sum^{n}_{j=i}m_{j}(\mu')}_{i}\right)$ into the weights by multiplying $y^{a}$ where $a$ is bottom label of the vertex. This concludes our proof.
\end{proof}

\subsection{Degenerations}\label{subsec:Expansions_Degenerations}
\subsubsection{Expansion of Whittaker polynomials into Inhomogeneous $q$-Whittaker polynomials}
We specialise the above equation to the following $u_{i}=v_{i}=0$ and $y_{i}=1$ for all $i\in\mathbb{Z}_{> 0}$.

\begin{cor} Fix a positive integer $n$, and let $\lambda$ be a partition. Then $q$-Whittaker polynomials can be expanded in terms of Inhomogeneous $q$-Whittaker polynomials as follows:
\begin{equation}
W_{\lambda}(x_{1},\dots,x_{n})\,=\,\sum_{\mu}\,a_{\lambda,\mu}(q)\,\mathbb{F}_{\mu}(x_{1},\dots,x_{n})
\end{equation}  
where
\begin{equation}
a_{\lambda,\mu}(q)\,:=\,
 \tikz{1}{
\foreach \x in {1,2,3}{
\draw[lgray,line width= 4 pt,->] (0,\x-0.5)--(3,\x-0.5);
};
\foreach \x in {1,2,3}{
\draw[lgray,line width= 4 pt,->] (\x-0.5,0)--(\x-0.5,3);
};
\node at (3.7,2.5) {$\ss m_{n}(\mu')$};
\node at (3.5,1.5) {$ \vdots$};
\node at (3.7,0.5) {$\ss m_{1}(\mu')$};
\node at (0.5,3.3) {$\ss 0$};
\node at (1.5,3.3) {$\dots$};
\node at (2.5,3.3) {$\ss 0$};
\node at (0.5,-0.5) {$\ss m_{n}(\lambda')$};
\node at (1.5,-0.5) {$\dots$};
\node at (2.5,-0.5) {$\ss m_{1}(\lambda')$};
\node at (-0.3,0.5) {$\ss 0$};
\node at (-0.3,1.5) {$\ss 0$};
\node at (-0.3,2.5) {$\ss 0$};
}
\end{equation}
using the weights:
\begin{equation}
\tikz{0.9}{
\draw[lgray,line width=4pt,->] (-1,0) -- (1,0);
\draw[lgray,line width=4pt,->] (0,-1) -- (0,1);
\node[left] at (-1,0) {\tiny $b$};\node[right] at (1,0) {\tiny $d$};
\node[below] at (0,-1) {\tiny $a$};\node[above] at (0,1) {\tiny $c$};
}
=\,  \mathbf{1}_{a + b = c + d}\,\mathbf{1}_{ b \leq c}\,\binom{a}{c-b}_{q}\,
\end{equation}

Then \(a_{\lambda,\mu}(q)\in\mathbb{N}[q]\); equivalently, the \(q\)-Whittaker polynomials \(W_{\lambda}\) admit a positive expansion in terms of the inhomogeneous \(q\)-Whittaker polynomials \(\mathbb{F}_{\mu}\).
\end{cor}

This result generalises Lenart's formula~\cite[Theorem 2.7]{Lenart2000} for the expansion of Schur polynomials in terms of Grothendieck polynomials.

\begin{ex}
We illustrate the expansion of a $q$-Whittaker polynomial into the basis of inhomogeneous $q$-Whittaker polynomials $\mathbb{F}_\lambda$. 
Consider the partition $\lambda=(2,0)$ and two variables $x_1, x_2$. 
The $q$-Whittaker polynomial is given by
\[
W_{(2,0)}
  = x_{1}^{2} + (1+q)\,x_{1}x_{2} + x_{2}^{2}.
\]
The inhomogeneous counterparts are computed as:
\begin{align*}
\mathbb{F}_{(2,0)}
  &= x_{1}^{2}
     + (1+q)x_{1}x_{2}+x_{2}^{2}-(1+q)(x^{2}_{1}x_{2}+x_{1}x^{2}_{2})+q x_{1}^{2}x_{2}^{2}, \\
\mathbb{F}_{(2,1)}
  &= x_{1}^{2}x_{2}
     + x_{1}x_{2}^{2}-x^{2}_{1}x_{2}^{2}, \\
\mathbb{F}_{(2,2)}
  &= x_{1}^{2}x_{2}^{2}.
\end{align*}
Consequently, the expansion reads
\[
W_{(2,0)}
   = \mathbb{F}_{(2,0)}
     + (1+q)\,\mathbb{F}_{(2,1)}
     + \mathbb{F}_{(2,2)}.
\]
\end{ex}

\subsubsection{Expansion of Inhomogeneous $q$-Whittaker polynomials into $q$-Whittaker polynomials}
We specialise the above equation to the following $v_{i}=0$ and $y_{i}=0$ for all $i\in \mathbb{Z}_{>0}$.

\begin{cor} Fix a positive integer $n$, and let $\lambda$ be a partition. Then $q$-Whittaker polynomials can be expanded in terms of Inhomogeneous $q$-Whittaker polynomials as follows:
\begin{equation}
\mathbb{F}_{\lambda}(x_{1},\dots,x_{n})\,=\,\sum_{\mu}\,b_{\lambda,\mu}(q)\,{W}_{\mu}(x_{1},\dots,x_{n})
\end{equation}  
where
\begin{equation}
b_{\lambda,\mu}(q)\,:=\,
 \tikz{1}{
\foreach \x in {1,2,3}{
\draw[lgray,line width= 4 pt,->] (0,\x-0.5)--(3,\x-0.5);
};
\foreach \x in {1,2,3}{
\draw[lgray,line width= 4 pt,->] (\x-0.5,0)--(\x-0.5,3);
};
\node at (3.7,2.5) {$\ss m_{n}(\mu')$};
\node at (3.5,1.5) {$ \vdots$};
\node at (3.7,0.5) {$\ss m_{1}(\mu')$};
\node at (0.5,3.3) {$\ss 0$};
\node at (1.5,3.3) {$\dots$};
\node at (2.5,3.3) {$\ss 0$};
\node at (0.5,-0.5) {$\ss m_{n}(\lambda')$};
\node at (1.5,-0.5) {$\dots$};
\node at (2.5,-0.5) {$\ss m_{1}(\lambda')$};
\node at (-0.3,0.5) {$\ss 0$};
\node at (-0.3,1.5) {$\ss 0$};
\node at (-0.3,2.5) {$\ss 0$};
}
\end{equation}
using the weights:
\begin{equation}
\tikz{0.9}{
\draw[lgray,line width=4pt,->] (-1,0) -- (1,0);
\draw[lgray,line width=4pt,->] (0,-1) -- (0,1);
\node[left] at (-1,0) {\tiny $b$};\node[right] at (1,0) {\tiny $d$};
\node[below] at (0,-1) {\tiny $a$};\node[above] at (0,1) {\tiny $c$};
}
=\, \mathbf{1}_{a + b = c + d}\,\mathbf{1}_{ b \leq c}\,(-1)^{c-b}q^{\binom{c-b}{2}}\,\binom{a}{c-b}_{q}.\,
\end{equation}
Furthermore, we can conclude that $(-1)^{|\lambda|-|\mu|}b_{\lambda,\mu}(q)\in \mathbb{N}[q]$.
\end{cor}
\begin{proof}
 The result immediately follows by setting $v_{i}=y_{i}=0$ and $u_{i}=1$ for all $i\in\mathbb{Z}_{>0}$ in Theorem~\ref{thm:general_expansion}. 

 \medskip
We are therefore left to prove that 
\[
(-1)^{|\lambda|-|\mu|}\,b_{\lambda,\mu}(q)\in \mathbb{N}[q].
\]
Observe that the only source of a negative sign is the factor $(-1)^{\,c-b}$ appearing in the Boltzmann weights.  
Thus, it suffices to compute the statistic $c-b$ for every vertex in a given configuration.

We begin with the statistic obtained by summing the top labels of all vertices.  
Consider the bottom row: in this row, $n$ particles enter from the bottom and $m_{1}(\mu')$ exit through the right boundary.  
Hence the sum of the top labels over all vertices in the bottom row is
\[
n - m_{1}(\mu').
\]
By the same argument, the sum of the top labels in the $i$-th row is
\[
n - \sum_{j=1}^{i} m_{j}(\mu').
\]
Therefore, the total sum of the top labels over all rows is
\[
n^{2}-\sum_{i=1}^{n}(n+1-i)
   = -n+\sum_{i=1}^{n} i\,m_{i}(\mu')
   = -n + |\mu|,
\]
where we used the identity \(\sum_{i=1}^{n} m_{i}(\mu') = n\).

By repeating the same argument, we find that the statistic obtained by summing the labels on the left edges of all vertices is equal to $-n + |\lambda|$.  
We may therefore conclude that the overall sign is
\[
(-n + |\mu|)\;-\;(-n + |\lambda|)\;=\;|\mu| - |\lambda|.
\]

\end{proof}
\begin{ex}
We illustrate the expansion of a inhomogeneous $q$-Whittaker polynomial $\mathbb{F}_{\lambda}$ into the basis of $q$-Whittaker polynomials $W_\lambda$. 
Consider the partition $\lambda=(2,0)$ and two variables $x_1, x_2$. 
The inhomogeneous $q$-Whittaker polynomial is given by
\[
\mathbb{F}_{(2,0)}
  = x_{1}^{2}
     + (1+q)x_{1}x_{2}(1-x_{1})
     + x_{2}^{2}(1-x_{1})(1-qx_{1}).
\]
The $q$-Whittaker polynomials are computed as:
\begin{align*}
W_{(2,0)}
  &= x_{1}^{2}
     + (1+q)x_{1}x_{2}
     + x_{2}^{2}, \\
W_{(2,1)}
  &= x_{1}^{2}x_{2}
     + x_{1}x_{2}^{2}, \\
W_{(2,2)}
  &= x_{1}^{2}x_{2}^{2}.
\end{align*}
Consequently, the expansion reads
\[
\mathbb{F}_{(2,0)}
   = W_{(2,0)}-(1+q)\,W_{(2,1)}
     +q W_{(2,2)}.
\]
\end{ex}

\subsubsection{Expansion of dual Inhomogeneous $q$-Whittaker polynomials into $q$-Whittaker polynomials}
We specialise the above equation to the following $\mathbf{u}=\mathbf{0}$ and $v_{i}=1$ and $y_{i}=0$ for all $i\in\mathbb{Z}_{>0}$.

\begin{cor} Fix a positive integer $n$, and let $\lambda$ be a partition. Then dual inhomogeneous $q$-Whittaker polynomials expand positively in terms of $q$-Whittaker polynomials as follows:
\begin{equation}
\mathbb{G}_{\lambda}(x_{1},\dots,x_{n})\,=\,\sum_{\mu}\,c_{\lambda,\mu}(q)\,{W}_{\mu}(x_{1},\dots,x_{n})
\end{equation}  
where
\begin{equation}
c_{\lambda,\mu}(q)\,:=\,
 \tikz{1}{
\foreach \x in {1,2,3}{
\draw[lgray,line width= 4 pt,->] (0,\x-0.5)--(3,\x-0.5);
};
\foreach \x in {1,2,3}{
\draw[lgray,line width= 4 pt,->] (\x-0.5,0)--(\x-0.5,3);
};
\node at (3.7,2.5) {$\ss m_{n}(\mu')$};
\node at (3.5,1.5) {$ \vdots$};
\node at (3.7,0.5) {$\ss m_{1}(\mu')$};
\node at (0.5,3.3) {$\ss 0$};
\node at (1.5,3.3) {$\dots$};
\node at (2.5,3.3) {$\ss 0$};
\node at (0.5,-0.5) {$\ss m_{n}(\lambda')$};
\node at (1.5,-0.5) {$\dots$};
\node at (2.5,-0.5) {$\ss m_{1}(\lambda')$};
\node at (-0.3,0.5) {$\ss 0$};
\node at (-0.3,1.5) {$\ss 0$};
\node at (-0.3,2.5) {$\ss 0$};
}
\end{equation}
using the weights:
\begin{equation}
\tikz{0.9}{
\draw[lgray,line width=4pt,->] (-1,0) -- (1,0);
\draw[lgray,line width=4pt,->] (0,-1) -- (0,1);
\node[left] at (-1,0) {\tiny $b$};\node[right] at (1,0) {\tiny $d$};
\node[below] at (0,-1) {\tiny $a$};\node[above] at (0,1) {\tiny $c$};
}
=\, \mathbf{1}_{a + b = c + d}\,\mathbf{1}_{ c \leq b}\,\binom{b}{c}_{q}\,
\end{equation}
\end{cor}

\begin{ex}
We illustrate the expansion of a dual inhomogeneous $q$-Whittaker polynomial $\mathbb{G}_{\lambda}$ into the basis of $q$-Whittaker polynomials $W_\lambda$. 
Consider the partition $\lambda=(2,2)$ and two variables $x_1, x_2$. 
The dual inhomogeneous $q$-Whittaker polynomial is given by
\[
\mathbb{G}_{2,2}\,=\, x_1^2 x_2^2 +(1+q) x_1^2 x_2 +(1+q)x_1 x_2^2 +x_1^2+(1+q) x_1 x_2 +x_2^2
\]

Then the expansion reads
\[
\mathbb{G}_{(2,2)}
   = W_{(2,2)}+(1+q)\,W_{(2,1)}
     + W_{(2,0)}.
\]
\end{ex}

\appendix
\section{Proof of prop~\ref{prop:RLL_spinL_vs_SpinL_}}\label{sec:appendix-proof}
\subsection{Fusion procedure}
We recall the fusion procedure in integrable lattice models~\cite{fusion1,fusion2}, 
following the conventions of~\cite[Section~4]{spin-BW2021}. 
For our purposes, this procedure describes how to obtain higher-spin weights 
\( W_{\l,\m} \) from the fundamental weights \( W_{1,1} \). 
In what follows, we illustrate the procedure in a special case, 
showing explicitly how the weights \( W_{\l,\m} \) 
from~\ref{theorem:weightsoftheYBE} are related to the weights 
\( W_{1,\m} \) in~\ref{weights:spin1_row_weights}:

\begin{multline}
W_{\l,\m}\left(\dfrac{x}{y};q;i,j,k,\ell\right)=\\
\tikz{0.7}{
\draw[lgray,line width=4pt,->] (-1,0) -- (1,0);
\draw[lgray,line width=4pt,->] (0,-1) -- (0,1);
\node[left] at (-1,0) {\tiny $j$};\node[right] at (1,0) {\tiny $\ell$};
\node[below] at (0,-1) {\tiny $i$};\node[above] at (0,1) {\tiny $k$};
\draw[->](-2.5,0) node[left] {\tiny $(x,\l)$}--(-2,0);
\draw[->](0,-2.5) node[below] {\tiny $(y,\m)$}--(0,-2);
}\,
=\sum_{\substack{0 \leq a_{1},\dots,a_{\l} \leq 1:|a| = j\leq \l \\[4pt]
0 \leq b_{1},\dots,b_{\l} \leq 1: |b| = \ell\leq \l}}\,\dfrac{q^{\sum^{\l}_{m=1} (m-1)a_{m}}}{Z_{j}(\l)}\times
\tikz{0.7}{
\draw[lgray,line width=1pt,->] (-1,0) -- (1,0);
\draw[lgray,line width=1pt,->] (-1,1) -- (1,1);
\draw[lgray,line width=1pt,->] (-1,2) -- (1,2);
\draw[lgray,line width=4pt,->] (0,-1) -- (0,3);
\node[left] at (-1,0) {\tiny $a_{1}$};\node[right] at (1,0) {\tiny $b_{1}$};
\node[left] at (-1,1) {\tiny $\vdots$};\node[right] at (1,1) {\tiny $\vdots$};
\node[left] at (-1,2) {\tiny $a_{\l}$};\node[right] at (1,2) {\tiny $b_{\l}$};
\node[below] at (0,-1) {\tiny $i$};\node[above] at (0,3) {\tiny $k$};
\draw[->](-2.5,0) node[left] {\tiny $x$}--(-2,0);
\draw[->](-2.5,1) node[left] {\tiny $ xq^{i-1}$}--(-2,1);
\draw[->](-2.5,2) node[left] {\tiny $xq^{\l-1}$}--(-2,2);
\draw[->](0,-2.5) node[below] {\tiny $(y,\m)$}--(0,-2);
}
\end{multline}
where:
\[
Z_{j}(\l)=\sum_{0\leq c_{1},\dots,c_{\l}\leq 1:|c|=j}\,q^{\sum^{\l}_{m=1}(m-1)c_{m}}=q^{j(j-1)/2}\,\dfrac{(q;q)_{\l}}{(q;q)_{j}(q;q)_{\l-j}}
\]

Furthermore, we have the following theorem which allows one to compute the partition function of single row lattice with the weights $W_{\l,\m}$ as follows:
\begin{thm}[]
   For any $N\geq 1$. $0\leq j,\ell \leq \l$ and $\{i_{0},\dots,i_{N}\}$, $\{k_{0},\dots,k_{N}\}$ such that each $i_{n}$ and $k_{n}$ $\in \{0,\dots,\m\}$ for all $0\leq n\leq N$, one has the following equality: 
\begin{multline}
\tikz{0.7}{
\draw[lgray,line width=4pt,->] (-1,0) -- (3,0);
\draw[lgray,line width=4pt,->] (0,-1) -- (0,1);
\draw[lgray,line width=4pt,->] (1,-1) -- (1,1);
\draw[lgray,line width=4pt,->] (2,-1) -- (2,1);
\node[left] at (-1,0) {\tiny $j$};
\node[right] at (3,0) {\tiny $\ell$};
\node[below] at (0,-1) {\tiny $i_{0}$};
\node[below] at (1,-1) {\tiny \dots};
\node[below] at (2,-1) {\tiny $i_{N}$};
\node[above] at (0,1) {\tiny $k_{0}$};
\node[above] at (1,1) {\tiny \dots};
\node[above] at (2,1) {\tiny $k_{N}$};
\draw[->](-2.5,0) node[left] {\tiny $(x,\l)$}--(-2,0);
\draw[->](0,-2.5) node[below] {\tiny $(y_{0},\m)$}--(0,-2);
\draw[->](2,-2.5) node[below] {\tiny $(y_{N},\m)$}--(2,-2);
}\,\\
=\sum_{\substack{0 \leq a_{1},\dots,a_{\l} \leq 1:|a| = j\leq \l \\[4pt]
0 \leq b_{1},\dots,b_{\l} \leq 1: |b| = \ell\leq \l}}\,\dfrac{q^{\sum^{\l}_{m=1} (m-1)a_{m}}}{Z_{j}(\l)}\times
\tikz{0.7}{
\draw[lgray,line width=1pt,->] (-1,0) -- (3,0);
\draw[lgray,line width=1pt,->] (-1,1) -- (3,1);
\draw[lgray,line width=1pt,->] (-1,2) -- (3,2);
\draw[lgray,line width=4pt,->] (0,-1) -- (0,3);
\draw[lgray,line width=4pt,->] (1,-1) -- (1,3);
\draw[lgray,line width=4pt,->] (2,-1) -- (2,3);
\node[left] at (-1,0) {\tiny $a_{1}$};\node[right] at (3,0) {\tiny $b_{1}$};
\node[left] at (-1,1) {\tiny $\vdots$};\node[right] at (3,1) {\tiny $\vdots$};
\node[left] at (-1,2) {\tiny $a_{\l}$};\node[right] at (3,2) {\tiny $b_{\l}$};
\node[below] at (0,-1) {\tiny $i$};\node[above] at (0,3) {\tiny $k$};
\draw[->](-2.5,0) node[left] {\tiny $ x$}--(-2,0);
\draw[->](-2.5,2) node[left] {\tiny $xq^{\l-1}$}--(-2,2);
\draw[->](0,-2.5) node[below] {\tiny $ (y_{0},\m)$}--(0,-2);
\draw[->](2,-2.5) node[below] {\tiny $ (y_{N},\m)$}--(2,-2);
}
\end{multline}
\end{thm}

\subsection{Weights for the definition $\GG$}

We recall the steps required to obtain the weights $\mathbb{W}_{x;(u,v)}(a,b,c,d)$ defined in Equation~\ref{weights:spinL_uv}, which are used to define the polynomials $\mathfrak{G}^{\buv}_{\lambda/\mu}$. To derive $\mathbb{W}_{x;(u,v)}$, we start from the general weights in~\ref{theorem:weightsoftheYBE}, apply the substitutions $x \mapsto x q^{-\l}$, $q^{-\m} = uv$, and $y = v$, then divide the resulting weights by $u^{d}$, and finally take the limit $q^{-\l} \mapsto 0$.

\begin{prop}
The transfer matrices $\TT$ and $\T$ satisfy the following commutation relation:
\begin{equation}\label{eq:communtation_T_and_TT_appendix}
\TT(x)\,\T(y)\,=\,\dfrac{1}{(x y;q)_{\infty}}\,\T(y)\,\TT(x)  
\end{equation}
whenever
\[
|x|,|y|<\,1.
\]
\end{prop}

\begin{proof}

We begin by performing the following substitutions in the general $W_{1,\m}$ weights in~\ref{weights:spin1_row_weights}: we take  $q^{-\m}=uv$ and $y=v$ and divide the weights by $u$ whenever the right edge is occupied. The resulting weights after these substitutions can be written in terms of the weights $\mathrm{W}_{x;(-u,-v)}$ which are defined in~\ref{weights:spin1_uv_generic}. Then the weights $\mathbb{W}_{x;(u,v)}$ can be obtained by performing the fusion procedure to $\mathrm{W}_{x;(-u,-v)}$ and then substituting $x\mapsto x q^{-\l}$ and then taking $q^{-\l}\mapsto 0$.

Observe that the commutation relation between the transfer matrices~\ref{eq:communtation_t_and_tt} remains the same even when one uses weights $\mathrm{W}_{x;(-u,-v)}$ instead of $\mathrm{W}_{x;(u,v)}$. We begin with the commutation relation between the matrices $\t$ and $\tt$:
\begin{equation}
\tt(x_{1}) \cdots \tt(x_{I})\,\t(y_{1}) \cdots \t(y_{J})
= \prod_{i=1}^{I} \prod_{j=1}^{J} \frac{1 - q x_{i} y_{j}}{1 - x_{i} y_{j}} \,
\t(y_{1}) \cdots \t(y_{J})\, \tt(x_{1}) \cdots \tt(x_{I}).
\end{equation}

We first specialise the $y$-variables as $y_{j} = y q^{j-1}$ for $1 \leq j \leq J$.  
Due to telescoping cancellations, this simplifies to
\begin{equation}
\tt(x_{1}) \cdots \tt(x_{I}) \, \t(y) \t(yq) \cdots \t(yq^{J-1})
= \prod_{i=1}^{I} \frac{1 - q^{J} x_{i} y}{1 - x_{i} y} \,
\t(y) \cdots \t(yq^{J-1}) \, \tt(x_{1}) \cdots \tt(x_{I}).
\end{equation}

Next, we specialise the $x$-variables as $x_{i} = x q^{i-1}$ for $1 \leq i \leq I$, yielding
\begin{multline}
\tt(x)\tt(xq)\cdots \tt(xq^{I-1}) \, \t(y)\t(yq)\cdots \t(yq^{J-1}) \\
= \frac{(xyq^{J};q)_{I}}{(xy;q)_{I}} \,
\t(y) \cdots \t(yq^{J-1}) \,
\tt(x)\tt(xq)\cdots \tt(xq^{I-1}).
\end{multline}

We now invoke the identity
\[
\frac{(xyq^{J};q)_{I}}{(xy;q)_{I}}
= \frac{(xyq^{J};q)_{\infty}(xyq^{I};q)_{\infty}}
{(xy;q)_{\infty}(xyq^{I+J};q)_{\infty}},
\]
which transforms the relation into
\begin{multline}\label{eq:before_fusion}
\tt(x)\tt(xq)\cdots \tt(xq^{I-1}) \,\t(y)\t(yq)\cdots \t(yq^{J-1}) \\
= \frac{(xyq^{I};q)_{\infty}(xyq^{J};q)_{\infty}}
{(xyq^{I+J};q)_{\infty}(xy;q)_{\infty}} \,
\t(y)\cdots \t(yq^{J-1}) \,
\tt(x)\tt(xq)\cdots \tt(xq^{I-1}).
\end{multline}

To~\eqref{eq:before_fusion} we apply the following steps:
\begin{enumerate}
    \item Perform the fusion procedure,
    \item Substitute $x \mapsto x q^{-I}$ and $y \mapsto y q^{-J}$, and
    \item Take the limits $q^{-I} \to 0$ and $q^{-J} \to 0$.
\end{enumerate}

Carrying out these steps yields the desired commutation relation:
\[
\TT(x)\,\T(y) = \frac{1}{(xy;q)_{\infty}} \, \T(y)\,\TT(x).
\]
\end{proof}

\

\bibliographystyle{plain}
\bibliography{references}{}

\end{document}